\documentclass[letterpaper,12pt]{article}
\usepackage[T1]{fontenc}
\usepackage[totalwidth=480pt, totalheight=680pt]{geometry} 
\usepackage{bm} 
\usepackage{tikz} 
\usetikzlibrary{arrows,automata}
\usepackage{graphicx}
\usepackage{caption}

\usepackage[titletoc]{appendix}

\usepackage{mathrsfs}
\DeclareMathAlphabet{\mathbfscr}{OMS}{mdugm}{b}{n}

\usepackage{mathtools}
\usepackage{amsmath, amsfonts, amssymb, amsbsy, amsthm}
\usepackage{hyperref}
\usepackage{float} 
\usepackage{enumerate}
\usepackage[margin=2cm]{caption}
\usepackage[font=footnotesize,labelfont=bf]{caption}
\usepackage{xcolor}
\hypersetup{
    colorlinks,
    linkcolor={red!40!black},
    citecolor={blue!50!black},
    urlcolor={blue!80!black}
  }

\usepackage{anyfontsize}
\usepackage{libertine}
\usepackage[libertine]{newtxmath}

\usepackage[mathscr]{euscript}

\newtheorem{theorem}{Theorem}[section] 
\newtheorem{definition}{Definition}[section] 
\newtheorem{proposition}{Proposition}[section] 

\newtheorem{lemma}{Lemma}[section]
\newtheorem{assumption}{Assumption}

\theoremstyle{remark}

\newcommand{\abs}[1]{ \left\lvert#1\right\rvert} 
\newcommand{\norm}[1]{\left\lVert#1\right\rVert} 

\newcommand\Given[1][]{\:#1\big\vert\:}
\DeclarePairedDelimiter\floor{\lfloor}{\rfloor} 

\newcommand{\subf}[2]{%
  {\footnotesize\begin{tabular}[t]{@{}c@{}}
  #1\\#2
  \end{tabular}}%
}
\title{\textbf{Variance Estimation in
    \\ Adaptive Sequential Monte Carlo
}} 
\author{}
\date{} 

\begin{document}
\maketitle

{\bf Qiming Du}\\
{\it  LPSM, Sorbonne Universit\'e \& CERMICS, France}\\
\textsf{qiming.du@upmc.fr}
\bigskip

{\bf Arnaud Guyader\footnote{Corresponding author.}}\\
{\it LPSM, Sorbonne Universit\'e \& CERMICS, France }\\
\textsf{arnaud.guyader@upmc.fr}
\bigskip

\abstract{Sequential Monte Carlo (SMC) methods represent a classical set of techniques to simulate a sequence of probability measures through a simple selection/mutation mechanism. However, the associated selection functions and mutation kernels usually depend on tuning parameters that are of first importance for the efficiency of the algorithm. A standard way to address this problem is to apply Adaptive Sequential Monte Carlo (ASMC) methods, which consist in exploiting the information given by the history of the sample to tune the parameters. This article is concerned with variance estimation in such ASMC methods. Specifically, we focus on the case where the asymptotic variance coincides with the one of the 
 ``limiting'' Sequential Monte Carlo algorithm as defined by Beskos \textit{et al.} \cite{beskos2016ASMC}. 
  We prove that, under natural assumptions, the estimator introduced by Lee and Whiteley \cite{lee2018var} in the nonadaptive case (i.e., SMC) is also a consistent estimator of the asymptotic variance for
  ASMC methods. To do this, we introduce a new
   estimator that is expressed in terms of coalescent tree-based
  measures, and explain its connection with the previous one. Our estimator is constructed by
  tracing the genealogy of the associated Interacting Particle System.
  The tools we use connect the study of Particle Markov Chain Monte Carlo
 methods and the variance estimation problem in
  SMC methods. As such, they may give some new insights when dealing with complex
  genealogy-involved problems of Interacting Particle Systems in more general scenarios.}
  
  \bigskip

\noindent \emph{Index Terms} --- Sequential Monte Carlo, CLT, Variance estimation, Interacting particle systems, Feynman-Kac semigroups.\medskip

\noindent \emph{2010 Mathematics Subject Classification}: 47D08, 65C35, 60J80, 65C05.

\newpage
{
\hypersetup{linkcolor=black}
\tableofcontents
}
\section{Introduction}
Sequential Monte Carlo (SMC) methods are classical Monte Carlo
techniques widely used in
Bayesian inference, filtering, rare events simulations and many
other fields (see for example \cite{smc2001} and references therein). The principle is to approximate a sequence of
probability measures $(\eta_n)_{n\geq 0}$ by
simulating an Interacting Particle System (IPS) via an importance sampling and
resampling mechanism. The flow of measures is then approximated by the
empirical version $(\eta_n^N)_{n\geq 0}$. 
A lot of convergence results when the sample size $N$ goes to infinity can be found in
the literature (see for example \cite{del2004feynman-kac,MR3060209}). 

In practice, when applying these SMC methods, it is also very
important to have a control on the constructed estimators, such as confidence
intervals. For this, if one has a CLT type
theorem for the test function $f$ such as (see, e.g., \cite{del2004feynman-kac,MR2153989,MR2458190})
$$
\sqrt{N}
\left(\eta_n^N(f) - \eta_n(f)\right) 
\xrightarrow[N\rightarrow\infty]{\mathrm{d}}
\mathcal{N}(0,\sigma_{n}(f)^2),
$$
it suffices to provide a consistent estimator
$\sigma_n^N(f)$ of $\sigma_{n}(f)$ since Slutsky's lemma then ensures that
$$
\frac{
\sqrt{N}
\left(\eta_n^N(f) - \eta_n(f)\right) 
}{\sigma_n^N(f)}
\xrightarrow[N\rightarrow\infty]{\mathrm{d}}
\mathcal{N}(0,1).
$$
A natural way to achieve this aim is by
resimulating the IPS independently many times and by estimating
$\sigma_{n}(f)^2$ with the crude variance
estimator. However, since a single run of the algorithm may take a lot of time,
this is usually intractable. In
addition, as the estimator $\eta_n^N(f)$ of $\eta_n(f)$ provided by SMC is typically biased, it is also nontrivial to implement parallel computing
for a large number of IPS with $N$ relatively small.  
As a consequence, a variance estimator available with a single run of the simulation
is of crucial interest for applications. 

The first consistent estimator of this type
was proposed by Chan and Lai
\cite{chan2013}, by using the ancestral information encoded in the genealogy of
the associated IPS. Then, Lee and Whiteley \cite{lee2018var} proposed an unbiased variance
estimator for the unnormalized measures $\gamma_n^N$ 
and a term by term estimator, 
with insights on the genealogy of the IPS. Both estimators are
studied in the classical SMC framework, meaning in a nonadaptive setting where the weight functions and the Markov
proposal kernels are fixed a priori. 

In this article, we deal with \emph{adaptive} SMC methods. 
At each resampling step, the weight functions and/or Markov proposal kernels
depend upon the history of the simulated process. 
The idea is to approximate an ideal ``limiting'' SMC algorithm, which is usually
out of reach, by exploiting the induced information tracked by some summary statistics.
Such approaches are expected to be more efficient and more automated than the nonadaptive ones
since they require less user-specified tuning parameters.

Specifically, we are interested in the case where the adaptive SMC algorithm is
asymptotically identical to a ``limiting'' SMC algorithm. More precisely,
we expect the asymptotic variance of the adaptive SMC algorithm to be identical
to the ``ideal'' nonadaptive one. This kind of stability property is at the
core of the pair of articles \cite{beskos2016ASMC} and
\cite{cerou2016fluctuation}. The framework discussed in the present paper is
just a slightly generalized version of the one presented in Section 2 in
\cite{beskos2016ASMC} but still ensures the stability property of their Theorem
2.3. 

Another remark is about  Adaptive Multilevel Splitting (AMS), also known as Subset Simulation, see for example \cite{Au2001263,au:901,cg2, cerou2012,cerou2016fluctuation}.
This is a class of ASMC algorithms dedicated to rare event estimation and simulation. 
Despite the fact that our assumptions are not verified in the AMS framework, we
expect that the
variance estimator would also work in this context. Nonetheless, we believe that this case requires a specific analysis as well as different assumptions. To account for this, one can notice that the proofs in \cite{cerou2016fluctuation} and \cite{beskos2016ASMC} differ in many points, although the take-home message is the same. In a nutshell, the main difficulty in the AMS framework comes from the indicator functions in the potential functions as well as in the transitions kernels, leading to severe regularity issues when dealing with CLT type results and asymptotic variances. 

From a theoretical viewpoint, to prove the consistency of the variance
estimator proposed in \cite{lee2018var}, we were not able to adapt their
technical tools. This is due to the additional randomness brought by the weight
functions and Markov kernels in the adaptive case.
As a consequence, we propose to
develop new techniques in order to estimate the terms $\Gamma_n^b$ that appear in the
expansion of the variance given in \cite{cerou2011var}. The mains ideas are: first,
our term by term estimator is consistent and, second, the difference
between our estimator and the one of Lee and Whiteley goes to $0$ in probability
when the sample size $N$ goes to infinity. However, in practice, one uses
the estimator proposed by Lee and Whiteley, which is computationally very simple, while the one we introduce here may
be seen as a handy tool to prove the consistency of the former.


The construction of our estimators $\Gamma_{n,N}^b$ uses
the idea of many-body Feynman-Kac models, which were designed in \cite{del2016gibbs} to study
propagation of chaos properties of Conditional Particle Markov Chain Monte
Carlo methods \cite{Andrieu2010pmcmc}. Above the specific context of the
present article, these connections may give some insights on
how to deal with complex genealogy-involved problems in more general settings.

\subsection*{Notation}

Before proceeding, let us provide some notation that will be of constant use in the following.
\begin{itemize}
  \item
    For any Polish space $E$, we denote respectively by $\mathcal{M}(E)$,
    $\mathcal{M}_+(E)$ and $\mathcal{P}(E)$ the sets of signed
    finite measures, nonnegative finite measures,
    and probability measures on
    $E$ endowed with Borel $\sigma$-algebra $\mathcal{B}(E)$,  
  while $\mathcal{B}_b(E)$ denotes the collection of the
    bounded measurable functions from $(E,\mathcal{B}(E))$ to
    $(\mathbf{R},\mathcal{B}(\mathbf{R}))$ equipped with uniform
    norm $\norm{\cdot}_{\infty}$.
  \item
    For any $\mu\in \mathcal{M}(E)$ and any test function $f \in
    \mathcal{B}_b(E)$, we write  
    $$
   \mu(f):= \int_E f(x) \mu(dx).
    $$ 
    A finite nonnegative kernel $Q$ from $(E,\mathcal{B}(E))$ to
    $(F,\mathcal{B}(F))$ is a function
    $$
    Q: E\times \mathcal{B}(F)\mapsto \mathbf{R}_+
    $$
    such that, for all $x\in E$, $Q(x,\cdot)\in\mathcal{M}_+(F)$ and, for all $A \in
    \mathcal{B}(F)$, $Q(x,A)$ is a $\mathcal{B}(E)$-measurable
    function. 
    We say that $Q$ is a Markov transition kernel if, moreover, for all $x\in E$, $Q(x,\cdot)$ is a probability measure in
    $\mathcal{P}(F)$.
    For a signed measure $\mu \in \mathcal{M}(E)$
    and a test function $f\in \mathcal{B}_b(F)$, we denote
   respectively  by $\mu Q \in \mathcal{M}(E)$ and $Q(f) \in
    \mathcal{B}_b(E)$ the measure and function respectively defined by
    $$
    \mu Q (A) := \int_E \mu(dx) Q(x,A)
    \qquad \forall A \in \mathcal{B}(F),
    $$
    and
    $$
    Q (f)(x) := \int_F  Q(x,dy) f(y)
    \qquad \forall x \in E.
    $$
    Given two finite nonnegative kernels $Q_1$ and $Q_2$ respectively from
    $E_0$ to $E_1$ and $E_1$ to $E_2$,
    $Q_1Q_2$ is the nonnegative kernel from
    $E_0$ to $E_2$ defined by
    $$
    Q_1Q_2(x,A) := \int_{E_1} Q_1(x,dy) Q_2(y,A)
    \qquad \forall (x,A) \in E_0\times \mathcal{B}(E_2).
    $$
  \item
    For two functions $f,g \in \mathcal{B}(E)$, their tensor product is the function
    $$
    f \otimes g:E^2\ni (x,y)\mapsto
    f(x)g(y)\in \mathbf{R},
    $$
    and, in particular, we denote $f^{\otimes 2}:= f\otimes f$. 
    For two finite nonnegative kernels $Q$ and $H$ from $(E,\mathcal{B}(E))$ to
    $(F,\mathcal{B}(F))$,
    we denote
    $$
    (Q \otimes H) 
    \left(
      (x,y),(A,B)
    \right)
    :=
    Q(x,A) \times H(y,B)
    $$
    for all
    $(x,y)\in E\times E$
    and all
    $(A,B)\in \mathcal{B}(F)\otimes \mathcal{B}(F)$.
    Accordingly, we write $Q^{\otimes 2} := Q\otimes Q$.
  \item In order to define the coalescent tree-based measures of size 2, we introduce
    the transition operators $C_0$ and $C_1$ as
    $$
    C_0((x,y),d(x',y'))
    := 
    \delta_{(x,y)} d(x',y'),
    $$
    and
    $$
    C_1((x,y),d(x',y'))
    := 
    \delta_{(x,x)} d(x',y').
    $$
    In other words, for any measurable
    function $H:E\times E\mapsto \mathbf{R}$, we have
    $$
    C_0(H)(x,y) = H(x,y)
    \qquad
    \text{and}
    \qquad
    C_1(H)(x,y) = H(x,x).
    $$
  \item 
    For all $\bm{x} = (x^1,\dots,x^N) \in E^{N}$, we define
    the empirical measure associated to $\bm{x}$ by
    $$
    m(\bm{x}) := \frac{1}{N}\sum_{i = 1}^N
    \delta_{x^{i}}
    \in \mathcal{P}(E).
    $$
    We also denote
    $$
    m^{\otimes 2}(\bm{x}) 
    := \frac{1}{N^2}\sum_{i,j}
    \delta_{(x^i,x^j)}
    \in \mathcal{P}(E^2),
    $$
    and
    $$
    m^{\odot 2}(\bm{x}) 
    := \frac{1}{N(N-1)}\sum_{i\neq j}
    \delta_{(x^i,x^j)}
    \in \mathcal{P}(E^2).
    $$
    A straightforward computation shows that
    \begin{equation}\label{lasjcn}
    m^{\otimes 2}(\bm{x})
    =
    \frac{N-1}{N}m^{\odot 2}(\bm{x})
    C_0
    +
    \frac{1}{N}m^{\odot 2}(\bm{x})
    C_1.
    \end{equation}
    With a slight abuse of notation, considering $[N] := \{1,2,\dots,N\}$, we write
    $$
    m([N]) := \frac{1}{N}\sum_{i=1}^N\delta_i
    \quad
    \text{and}
    \quad
    m^{\otimes 2}([N]) := m([N])\otimes m([N]). 
    $$
\end{itemize}

\section{Adaptive Sequential Monte Carlo}

This section presents the formal definition and the regularity assumptions
of the ASMC framework studied in this article. The motivation is mainly from
ASMC via summary statistics introduced in Section 2 of
\cite{beskos2016ASMC}. We refer the reader to the latter for details on motivating examples such as filtering or sequential Bayesian parameter inference.

\subsection{Framework}
The notations that are adopted are essentially those in the pair of books
\cite{del2004feynman-kac,MR3060209}. 
Let $(E_n,\mathcal{B}(E_n))_{n\geq 0}$ be
a sequence of Polish
spaces. For each level $n\geq 1$, we
consider a family of potential functions $G_{n-1,z}:E_{n-1}\mapsto \mathbf{R}_+$
and Markov kernels $M_{n,z}:(E_{n-1},\mathcal{B}(E_n))\mapsto
[0,1]$  parametrized by $z\in \mathbf{R}^d$. Accordingly, we define
the family of nonnegative Feynman-Kac kernels $Q_{n,z}$ by
$$
Q_{n,z}(x,A):=G_{n-1,z}(x)M_{n,z}(x,A).
$$
We suppose that there exists a sequence of reference parameters $(z_n^*)_{n\geq
0}$ and, for each $n\geq 1$, we denote
$$
G_{n-1}:=G_{n-1,z_{n-1}^*},
\qquad 
M_n:=M_{n,z_{n-1}^*}
\qquad 
\text{and}
\qquad 
Q_n:=Q_{n,z_{n-1}^*}.
$$
Starting with a known probability measure $\gamma_0:=\eta_0\in \mathcal{P}(E_0)$, 
we define the unnormalized Feynman-Kac measures $\gamma_n$ by 
$$
\gamma_n:= \gamma_0 Q_1\cdots Q_n,
$$
along with the normalized measures 
$$
\eta_n:=
\frac{1}{\gamma_n(1)}\gamma_n.
$$
Assumption \ref{A1} below ensures that, for all $n\geq 0$, $G_n$ is strictly positive so that 
$$
\gamma_n(1) = \prod_{p=0}^{n-1}\eta_p(G_p) > 0.
$$
Another formulation of the connection between normalized and unnormalized measures is thus given by
\begin{equation}\label{cvnaeohncv}
\gamma_n(f_n)=\eta_n(f_n)\prod_{p=0}^{n-1}\eta_p(G_p).
\end{equation}
For $p<n$, we define the Feynman-Kac semigroup 
$$
Q_{p,n}:= Q_{p+1} \cdots Q_n,
$$
and $Q_{n,n}(x,A) := \delta_x(A)$. 
In this context, ASMC algorithms
aim at approximating the sequences of measures $(\gamma_n)_{n\geq 0}$ and
$(\eta_n)_{n\geq 0}$ by exploiting some summary statistics
$$
\zeta_n:E_n\mapsto \mathbf{R}^d
$$
such that, for all $n\geq 0$, we have
$$
\eta_n(\zeta_n) = z_n^*.
$$

\subsection{ASMC algorithm}\label{sec-old:IPS}
In practice, ASMC and SMC algorithms share the same selection/mutation mechanisms.
However, since in most situations of interest the parameters $(z_n^*)_{n\geq 0}$ are not analytically tractable, the potential functions $(G_n)_{n\geq 0}$ and transition kernels
$(M_n)_{n\geq 1}$ are estimated on the fly through the design
of an adaptive algorithm.    

Let $N\in \mathbf{N}^*$ be the number of particles (or samples).
The Interacting Particle System (IPS) associated to the ASMC algorithm is a Markov chain $(\mathbf{X_{n}})_{n\geq
0}$ taking values in $(E_n^{N},\mathcal{B}(E_n)^{\otimes N})_{n\geq 0}$
with genealogy $(\mathbf{A_{n}})_{n\geq 0}$ tracking the indice of the parent of each particle
at each level. Specifically, $A_{p-1}^i =j$ means that the parent of the particle
$X_{p}^i$ at layer $p$ is $X_{p-1}^j$ at layer $p-1$.
The estimation of the normalized measure $\eta_n$ is given by the empirical measure
$$
\eta_n^N
:=
\frac{1}{N}\sum_{i=1}^N
\delta_{X_n^i}.
$$
At each level $n\geq 0$, the estimated parameters are defined by
$Z_n^N:=\eta_n^N(\zeta_n)$.
In order to lighten the notation, we denote
$$
G_{n-1,N} := G_{n-1,Z_{n-1}^N},
\qquad
M_{n,N} := M_{n,Z_{n-1}^N},
\qquad
\text{and}
\qquad
Q_{n,N} := Q_{n,Z_{n-1}^N}.
$$
Then, considering (\ref{cvnaeohncv}), the unnormalized Feynman-Kac measures are estimated by
$$
\gamma_n^N(f_n)
:=\eta_n^N(f_n)
\prod_{p=0}^{n-1} \eta_p^N(G_{p,N}).
$$
In the following sections, we use the convention
$$
\eta_{-1}^N
=
\gamma_{-1}^N
:=\eta_0.
$$
Let us give the formal definition of the IPS associated with the ASMC
algorithm:
\begin{enumerate}[(i)]
  \item
Initial distribution: 

At step $0$, let 
$\mathbf{X_0} \sim \eta_0^{\otimes N}$.

\item
Transition kernels: 

For all $p\geq 0$, set $Z_p^N =
\eta_p^N(\zeta_p)$. The transition $X_{p}^i\rightsquigarrow X_{p+1}^i$ is decomposed into two steps:
\begin{itemize}
  \item Selection: given $\mathbf{X_{p}}=\bm{x_{p}}$, we make an
    independent multinomial selection of the parent of each particle by
    \begin{equation}\label{eq:selection}
    S_{p,N}(\bm{x_{p}},da_{p}^i)
    =
    \sum_{k=1}^{N}
    \frac{G_{p,N}(x_{p}^k)}{\sum_{j=1}^N
    G_{p,N}(x_{p}^j)}
    \delta_k(da_{p}^i).
    \end{equation}
    Thus, the genealogy of level $p$ to level $p+1$ is tracked by
    $$
    \mathbf{A_{p}}\sim 
    \bigotimes_{i=1}^N
    S_{p,N}(\mathbf{X_{p}},\cdot)
    $$
  \item Mutation: given the parent indices $\mathbf{A_{p}}=\bm{a_{p}}$, each
    particle at level $p$ evolves independently according to the transition
    kernel $M_{p+1,N}$, meaning that for $i\in[N]$,
    $$
    X_{p+1}^i \sim M_{p+1,N}(X_{p}^{a_{p}^i},\cdot). 
    $$
    Said differently, given $\mathbf{X_{p}}$ and $\mathbf{A_{p}}$, we have
 $$
    \mathbf{X_{p+1}}\sim 
    \bigotimes_{i=1}^N
    M_{p+1,N}(X_{p}^{A_{p}^i},\cdot). 
    $$   
\end{itemize}
%

\end{enumerate}

\subsection{Assumptions}

Our assumptions are introduced in
a similar way as in \cite{beskos2016ASMC}, but just slightly weaker. The reason why we can relax their assumptions is because we are only interested in the specific situation where the asymptotic variance of the ASMC estimator  is identical to the ``limiting'' SMC algorithm which uses ideal potential functions and proposal kernels, namely $G_p=G_{p,z_{p-1}^*}$ and $M_p=M_{p,z_{p-1}^*}$. Considering stability properties, Section 2.7 in \cite{beskos2016ASMC} explains why this case is particularly interesting in practice.
In the following sections, we use $\mathcal{A}$ as a short-hand for \emph{Assumption}.
 
\begin{assumption}\label{A1}
  For each $n\geq 0$, we assume that $G_{n,z}$ is strictly positive and bounded uniformly
  over $z\in \mathbf{R}^d$, i.e.,
  $$
 \norm{G_{n,\cdot}}_{\infty}:= \sup_{(x,z)\in E_n\times\mathbf{R}^d} G_{n,z}(x)<+\infty.
  $$
\end{assumption}

Notice that, under $\mathcal{A}$\ref{A1}, Equation
(\ref{eq:selection}) above is always well-defined for the denominator is always
strictly positive. 
  In the case where $G_{p,z}$ is only assumed to be nonnegative, as in the AMS framework, one
  may consider the stopping time $\tau_N$ defined by 
  $$
  \tau_N := \inf\left\{p\in \mathbf{N} : \eta_p^N(G_{p,N}) =0\right\}.  
  $$ 
 We believe that similar techniques can be applied to obtain results of the
 same taste as in the present paper, but at the cost of considerable technical
 complications which are out of the scope of this article. Let us mention that the strict positivity and boundedness of the potential functions is also required in 
 \cite{beskos2016ASMC} (see page 1116 and Assumption 1 page 1118).
In our second assumption, ``$\langle \cdot,\cdot \rangle$'' stands for the Euclidean
scalar product in $\mathbf{R}^d$ and $|\cdot|$ for the associated  norm.

\begin{assumption}\label{A2}
  For any test function $f_{n+1}\in\mathcal{B}_b(E_{n+1})$,
  there exists a measurable function 
  $h_{n}:(E_{n}\times \mathbf{R}^d,\mathcal{B}(E_{n})\otimes\mathcal{B}(\mathbf{R}^d)) \to (\mathbf{R}^d,\mathcal{B}(\mathbf{R}^d))$ such that, for all $(x,z_{n})\in E_{n}\times \mathbf{R}^d$,
  $$
  Q_{n+1,z_{n}}(f_{n+1})(x) - Q_{n+1}(f_{n+1})(x) = \left\langle h_{n}(x,z_{n}),z_{n} -z_{n}^*\right\rangle.
  $$
 The function $h_{n}$ is assumed to satisfy the following properties:
  \begin{itemize}
    \item
      The Euclidean norm $\abs{h_{n}}$ is bounded over $E_{n}\times \mathbf{R}^d$ by
      $\norm{h_{n}}_{\infty}$.
    \item
      The application $z\mapsto h_{n}(x,z)$ is continuous at $z_{n}^*$
      uniformly over $x\in E_{n}$. More precisely, for any $\epsilon>0$,
      there exists $\delta>0$, such that $\abs{z_{n}-z_{n}^*}<\delta$
      implies
      $$
      \sup_{x\in E_{n}}\abs{h_{n}(x,z_{n}) -h_{n}(x,z_{n}^*)} 
      < \epsilon.
      $$
    \item
      $h_{n}$ satisfies the equality
      $
      \eta_{n}\left(h_{n}(\cdot,z_{n}^*)\right) =0.
      $
  \end{itemize}
Moreover, the summary statistics $\zeta_n = (\zeta_n^1,\dots,\zeta_n^d)$
satisfies $z_n^*=\eta_n(\zeta_n)$ and is such that, for all $k\in[d]$,
$\zeta_n^k$ belongs to $\mathcal{B}_b(E_n)$. 
\end{assumption}

$\mathcal{A}$\ref{A2}
guarantees some regularity properties of the transition kernels $Q_{n,z}$
with respect to the parameter $z$ and
is just a slight generalization of the framework studied in Section 2 of
\cite{beskos2016ASMC}. Indeed, our function $h_{n}$ coincides with the function $\omega$ defined in (2.17) of \cite{beskos2016ASMC}, that is
$$h_{n}(x,z_{n})=\int_0^1\left.\partial_zQ_{n+1,z}(f_{n+1})(x)\right|_{z=z_{n}^*+\lambda(z_{n} -z_{n}^*)}d\lambda.$$
As such, the first two conditions on $h_{n}$ are satisfied as soon as Assumption 2 in \cite{beskos2016ASMC} is verified. In this respect, our third condition on $h_{n}$ corresponds to their condition (2.19) in Theorem 2.3, which is precisely the ``limiting'' case mentioned above. Finally, the hypothesis that the summary statistics are bounded is also required in their Assumption 1, while the relation $z_n^*=\eta_n(\zeta_n)$ corresponds in their notation to $\bar\xi_n=\eta_{n-1}(\xi_n)$.

We also want to mention that the second point is
equivalent to
\begin{multline*}
\forall \epsilon>0,\;
\exists g_n\in \mathcal{B}_b(E_{n}),\;
\exists\delta>0,\;s.t.\\
\abs{z_{n} -z_{n}^*} < \delta
\implies
\forall x\in E_{n},\ \abs{h_{n}(x,z_{n}) -h_{n}(x,z_{n}^*)} 
< g_n(x)\epsilon\leq\|g_n\|_\infty\epsilon.
\end{multline*}
We expect that, in this alternative formulation, the functions $g_n$ and $h_{n}$ can be relaxed to some unbounded functions,
belonging for example to $\mathbb{L}^{\mathrm{2}}(\eta_{n})$, along with stronger conditions on the
test function $f_{n+1}$. 
We believe that this is one of the
main differences between the ASMC framework studied in \cite{beskos2016ASMC} and
the AMS framework studied in \cite{cerou2016fluctuation}.

In general, it is not easy to verify the existence of such $h_{n}$.
However, we have, at least, a direction to explore in the case where
$Q_{n,z}(f)$ is not globally differentiable with respect to $z$. We also remark that we do not study
the consistency of $\gamma_n^N(f)$ and $\eta_n^N(f)$ with weaker
assumptions, as we are only interested in the 
CLT type result of Theorem \ref{thm-old:clt} below and, more specifically, in the estimation of the asymptotic variance.
Nevertheless, let us briefly mention that to establish the consistency of $\gamma_n^N$ and $\eta_n^N$, one just needs 
  $$
  \gamma_{n-1}^N Q_{n,N} (f_n)
  -
  \gamma_{n-1}^N Q_{n} (f_n)
  =o_{\mathbf{p}}(1)
  $$
   for any test function $f_n\in \mathcal{B}_b(E_n)$. This does not
  require such a strong assumption as $\mathcal{A}$\ref{A2}. However, for
  CLT type results with the ``stable'' asymptotic variance, it is necessary that 
  $$
  \gamma_{n-1}^N Q_{n,N} (f_n)
  -
  \gamma_{n-1}^N Q_{n} (f_n)
  =o_{\mathbf{p}}\left(\frac{1}{\sqrt{N}}\right).
  $$
  A stronger regularity assumption like $\mathcal{A}$\ref{A2} over the  parametrization is therefore
  required.


\subsection{Central limit theorems}

As explained before, the present article only deals with the case where the asymptotic variance
is identical to the ``limiting'' one, which is only a special case of the Central Limit
Theorem 2.2 given in
\cite{beskos2016ASMC} under slightly weaker assumptions. 
This is why, in Section \ref{sec:pf-clt}, we propose a different strategy for
the proof.

\begin{theorem}\label{thm-old:clt}
  Assume $\mathcal{A}$\ref{A1}-$\mathcal{A}$\ref{A2}.
  For any test function $f\in
  \mathcal{B}_n(E_n)$, we have

      $$
      \sqrt{N}
      \left(
      \gamma_n^N(f) - \gamma_n(f) 
      \right)
      \xrightarrow[N\rightarrow \infty]{\mathrm{d}}
      \mathcal{N}\left(
        0,\sigma_{\gamma_n}^2(f)
      \right),
      $$
and
      $$
      \sqrt{N}
      \left(
      \eta_n^N(f) - \eta_n(f) 
      \right)
      \xrightarrow[N\rightarrow \infty]{\mathrm{d}}
      \mathcal{N}\left(
        0,\sigma_{\eta_n}^2(f-\eta_n(f))
        \right),
      $$
  where
  $$
  \sigma_{\gamma_n}^2(f)
  :=
  \sum_{p=0}^n
  \left(
    \gamma_p(1)\gamma_p(Q_{p,n}(f)^2)
    -
    \gamma_n(f)^2
  \right)
  \qquad
  \text{and}
  \qquad
  \sigma_{\eta_n}^2(f) := \sigma_{\gamma_n}^2(f)\slash
  \gamma_n(1)^2.
  $$
\end{theorem}

One can notice that the CLT for $\eta_n^N$ is just a consequence of the CLT for $\gamma_n^N$, Slutsky's Lemma, and the decomposition
  $$
   \sqrt{N}\left(\eta_n^N(f) - \eta_n(f)\right)
  =
  \frac{1}{\gamma_n^N(1)}
  \sqrt{N}
  \left(\gamma_n^N(
    f-\eta_n(f))-\gamma_n(
    f-\eta_n(f))
  \right).
  $$

The main goal of this paper is to estimate the asymptotic variances $\sigma_{\gamma_n}^2(f)$
and $\sigma_{\eta_n}^2(f-\eta_n(f))$ by a single simulation of the particle
system, exactly as is done by Lee and Whiteley in \cite{lee2018var} in a  nonadaptive context. 

\section{Variance estimations}

In this section, we recall the coalescent tree-based
expansion of the variance firstly introduced in \cite{cerou2011var} from which we deduce a new variance estimator.
We also recall the variance estimator proposed by Lee and
Whiteley in \cite{lee2018var} and explain the connection between both estimators.

\subsection{Coalescent tree-based variance expansion}\label{ajsncoac}

We call $b:=(b_0,\dots,b_n)\in \{0,1\}^{n+1}$ a coalescence indicator where $b_p=1$ indicates that there is a coalescence at level $p$.

\begin{definition}\label{aecac}
  We associate with any coalescence
  indicator $b\in\{0,1\}^{n+1}$ the nonnegative measures
  $\Gamma_n^b$ and $\bar{\Gamma}_n^b \in\mathcal{M}_{+}(E_n^2)$ defined for any $F \in \mathcal{B}_b(E_n^2)$ by 
  $$
  \Gamma_n^b(F) := 
  \eta_0^{\otimes 2}C_{b_0}
  Q_1^{\otimes 2}C_{b_1}\cdots
  Q_n^{\otimes 2}C_{b_n}(F),
  $$
  and
  $$
  \bar{\Gamma}_n^b(F) :=
  \frac{1}{\gamma_n(1)^2} 
  \Gamma_n^b(F).
  $$
  When there is only one coalescence at, say, level $p$, we write
  $\Gamma_n^{(p)}(F)$ and $ \bar{\Gamma}_n^{(p)}(F)$ instead of $\Gamma_n^b(F)$ and $ \bar{\Gamma}_n^b(F)$ (see Figure \ref{sjcnlonc}).
  When there is no coalescence at all, that is $b=(0,\dots,0)$, we have   
  $$
  \Gamma_n^{(\varnothing)} (F) =
  \gamma_n^{\otimes 2} (F)\qquad \text{and}\qquad
  \bar{\Gamma}_n^{(\varnothing)} (F) = \eta_n^{\otimes
  2} (F).
  $$
\end{definition}

\begin{figure}[H]
\centering
\begin{tikzpicture}[scale = 0.5, every node/.style={transform shape}]
      \path [use as bounding box] (-7.5*1.3,-1.4*1.3) rectangle
        (7.5*1.3,1.4*1.3);

    \tikzset{norm/.style={circle, black, draw=black,densely dotted,
    fill=black, minimum size=1mm}}

    \tikzset{noire/.style={circle, black, draw=black,
    fill=black, minimum size=5mm}}

    \node [style=noire] (x0)    at (-7*1.3, -1*1.3)   {};
    \node               (xdots1)at (-5*1.3, -1*1.3)   {$\dots$};
    \node [style=noire] (xp-1)  at (-3*1.3, -1*1.3)   {};
    \node [style=noire] (xp)    at (-1*1.3, -1*1.3)   {};
    \node [style=noire] (xp+1)  at (1*1.3,  -1*1.3)    {};
    \node [style=noire] (xp+2)  at (3*1.3,  -1*1.3)    {};
    \node  (xdots2)             at (5*1.3,  -1*1.3)    {$\dots$};
    \node [style=noire] (xn)    at (7*1.3,  -1*1.3)     {};

    \node [style=noire] (y0)    at (-7*1.3, 1*1.3)    {};
    \node               (ydots1)at (-5*1.3, 1*1.3)    {$\dots$};
    \node [style=noire] (yp-1)  at (-3*1.3, 1*1.3)    {};
    \node [style=noire] (yp)    at (-1*1.3, 1*1.3)    {};
    \node [style=noire] (yp+1)  at (1*1.3,  1*1.3)     {};
    \node [style=noire] (yp+2)  at (3*1.3,  1*1.3)     {};
    \node               (ydots2)at (5*1.3,  1*1.3)     {$\dots$};
    \node [style=noire] (yn)    at (7*1.3,  1*1.3)     {};

    \draw [arrows=-latex] (x0)        to (xdots1);
    \draw [arrows=-latex] (xdots1)    to (xp-1);
    \draw [arrows=-latex] (xp-1)      to (xp);
    \draw [arrows=-latex] (xp)        to (xp+1);
    \draw [arrows=-latex] (xp+1)      to (xp+2);
    \draw [arrows=-latex] (xp+2)      to (xdots2);
    \draw [arrows=-latex] (xdots2)    to (xn);

    \draw [arrows=-latex] (y0)        to (ydots1);
    \draw [arrows=-latex] (ydots1)    to (yp-1);
    \draw [arrows=-latex] (yp-1)      to (yp);
    \draw [arrows=-latex] (xp)        to (yp+1);
    \draw [arrows=-latex] (yp+1)      to (yp+2);
    \draw [arrows=-latex] (yp+2)      to (ydots2);
    \draw [arrows=-latex] (ydots2)    to (yn);

    \node at (-6*1.3,-1) {$Q_{1}$};
    \node at (-4*1.3,-1) {$Q_{p-1}$};
    \node at (-2*1.3,-1) {$Q_{p}$};
    \node at (-0*1.3,-1) {$Q_{p+1}$};
    \node at (2*1.3,-1) {$Q_{p+2}$};
    \node at (4*1.3,-1) {$Q_{p+3}$};
    \node at (6*1.3,-1) {$Q_{n}$};
    
    \node at (-6*1.3, 1.6) {$Q_{1}$};
    \node at (-4*1.3, 1.6) {$Q_{p-1}$};
    \node at (-2*1.3, 1.6) {$Q_{p}$};
    \node at (-0*1.3-0.5, 0.2) {$Q_{p+1}$};
    \node at (2*1.3,  1.6) {$Q_{p+2}$};
    \node at (4*1.3,  1.6) {$Q_{p+3}$};
    \node at (6*1.3,  1.6) {$Q_{n}$};
    \node at (-7.5*1.3, -1*1.3)   {$\eta_0$};
    \node at (-7.5*1.3, 1*1.3)   {$\eta_0$};
\end{tikzpicture}
\caption{A representation of the coalescent tree-based measure
$\Gamma_n^{(p)}$.}
\label{sjcnlonc}
\end{figure}
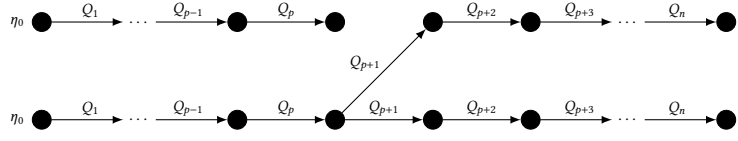

\medskip

\noindent
It is easy to verify from the definition that 
$$
\Gamma_n^{(p)} (f^{\otimes 2}) = \gamma_p(1)\ \gamma_p\left( Q_{p,n}(f)^2\right).
$$
As noticed in  \cite{cerou2011var}, the latter yields alternative representations for the asymptotic variances of Theorem \ref{thm-old:clt}, namely
\begin{equation}\label{eq:var-coal-gamma}
\sigma_{\gamma_n}^2(f) = \sum_{p=0}^{n} \left(
\Gamma_n^{(p)} (f^{\otimes 2})
- \Gamma_n^{(\varnothing)}(f^{\otimes 2})\right),
\end{equation}
and
\begin{equation}\label{eq:var-coal-eta}
\sigma_{\eta_n}^2(f) = \sum_{p=0}^{n} \left(
  \bar{\Gamma}_n^{(p)} (f^{\otimes 2})
- \bar{\Gamma}_n^{(\varnothing)}(f^{\otimes 2})\right).
\end{equation}
As a consequence, if for any coalescence indicator $b:=(b_0,\dots,b_n)\in \{0,1\}^{n+1}$, we can construct a consistent estimator $\bar{\Gamma}_{n,N}^b$ of $\bar{\Gamma}_{n}^b$, then we automatically deduce consistent estimators for the asymptotic variances of Theorem \ref{thm-old:clt}. This is the idea behind our next definition. 

In this definition, $\tilde{a}_p^{[2]}=(\tilde{a}_p^1,\tilde{a}_p^2)$ and $\ell_p^{[2]}=(\ell_p^1,\ell_p^2)$ denote two couples of indices between 1 and $N$, while an $(n+1)-$sequence of couples of indices such that $\ell_p^1\neq\ell_p^2$ for all $0\leq p\leq n$ is written
$$\ell_{0:n}^{[2]}=(\ell_0^{[2]},\cdots,\ell_n^{[2]})\in \left((N)^2\right)^{\times(n+1)}.$$ 
Additionally, we use the notation $X_n^{\ell_n^{[2]}}$ as a short-hand for $(X_n^{\ell_n^{1}},X_n^{\ell_n^{2}})$.

\begin{definition}
  \label{def:coal-estimator}
  For any test function $F\in\mathcal{B}_b(E_n^2)$
  and any coalescence indicator
  $b$, we define the
  estimator $\bar{\Gamma}_{n,N}^b$ 
  of the measure $\bar{\Gamma}_n^b$ by
  \begin{equation*}
  \bar{\Gamma}_{n,N}^b(F):=
  \frac{N^{n-1}}{(N-1)^{n+1}}
  \sum_{\ell_{0:n}^{[2]}\in \left((N)^2\right)^{\times(n+1)}}
 \left\{ \prod_{p=0}^{n-1}
  \lambda_p^b(A_p^{\ell_{p+1}^{[2]}},\ell_p^{[2]})
  \right\}
  C_{b_n}(F)(X_n^{\ell_n^{[2]}}),
  \end{equation*}
  where
  $\lambda_p^b(\tilde{a}_p^{[2]},\ell_p^{[2]})\in\{0,1\}$
  is an indicator function defined by
  $$
  \lambda_p^b(\tilde{a}_p^{[2]},\ell_p^{[2]})
  :=
  \mathbf{1}_{\{b_p = 0\}} \mathbf{1}_{\{\tilde{a}_p^1  =
  \ell_p^1\neq\tilde{a}_p^2=\ell_p^2\}}
  +
  \mathbf{1}_{\{b_p = 1\}} \mathbf{1}_{\{\tilde{a}_p^1 =
  \ell_p^1=\tilde{a}_p^2
  \neq\ell_p^2\}}.
  $$
 The estimator of $\Gamma_n^b$ is defined by
  \begin{equation*}
  \Gamma_{n,N}^b(F) =
  \gamma_n^N(1)^2\ \bar{\Gamma}_{n,N}^b(F). 
  \end{equation*}
\end{definition}

Since $\ell_p^1\neq\ell_p^2$, we also have 
$$ \lambda_p^b(\tilde{a}_p^{[2]},\ell_p^{[2]})= \mathbf{1}_{\{b_p = 0\}} \mathbf{1}_{\{\tilde{a}_p^1  =
  \ell_p^1,\tilde{a}_p^2=\ell_p^2\}}
  +
  \mathbf{1}_{\{b_p = 1\}} \mathbf{1}_{\{\tilde{a}_p^1 =
  \ell_p^1=\tilde{a}_p^2\}}.$$
Notice that, for $n=0$, we get
\begin{equation}\label{sjcn}
  \bar{\Gamma}_{0,N}^b(F):=
  \frac{1}{N(N-1)}
  \sum_{\ell_{0}^{[2]}\in (N)^2}
  C_{b_0}(F)(X_0^{\ell_0^{[2]}})
  =
  \frac{1}{N(N-1)}
  \sum_{i\neq j}
  C_{b_0}(F)(X_0^i,X_0^j).
 \end{equation}
  We also adopt the convention
  $$
  \bar{\Gamma}_{-1,N}^b(F)=
  \Gamma_{-1,N}^b(F):=
  \eta_0^{\otimes 2}C_{b_0}(F).
  $$

\paragraph{A toy example}
As the definition of the estimator $\Gamma_{n,N}^b$
is not completely straightforward, we illustrate the idea on a simple example.
For this, we consider the IPS of Figure \ref{aoecochz}.
\begin{figure}[htb]
  \centering
      \begin{tikzpicture}[scale = 0.7, every
        node/.style={transform shape}]
      \path [use as bounding box] (-6,-3) rectangle (8,3);
    \tikzset{norm/.style={circle, black, draw=black,
    fill=white, minimum size=7mm,inner sep=0pt}}
    \node [style=norm] (01) at (-5, 2) {$X_0^1$};
    \node [style=norm] (02) at (-5, 1) {$X_0^2$};
    \node [style=norm] (03) at (-5, -0) {$X_0^3$};
    \node [style=norm] (04) at (-5, -1) {$X_0^4$};
    \node [style=norm] (05) at (-5, -2) {$X_0^5$};
    \node [] (step0) at (-5, -3) {step 0};

    \node [style=norm] (11) at (-3, 2) {$X_1^1$};
    \node [style=norm] (12) at (-3, 1) {$X_1^2$};
    \node [style=norm] (13) at (-3, -0) {$X_1^3$};
    \node [style=norm] (14) at (-3, -1) {$X_1^4$};
    \node [style=norm] (15) at (-3, -2) {$X_1^5$};
    \node [] (step1) at (-3, -3) {step 1};

    \draw [arrows=-latex] (01) to (11);
    \draw [arrows=-latex] (01) to (12);
    \draw [arrows=-latex] (03) to (13);
    \draw [arrows=-latex] (05) to (14);
    \draw [arrows=-latex] (03) to (15);

    \node [style=norm] (21) at (-1, 2) {$X_2^1$};
    \node [style=norm] (22) at (-1, 1) {$X_2^2$};
    \node [style=norm] (23) at (-1, -0) {$X_2^3$};
    \node [style=norm] (24) at (-1, -1) {$X_2^4$};
    \node [style=norm] (25) at (-1, -2) {$X_2^5$};
    \node [] (step2) at (-1, -3) {step 2};

    \draw [arrows=-latex] (14) to (21);
    \draw [arrows=-latex] (14) to (22);
    \draw [arrows=-latex] (12) to (23);
    \draw [arrows=-latex] (14) to (24);
    \draw [arrows=-latex] (13) to (25);

    \node [style=norm] (31) at (1, 2) {$X_3^1$};
    \node [style=norm] (32) at (1, 1) {$X_3^2$};
    \node [style=norm] (33) at (1, -0) {$X_3^3$};
    \node [style=norm] (34) at (1, -1) {$X_3^4$};
    \node [style=norm] (35) at (1, -2) {$X_3^5$};
    \node [] (step3) at (1, -3) {step 3};

    \draw [arrows=-latex] (25) to (31);
    \draw [arrows=-latex] (22) to (32);
    \draw [arrows=-latex] (21) to (33);
    \draw [arrows=-latex] (25) to (34);
    \draw [arrows=-latex] (22) to (35);

    \node [style=norm] (41) at (3, 2) {$X_4^1$};
    \node [style=norm] (42) at (3, 1) {$X_4^2$};
    \node [style=norm] (43) at (3, -0) {$X_4^3$};
    \node [style=norm] (44) at (3, -1) {$X_4^4$};
    \node [style=norm] (45) at (3, -2) {$X_4^5$};
    \node [] (step4) at (3, -3) {step 4};

    \draw [arrows=-latex] (31) to (41);
    \draw [arrows=-latex] (32) to (42);
    \draw [arrows=-latex] (34) to (43);
    \draw [arrows=-latex] (32) to (44);
    \draw [arrows=-latex] (32) to (45);

    \node [style=norm] (51) at (5, 2) {$X_5^1$};
    \node [style=norm] (52) at (5, 1) {$X_5^2$};
    \node [style=norm] (53) at (5, -0) {$X_5^3$};
    \node [style=norm] (54) at (5, -1) {$X_5^4$};
    \node [style=norm] (55) at (5, -2) {$X_5^5$};
    \node [] (step5) at (5, -3) {step 5};

    \draw [arrows=-latex] (42) to (51);
    \draw [arrows=-latex] (43) to (52);
    \draw [arrows=-latex] (45) to (53);
    \draw [arrows=-latex] (43) to (54);
    \draw [arrows=-latex] (43) to (55);

    \node [style=norm] (61) at (7, 2) {$X_6^1$};
    \node [style=norm] (62) at (7, 1) {$X_6^2$};
    \node [style=norm] (63) at (7, -0) {$X_6^3$};
    \node [style=norm] (64) at (7, -1) {$X_6^4$};
    \node [style=norm] (65) at (7, -2) {$X_6^5$};
    \node [] (step6) at (7, -3) {step 6};

    \draw [arrows=-latex] (52) to (61);
    \draw [arrows=-latex] (51) to (62);
    \draw [arrows=-latex] (52) to (63);
    \draw [arrows=-latex] (53) to (64);
    \draw [arrows=-latex] (55) to (65);
\end{tikzpicture}
\caption{An IPS with $n+1=7$ levels and $N=5$ particles at each level.}
\label{aoecochz}
\end{figure}
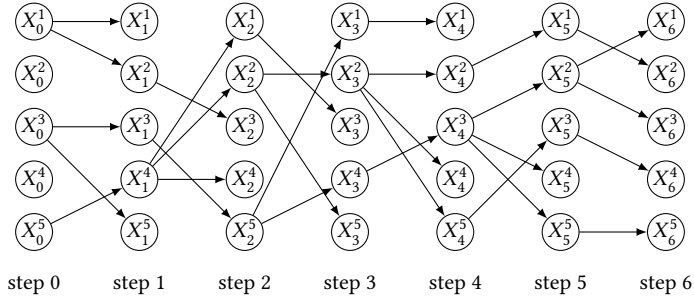

Suppose we want to estimate $\bar{\Gamma}_6^{(3)}(F)$ by $\bar{\Gamma}_{6,5}^{(3)}(F)$. 
We denote $b^* =  (0,0,0,1,0,0,0)$ the corresponding coalescence indicator. 
In the associated IPS, we have to find the choices of
$\ell_{0:6}^{[2]}$ such that
\begin{equation}\label{eq:b-star}
\prod_{p=0}^{5}
\lambda_p^{b^*}(A_p^{\ell_{p+1}^{[2]}},\ell_p^{[2]})
=1.
\end{equation}

\begin{figure}[htb]
  \centering
      \begin{tikzpicture}[scale = 0.7, every
        node/.style={transform shape}]
      \path [use as bounding box] (-6,-3) rectangle (8,3);

    \tikzset{norm/.style={circle, black, draw=black,densely dotted,
    fill=white, minimum size=1mm}}

    \tikzset{noire/.style={circle, black, draw=black,
    fill=black, minimum size=1mm}}

    \node [style=norm] (01) at (-5, 2) {};
    \node [style=norm] (02) at (-5, 1) {};
    \node [style=noire] (03) at (-5, -0) {};
    \node [style=norm] (04) at (-5, -1) {};
    \node [style=noire] (05) at (-5, -2) {};
    \node [] (step0) at (-5, -3) {step 0};

    \node [style=norm] (11) at (-3, 2) {};
    \node [style=norm] (12) at (-3, 1) {};
    \node [style=noire] (13) at (-3, -0) {};
    \node [style=noire] (14) at (-3, -1) {};
    \node [style=norm] (15) at (-3, -2) {};
    \node [] (step1) at (-3, -3) {step 1};

    \draw [arrows=-latex,style=densely dotted] (01) to (11);
    \draw [arrows=-latex,style=densely dotted] (01) to (12);
    \draw [arrows=-latex] (03) to (13);
    \draw [arrows=-latex] (05) to (14);
    \draw [arrows=-latex,style=densely dotted] (03) to (15);

    \node [style=norm] (21) at (-1, 2) {};
    \node [style=noire] (22) at (-1, 1) {};
    \node [style=norm] (23) at (-1, -0) {};
    \node [style=norm] (24) at (-1, -1) {};
    \node [style=noire] (25) at (-1, -2) {};
    \node [] (step2) at (-1, -3) {step 2};

    \draw [arrows=-latex,style=densely dotted] (14) to (21);
    \draw [arrows=-latex] (14) to (22);
    \draw [arrows=-latex,style=densely dotted] (12) to (23);
    \draw [arrows=-latex,style=densely dotted] (14) to (24);
    \draw [arrows=-latex] (13) to (25);

    \node [style=norm] (31) at (1, 2) {};
    \node [style=noire] (32) at (1, 1) {};
    \node [style=norm] (33) at (1, -0) {};
    \node [style=noire] (34) at (1, -1) {};
    \node [style=norm] (35) at (1, -2) {};
    \node [] (step3) at (1, -3) {step 3};

    \draw [arrows=-latex,style=densely dotted] (25) to (31);
    \draw [arrows=-latex] (22) to (32);
    \draw [arrows=-latex,style=densely dotted] (21) to (33);
    \draw [arrows=-latex] (25) to (34);
    \draw [arrows=-latex,style=densely dotted] (22) to (35);

    \node [style=norm] (41) at (3, 2) {};
    \node [style=noire] (42) at (3, 1) {};
    \node [style=norm] (43) at (3, -0) {};
    \node [style=norm] (44) at (3, -1) {};
    \node [style=noire] (45) at (3, -2) {};
    \node [] (step4) at (3, -3) {step 4};

    \draw [arrows=-latex,style=densely dotted] (31) to (41);
    \draw [arrows=-latex] (32) to (42);
    \draw [arrows=-latex,style=densely dotted] (34) to (43);
    \draw [arrows=-latex,style=densely dotted] (32) to (44);
    \draw [arrows=-latex] (32) to (45);

    \node [style=noire] (51) at (5, 2) {};
    \node [style=norm] (52) at (5, 1) {};
    \node [style=noire] (53) at (5, -0) {};
    \node [style=norm] (54) at (5, -1) {};
    \node [style=norm] (55) at (5, -2) {};
    \node [] (step5) at (5, -3) {step 5};

    \draw [arrows=-latex] (42) to (51);
    \draw [arrows=-latex,style=densely dotted] (43) to (52);
    \draw [arrows=-latex] (45) to (53);
    \draw [arrows=-latex,style=densely dotted] (43) to (54);
    \draw [arrows=-latex,style=densely dotted] (43) to (55);

    \node [style=norm] (61) at (7, 2) {};
    \node [style=noire] (62) at (7, 1) {};
    \node [style=norm] (63) at (7, -0) {};
    \node [style=noire] (64) at (7, -1) {};
    \node [style=norm] (65) at (7, -2) {};
    \node [] (step6) at (7, -3) {step 6};

    \draw [arrows=-latex,style=densely dotted] (52) to (61);
    \draw [arrows=-latex] (51) to (62);
    \draw [arrows=-latex,style=densely dotted] (52) to (63);
    \draw [arrows=-latex] (53) to (64);
    \draw [arrows=-latex,style=densely dotted] (55) to (65);

\end{tikzpicture}
\caption{The first family of $\ell_{0:6}^{[2]}$ such that \eqref{eq:b-star} is
verified.}
\label{kjzcbxkzjxb}
\end{figure}
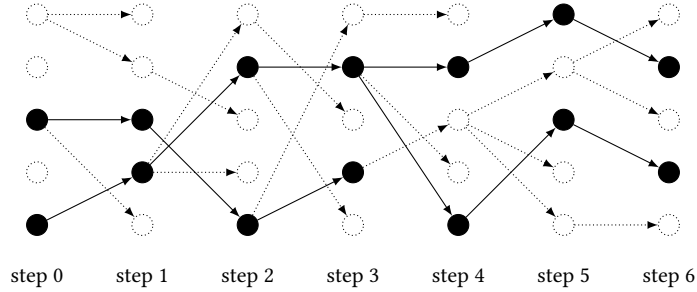
It turns out that there are 4 possible choices, taking into account that $F(x,x')$ is not necessarily symmetric in its variables. Namely, the first couple of ancestral lines is (see Figure \ref{kjzcbxkzjxb}):
\begin{itemize}
  \item
    $
    \ell_{0:6}^{[2]} = \left(
    (5,3),(4,3),(2,5),(2,4),(2,5),(1,3),(2,4)\right);
    $ 
\item
    $
    \ell_{0:6}^{[2]} = \left(
    (5,3),(4,3),(2,5),(2,4),(5,2),(3,1),(4,2)\right).
    $
\end{itemize}

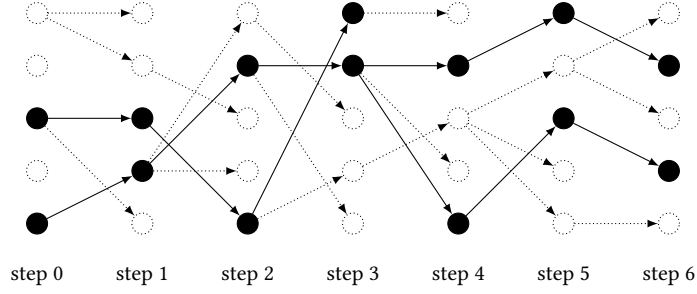
\begin{figure}[htb]
  \centering
      \begin{tikzpicture}[scale = 0.7, every
        node/.style={transform shape}]
      \path [use as bounding box] (-6,-3) rectangle (8,3);

    \tikzset{norm/.style={circle, black, draw=black,densely dotted,
    fill=white, minimum size=1mm}}

    \tikzset{noire/.style={circle, black, draw=black,
    fill=black, minimum size=1mm}}

    \node [style=norm] (01) at (-5, 2) {};
    \node [style=norm] (02) at (-5, 1) {};
    \node [style=noire] (03) at (-5, -0) {};
    \node [style=norm] (04) at (-5, -1) {};
    \node [style=noire] (05) at (-5, -2) {};
    \node [] (step0) at (-5, -3) {step 0};

    \node [style=norm] (11) at (-3, 2) {};
    \node [style=norm] (12) at (-3, 1) {};
    \node [style=noire] (13) at (-3, -0) {};
    \node [style=noire] (14) at (-3, -1) {};
    \node [style=norm] (15) at (-3, -2) {};
    \node [] (step1) at (-3, -3) {step 1};

    \draw [arrows=-latex,style=densely dotted] (01) to (11);
    \draw [arrows=-latex,style=densely dotted] (01) to (12);
    \draw [arrows=-latex] (03) to (13);
    \draw [arrows=-latex] (05) to (14);
    \draw [arrows=-latex,style=densely dotted] (03) to (15);

    \node [style=norm] (21) at (-1, 2) {};
    \node [style=noire] (22) at (-1, 1) {};
    \node [style=norm] (23) at (-1, -0) {};
    \node [style=norm] (24) at (-1, -1) {};
    \node [style=noire] (25) at (-1, -2) {};
    \node [] (step2) at (-1, -3) {step 2};

    \draw [arrows=-latex,style=densely dotted] (14) to (21);
    \draw [arrows=-latex] (14) to (22);
    \draw [arrows=-latex,style=densely dotted] (12) to (23);
    \draw [arrows=-latex,style=densely dotted] (14) to (24);
    \draw [arrows=-latex] (13) to (25);

    \node [style=noire] (31) at (1, 2) {};
    \node [style=noire] (32) at (1, 1) {};
    \node [style=norm] (33) at (1, -0) {};
    \node [style=norm] (34) at (1, -1) {};
    \node [style=norm] (35) at (1, -2) {};
    \node [] (step3) at (1, -3) {step 3};

    \draw [arrows=-latex] (25) to (31);
    \draw [arrows=-latex] (22) to (32);
    \draw [arrows=-latex,style=densely dotted] (21) to (33);
    \draw [arrows=-latex,style=densely dotted] (25) to (34);
    \draw [arrows=-latex,style=densely dotted] (22) to (35);

    \node [style=norm] (41) at (3, 2) {};
    \node [style=noire] (42) at (3, 1) {};
    \node [style=norm] (43) at (3, -0) {};
    \node [style=norm] (44) at (3, -1) {};
    \node [style=noire] (45) at (3, -2) {};
    \node [] (step4) at (3, -3) {step 4};

    \draw [arrows=-latex,style=densely dotted] (31) to (41);
    \draw [arrows=-latex] (32) to (42);
    \draw [arrows=-latex,style=densely dotted] (34) to (43);
    \draw [arrows=-latex,style=densely dotted] (32) to (44);
    \draw [arrows=-latex] (32) to (45);

    \node [style=noire] (51) at (5, 2) {};
    \node [style=norm] (52) at (5, 1) {};
    \node [style=noire] (53) at (5, -0) {};
    \node [style=norm] (54) at (5, -1) {};
    \node [style=norm] (55) at (5, -2) {};
    \node [] (step5) at (5, -3) {step 5};

    \draw [arrows=-latex] (42) to (51);
    \draw [arrows=-latex,style=densely dotted] (43) to (52);
    \draw [arrows=-latex] (45) to (53);
    \draw [arrows=-latex,style=densely dotted] (43) to (54);
    \draw [arrows=-latex,style=densely dotted] (43) to (55);

    \node [style=norm] (61) at (7, 2) {};
    \node [style=noire] (62) at (7, 1) {};
    \node [style=norm] (63) at (7, -0) {};
    \node [style=noire] (64) at (7, -1) {};
    \node [style=norm] (65) at (7, -2) {};
    \node [] (step6) at (7, -3) {step 6};

    \draw [arrows=-latex,style=densely dotted] (52) to (61);
    \draw [arrows=-latex] (51) to (62);
    \draw [arrows=-latex,style=densely dotted] (52) to (63);
    \draw [arrows=-latex] (53) to (64);
    \draw [arrows=-latex,style=densely dotted] (55) to (65);

\end{tikzpicture}
\caption{The second family of $\ell_{0:6}^{[2]}$ such that \eqref{eq:b-star} is
verified.}
\label{kjzcbxkzjxbg}
\end{figure}
The second couple of ancestral lines is (see Figure \ref{kjzcbxkzjxbg}):
\begin{itemize}
  \item
   $
   \ell_{0:6}^{[2]} = \left(
   (5,3),(4,3),(2,5),(2,1),(2,5),(1,3),(2,4)\right);
   $ 
  \item
   $
   \ell_{0:6}^{[2]} = \left(
   (5,3),(4,3),(2,5),(2,1),(5,2),(3,1),(4,2)\right).
   $
\end{itemize}
Hence, the number of choices of $\ell_{0:6}^{[2]}$ where $\ell_6^{[2]} = (2,4)$ is 2, and
the number of choices of $\ell_{0:6}^{[2]}$ where $\ell_6^{[2]} = (4,2)$ is also 2.
As a consequence, we have
$$
\bar{\Gamma}_{6,5}^{(3)}(F)
=
2\times\left\{\frac{5^5}{4^7}
\left(
  F(X_6^2,X_6^4)
  +
  F(X_6^4,X_6^2)
  \right)\right\}.
$$

Our next result ensures the convergence of our estimators.

\begin{theorem}[Convergence of $\Gamma_{n,N}^b$]\label{thm:coalescent-measure}
  Assume $\mathcal{A}$\ref{A1}-$\mathcal{A}$\ref{A2}. For any test functions 
$\phi,\psi\in \mathcal{B}_b(E_n)$ and for any coalescence indicator
$b\in\{0,1\}^{n+1}$, we have
$$
\Gamma_{n,N}^b(\phi\otimes\psi)
-
\Gamma_n^b(\phi\otimes\psi)
=\mathscr{O}_{\mathbf{p}}\left(\frac{1}{\sqrt{N}}\right).
$$
\end{theorem}

The proof is given in Section \ref{oachohc}.

\subsection{Term by term estimator}

Considering (\ref{eq:var-coal-gamma}), (\ref{eq:var-coal-eta}), and Theorem \ref{thm:coalescent-measure}, we are now in a position to provide term by term variance estimators for $\sigma^2_{\gamma_{n}}(f)$ and $\sigma^2_{\eta_{n}}(f)$.

\begin{definition}[Estimators of the asymptotic variances]\label{aocoicj}
Given a test
function $f\in \mathcal{B}_b(E_n)$, we let
$$
\sigma^2_{\gamma_{n,N}}(f) 
:=
\sum_{p=0}^n\left(
  \Gamma_{n,N}^{(p)}(f^{\otimes 2}) - 
    \Gamma_{n,N}^{(\varnothing)}(f^{\otimes 2}) 
\right),
$$
and
$$
\sigma^2_{\eta_{n,N}}(f) 
:=
\sum_{p=0}^n\left(
  \bar{\Gamma}_{n,N}^{(p)}(f^{\otimes 2}) - 
  \bar{\Gamma}_{n,N}^{(\varnothing)}(f^{\otimes 2}) 
\right).
$$
\end{definition}

Theorem \ref{thm:coalescent-measure} ensures the
consistency of both 
$\Gamma_{n,N}^{(p)}(f^{\otimes 2})$  
and 
$\Gamma_{n,N}^{(\varnothing)}(f^{\otimes 2})$. Returning to (\ref{eq:var-coal-gamma}),
this amounts to saying that
\begin{equation*}
  \sigma_{\gamma_n,N}^2(f)
  =
 \sum_{p=0}^n\left(
     \Gamma_{n,N}^{(p)}(f^{\otimes 2}) - 
     \Gamma_{n,N}^{(\varnothing)}(f^{\otimes 2}) 
\right)\xrightarrow[N\rightarrow\infty]{\mathbf{P}}
 \sum_{p=0}^n\left(
     \Gamma_{n}^{(p)}(f^{\otimes 2}) - 
     \Gamma_{n}^{(\varnothing)}(f^{\otimes 2})\right)= \sigma_{\gamma_n}^2(f).
\end{equation*}
Similarly, for the consistency of $\sigma^2_{\eta_{n,N}}(f-\eta_n^N(f))$, since by (\ref{eq:var-coal-eta}) we know that
$$\sigma_{\eta_n}^2(f) = \sum_{p=0}^{n} \left(
  \bar{\Gamma}_n^{(p)} (f^{\otimes 2})
- \bar{\Gamma}_n^{(\varnothing)}(f^{\otimes 2})\right),$$
it suffices to verify that, for any coalescent indicator $b$,
\begin{equation}
\label{pf:cvg-Gamma-bar}
\bar{\Gamma}_{n,N}^b
\left(
  \left[f-\eta_n^N(f)\right]^{\otimes 2}
\right) 
\xrightarrow[N\rightarrow\infty]{\mathbf{P}}
\bar{\Gamma}_{n}^b 
\left(
  \left[f-\eta_n^N(f)\right]^{\otimes 2}
\right). 
\end{equation}
Clearly, the linearity of $\bar{\Gamma}_{n,N}^b$ yields
\begin{multline*}
\bar{\Gamma}_{n,N}^b
\left(
 \left[f-\eta_n^N(f)\right]^{\otimes 2}
\right) 
\\
=
\bar{\Gamma}_{n,N}^b
(f^{\otimes 2})
-
\eta_n^N(f)
\left(
\bar{\Gamma}_{n,N}^b
\left(
  1\otimes f
\right) 
+
\bar{\Gamma}_{n,N}^b
\left(
  f\otimes 1
\right) 
\right)
+
\eta_n^N(f)^2
\bar{\Gamma}_{n,N}^b
\left(
  1^{\otimes 2}
\right).
\end{multline*}
Mutatis mutandis, the same relation holds for $\bar{\Gamma}_{n}^b
\left(
 \left[f-\eta_n^N(f)\right]^{\otimes 2}
\right) $. Since a by-product of Theorem \ref{thm-old:clt} is that
$$
\eta_n^N(f)
-
\eta_n(f)
=\mathscr{O}_{\mathbf{p}}\left(\frac{1}{\sqrt{N}}\right)
,
$$
the verification of \eqref{pf:cvg-Gamma-bar} is just a consequence of Theorem \ref{thm:coalescent-measure} and Slutsky's Lemma. Hence, we have obtained the following result.

\begin{theorem}[Consistency of $\sigma^2_{\gamma_{n,N}}$ and $\sigma^2_{\eta_{n,N}}$]\label{thm-old:tbt}
  Assume $\mathcal{A}$\ref{A1}-$\mathcal{A}$\ref{A2}. For $f \in
  \mathcal{B}_b(E_n)$, we have
    $$
    \sigma^2_{\gamma_{n,N}}(f) 
    -
    \sigma_{\gamma_n}^2(f)
    =\mathscr{O}_{\mathbf{p}}\left(\frac{1}{\sqrt{N}}\right)
    ,
    $$
    as well as
    $$
    \sigma^2_{\eta_{n,N}}(f-\eta_n^N(f)) 
    -
      \sigma_{\eta_n}^2(f-\eta_n(f))
    =\mathscr{O}_{\mathbf{p}}\left(\frac{1}{\sqrt{N}}\right).
    $$
\end{theorem}

  Even if the term by term estimator is
  very natural in theory, the computational cost is
  quite heavy in practice since one has to trace the whole genealogy of a
  particle system and calculate all the corresponding terms one by one.
  Therefore, we do not provide an efficient algorithm to calculate this
  estimator. Instead, we show in the next section that this estimator can be
  connected to the one given by Lee \& Whiteley in a nonadaptive context (SMC), which is very
  simple and fast to calculate. Let us also mention that our term by term estimator
  is different from the one introduced in Section 4.1 of
  \cite{lee2018var}. The interested reader can find more details on this point in Appendix
  \ref{final-remarks}.

\subsection{Disjoint ancestral lines estimator}
Let us now recall the variance estimator proposed in \cite{lee2018var}, which can be seen as a disjoint ancestral lines estimator. Namely, given a test
function $f\in \mathcal{B}_b(E_n)$, consider
\begin{equation}\label{ljsncn}
V_n^N(f) := \eta_n^N(f)^2 - \frac{N^{n-1}}{(N-1)^{n+1}}
\sum_{E_n^i\neq E_n^j} f(X_n^i)
f(X_n^j),
\end{equation}
where $E_n^i$ is the ancestor index of $X_n^i$ at level 0. Returning to the toy example of Section \ref{ajsncoac}, the couples $(i,j)$ such that $i<j$ and $E_n^i\neq E_n^j$ are: $(1,2), (1,4), (2,3), (2,5), (3,4), (4,5)$.

In a nonadaptive context (SMC), this is the variance estimator introduced in \cite{lee2018var} when the number $N$ of
particles is the same at each step. The reader is referred to  
\cite{lee2018var} for an efficient algorithm to compute this estimator. 

According to our notation, since $E_n^i\neq E_n^j$ corresponds to the case $b=(0,\dots,0)=(\varnothing)$ of disjoint ancestral lines, we may also write
$$
V_n^N(f) =
\eta_n^N(f)^2 - \bar\Gamma_{n,N}^{(\varnothing)}(f^{\otimes 2}).
$$

The following proposition makes a connection between $V_n^N(f)$ and our estimators. Notice that this result does not depend on $\mathcal{A}$\ref{A2}, but is provided by the structure of the IPS and the
underlying multinomial selection scheme. The proof is housed in Section \ref{zicj}.

\begin{proposition}\label{lm:convergence-proba}
  Assume $\mathcal{A}$\ref{A1}.
  For any test function $f \in \mathcal{B}_b(E_n)$, 
  we have 
  $$
  N V_n^N(f)
  -
 \sigma^2_{\eta_{n,N}}(f) 
  =\mathscr{O}_{\mathbf{p}}\left(\frac{1}{N}\right),
  $$
  and
  $$
  N V_n^N(f-\eta_n^N(f))
  -
 \sigma^2_{\eta_{n,N}}(f-\eta_n^N(f))
  =
  \mathscr{O}_{\mathbf{p}}\left(\frac{1}{N}\right).
  $$
\end{proposition}

By combining Theorem \ref{thm-old:tbt} and
Proposition \ref{lm:convergence-proba}, we finally obtain the main result of
the present article. 
\begin{theorem}\label{thm:main}
  Assume $\mathcal{A}$\ref{A1}-$\mathcal{A}$\ref{A2}. For any test function $f \in
  \mathcal{B}_b(E_n)$, we have
    $$
    N\gamma_n^N(1)^2V_n^N(f) 
    -
    \sigma_{\gamma_n}^2(f)
    =\mathscr{O}_{\mathbf{p}}\left(\frac{1}{\sqrt{N}}\right)
    ,
    $$
    and
    $$
    NV_n^N(f-\eta_n^N(f)) 
    -
    \sigma_{\eta_n}^2(f-\eta_n(f))
    =\mathscr{O}_{\mathbf{p}}\left(\frac{1}{\sqrt{N}}\right).
    $$
\end{theorem}

Hence, the main message of the present work is that the computationally very simple estimator proposed by Lee and Whiteley in a  nonadaptive framework (SMC) is still consistent in an adaptive one (ASMC). However, since we could not adapt easily their proof to our adaptive context, we propose a new approach to show this consistency result. More details on the connection between both estimators are given in Appendix \ref{final-remarks}.

As emphasized before, among other ingredients, the tools we use connect the study of Particle Markov Chain Monte Carlo methods and the variance estimation problem in SMC methods. As such, more generally, they may give some new insights when dealing with complex genealogy-involved problems of Interacting Particle Systems. 

Before going into the details of the proofs, let us mention that a numerical experiment on a toy example is proposed in Appendix \ref{sec:num} to illustrate the consistency of the Lee and Whiteley variance estimators in the adaptive and nonadaptive cases. Finally, Appendix \ref{jscnnlncln} makes a connection between our term by term estimators and the truncated variance estimators recently proposed by Olsson and Douc in \cite{olsson2019} to address the issue of degeneracy in the ancestral lines.


\section{Proofs}
\subsection{Almost sure convergence}

In this section, we provide classical almost sure convergence results on SMC framework
under our specific  parameterization, namely with adaptive potential functions
and transition kernels. We focus on the properties that do not use the
additional information given by the genealogy of the associated IPS.
Therefore, in order to simplify the story, we give a ``rougher''
definition of the associated IPS without considering the genealogy.  

\begin{itemize}
  \item $\mathbf{X_{0}}\sim \eta_0^{\otimes N}$
  \item For $p\geq 1$, we let
    $$
    \mathbf{X_{p}}\sim \bigotimes_{i=1}^N K_{p,\eta_{p-1}^N}(X_{p-1}^i,\cdot)
    $$
    where, given $\mathbf{X_{p-1}}$, $K_{p,\eta_{p-1}^N}$ is the Markov kernel defined by
    $$
    \forall(x,A) \in E_{p-1}\times \mathcal{B}(E_p),
    \qquad
    K_{p,\eta_{p-1}^N}(x,A) := \frac{\eta_{p-1}^N
    Q_{p,N}(x,A)}{\eta_{p-1}^N\left(G_{p-1,N}\right)}.
    $$
\end{itemize}
It is easy to check that the distributions of the particles are identical to
the ones defined in Section \ref{sec-old:IPS}. Let us begin with the consistency of
the corresponding adaptive estimators. Recall that, by 
$\mathcal{A}$\ref{A2}, the summary statistics
$\zeta_n=(\zeta_n^1,\cdots,\zeta_n^d)$ satisfies $\eta_n(\zeta_n)=z_n^*$ and,
for all $k\in[d]$, $\zeta_n^k$ belongs to $\mathcal{B}_b(E_n)$.

\begin{theorem}\label{thm:consistency-gamma}
  Assume $\mathcal{A}$\ref{A1}-$\mathcal{A}$\ref{A2}.
  For any $f\in \mathcal{B}_b(E_n)$, we
  have
  $$
  \gamma_n^N(f) 
  \xrightarrow[N\rightarrow \infty]{a.s.}
  \gamma_n(f), 
  $$
  and
  $$
  \eta_n^N(f) 
  \xrightarrow[N\rightarrow \infty]{a.s.}
  \eta_n(f).
  $$
  In particular, we also have
  $$
  Z_n^N=\eta_n^N(\zeta_n) 
  \xrightarrow[N\rightarrow \infty]{a.s.}
  \eta_n(\zeta_n)=z_n^*.
  $$
\end{theorem}
\begin{proof}
  By definition, it is clear that the convergence of $\gamma_n^N$ implies the
  convergence of $\eta_n^N$. Therefore, it is sufficient to establish the first
  one. We prove by induction that
  $$
  \forall f\in \mathcal{B}_b(E_n),\quad
  \gamma_n^N(f) 
  \xrightarrow[N\rightarrow \infty]{a.s.}
  \gamma_n(f). 
  $$
  
  \medskip

  \noindent
  Step 0:
  \medskip

  The almost sure convergence of $\gamma_0^N=\eta_0^N$ to $\gamma_0=\eta_0$ with respect to a test function in $\mathcal{B}_b(E_0)$ is given by
  the strong law of large numbers.  
  
   \medskip
  \noindent
  Step $n\geq 1$:
  \medskip

  We assume that 
  $$
  Z_{n-1}^N
  \xrightarrow[N\rightarrow \infty]{a.s.}
  z_{n-1}^*
 $$
  and, for any $\phi \in \mathcal{B}_b(E_{n-1})$, 
  $$
  \gamma_{n-1}^N(\phi) 
  \xrightarrow[N\rightarrow \infty]{a.s.}
  \gamma_{n-1}(\phi).
  $$
  For any $f\in \mathcal{B}_b(E_n)$, the triangular inequality yields
  \begin{multline}\label{cjcjaocjjc}
    \abs{\gamma_n^N(f)-\gamma_n(f)}
    \\
    \leq
    \underbrace{\abs{\gamma_n^N(f) - \gamma_{n-1}^NQ_{n,N}(f)}}_{P_1(N)}
    +
    \underbrace{\abs{\gamma_{n-1}^NQ_{n,N}(f) -
    \gamma_{n-1}^NQ_{n}(f)}}_{P_2(N)}
    +
    \underbrace{\abs{\gamma_{n-1}^NQ_{n}(f) - \gamma_{n-1}Q_{n}(f)}}_{P_3(N)}.
  \end{multline}
  \begin{itemize}
    \item
      For $P_1(N)$, we denote 
      $$
      U_{n,N}^i :=
      \eta_{n-1}^N(G_{n-1,N}) f(X_n^i)
      -
      \eta_{n-1}^NQ_{n,N}(f).
      $$
      It is readily seen that
      $$
      P_1(N) = \gamma_{n-1}^N(1)
      \frac{1}{N}
      \sum_{i=1}^N U_{n,N}^i.
      $$
      Given $\mathscr{F}_{n-1}^N:=\sigma(\mathbf{X_{0}},\dots,\mathbf{X_{n-1}})$, the random variables $(U_{n,N}^i)_{1\leq i\leq N}$ are i.i.d. and such that
      $$
      \mathbf{E}\left[
        U_{n,N}^i
        \Given{\mathscr{F}_{n-1}^N}
      \right]
      =
      \eta_{n-1}^N(G_{n-1,N})
      \frac{\eta_{n-1}^N
      Q_{n,N}(f)}{\eta_{n-1}^N\left(G_{n-1,N}\right)}
      -
      \eta_{n-1}^NQ_{n,N}(f)
      =0.
      $$
      Under $\mathcal{A}$\ref{A1}, we can also see that
      $$
      \abs{U_{n,N}^i} \leq  C_n := 2\norm{G_{n-1,\cdot}}_{\infty}\norm{f}_{\infty}.
      $$
      Therefore, for any $\epsilon>0$, Hoeffding's inequality gives
      $$
      \mathbf{P}\left(\abs{\sum_{i=1}^N U_{n,N}^i} \geq N\epsilon\;\Bigg\vert\; \mathscr{F}_{n-1}^N\right)
      \leq
      2\exp\left(
        \frac{-\epsilon^2 N}{2C_n^2}
      \right).
      $$
      Since this upper-bound is deterministic, this amounts to saying that
      $$
      \mathbf{P}\left(
        \abs{\sum_{i=1}^N U_{n,N}^i} \geq N\epsilon
      \right)
      \leq
      2\exp\left(
        \frac{-\epsilon^2 N}{2C_n^2}
      \right).
      $$
      Consequently, Borel-Cantelli
      Lemma ensures that 
      $$
      \frac{1}{N}\sum_{i=1}^N U_{n,N}^i
      \xrightarrow[N\rightarrow \infty]{a.s.}
      0.
      $$
      Combined with the induction hypothesis, we get
      $$
      P_1(N)= \gamma_{n-1}^N(1)
      \frac{1}{N}
      \sum_{i=1}^N U_{n,N}^i
      \xrightarrow[N\rightarrow \infty]{a.s.}
      0.
      $$
    \item
      For $P_2(N)$, $\mathcal{A}$\ref{A2} implies that there exists a function
      $h_{n-1}$ such that
      $$
      Q_{n,N}(f)(x) - Q_n(f) (x)
      =
      \left\langle h_{n-1}(x,Z_{n-1}^N), Z_{n-1}^N - z_{n-1}^*\right\rangle.
      $$
      Hence, since $\abs{h_{n-1}}$ and the potential functions $G_{n,z}$ are bounded,
      Cauchy-Schwarz inequality gives 
      $$
      P_2(N) 
      \leq
      \gamma_{n-1}^N (1) 
      \norm{h_{n-1}}_{\infty}
      \abs{Z_{n-1}^N -z_{n-1}^*}
      \leq
      \left\{
      \prod_{p =0}^{n-2} \norm{G_{p,\cdot}}_{\infty}
      \right\}
      \norm{h_{n-1}}_{\infty}
      \abs{Z_{n-1}^N -z_{n-1}^*}.
       $$
     By induction hypothesis, we conclude that
      $$
      P_2(N)
      \xrightarrow[N\rightarrow \infty]{a.s.}
      0.
      $$
    \item
      For $P_3(N)$, under $\mathcal{A}$\ref{A1}, we have that $Q_n(f)\in \mathcal{B}_b(E_{n-1})$. 
      Thus, the induction hypothesis gives 
            $$
      P_3(N)
      \xrightarrow[N\rightarrow \infty]{a.s.}
      0.
      $$
  \end{itemize}
  Considering (\ref{cjcjaocjjc}), the verification of the convergence
  $$
  \forall f\in\mathcal{B}_b(E_n),
  \quad
  \gamma_n^N(f) 
  \xrightarrow[N\rightarrow \infty]{a.s.}
  \gamma_n(f)
  $$
  is then complete.
\end{proof}

\subsection{Proof of Theorem \ref{thm-old:clt}}\label{sec:pf-clt}

  We prove by induction that
  $$
  \sqrt{N}
  \left(
  \gamma_n^N(f) - \gamma_n(f) 
  \right)
  \xrightarrow[N\rightarrow \infty]{\mathrm{d}}
  \mathcal{N}\left(
    0,\sigma_{\gamma_n}^2(f)
  \right).
  $$
  The verification of step 0 comes from the CLT for i.i.d. random variables. For step $n\geq 1$,
  we suppose that 
  $$
  \forall 0\leq p\leq n-1,
\quad
  \sqrt{N}
  \left(
    \gamma_{p}^N(f) - \gamma_{p}(f) 
  \right)
  \xrightarrow[N\rightarrow \infty]{\mathrm{d}}
  \mathcal{N}\left(
    0,\sigma_{\gamma_{p}}^2(f)
  \right).
  $$
  Notice that, by $\mathcal{A}$\ref{A2}, this implies that
      \begin{equation}\label{pf:ZN}
       \forall 0\leq p\leq n-1,
\quad
      \sqrt{N}
      \abs{Z_{p-1}^N - z_{p-1}^*}
      =
      \mathscr{O}_{\mathbf{p}}\left(1\right).
      \end{equation}
  For any test function $f\in \mathcal{B}_b(E_n)$, we denote $f_p := Q_{p,n}(f)\in \mathcal{B}_b(E_p)$. For any $(x,A)\in E_0\times{\cal B}(E_0)$ we set $Q_0(x,A)=Q_{0,N}(x,A)=\delta_x(A)$. 
  Taking into account the convention $\gamma_{-1}^N = \gamma_0=\eta_0$ and the fact that $\gamma_n=\gamma_0 Q_{0,n}$, we have the telescoping decomposition 
  \begin{equation*}
  \begin{split}
  &\gamma_n^N(f) - \gamma_n(f)
  \\
  &=
  \sum_{p=0}^n\left(
    \gamma_p^N(f_p) - \gamma_{p-1}^NQ_p(f_p)
  \right)
  \\
  &=
  \frac{1}{N}
  \sum_{p=0}^n
  \sum_{i=1}^N
  \left\{
    \left(
    \gamma_p^N(1)f_p(X_p^i) - \gamma_{p-1}^NQ_{p,N}(f_p)
  \right)
  +
  \left(
    \gamma_{p-1}^NQ_{p,N}(f_p)
    -
    \gamma_{p-1}^NQ_{p}(f_p)
  \right)
  \right\}.
  \end{split}
  \end{equation*}
  For $k \in [(n+1)N]$, we denote 
  $$
  p_k := 
  \floor*
  {\frac{k}{N}}
  \qquad
  \text{and}
  \qquad
  i_k :=
  k - p_k \times N.
  $$
  We define the filtration 
  $$
  \forall k \in[(n+1)N],
  \quad
  \mathscr{E}_k^N 
  =
  \mathscr{F}_{p_k-1}^N 
  \vee
  \sigma(X_{p_k}^1,\cdots,X_{p_k}^{i_k}).
  $$
  Then, we set
  $$
  U_k^N 
  := 
  \frac{1}{\sqrt{N}}
  \left(
  \gamma_{p_k}^N(1)f_{p_k}(X_{p_k}^{i_k}) -
  \gamma_{p_{k}-1}^NQ_{p_{k},N}(f_{p_k})
  \right),
  $$
  and
  $$
  D_p^N := 
  \sqrt{N}
  \left(
  \gamma_{p-1}^NQ_{p,N}(f_p)
  -
  \gamma_{p-1}^NQ_{p}(f_p)
  \right),
  $$
  so that
  \begin{equation}\label{pf:clt-decomp}
  \sqrt{N}
  \left(
  \gamma_n^N(f) - \gamma_n(f)
  \right)
  =
  \sum_{k=1}^{(n+1)N}
  \left(
    U_k^N
  +
  \frac{1}{N} D_{p_k}^N
  \right)
  =
  \sum_{k=1}^{(n+1)N}
    U_k^N
  +
  \sum_{p=0}^{n}
    D_{p}^N.
  \end{equation}
  From $\mathcal{A}$\ref{A2}, we know that there exists a function
  $h_{p-1}$ such that 
  \begin{equation*}
    \begin{split}
  D_p^N
  =
  &
  \sqrt{N} \left\langle \gamma_{p-1}^N \left(h_{p-1}(\cdot,Z_{p-1}^N)\right),Z_{p-1}^N - z_{p-1}^*\right\rangle
  \\
  =
  &
  \sqrt{N}
 \left\langle \gamma_{p-1}^N\left(
  h_{p-1}(\cdot,Z_{p-1}^N)
  -
  h_{p-1}(\cdot,z_{p-1}^*)
\right),
  Z_{p-1}^N - z_{p-1}^*\right\rangle
  \\
  &
  +
  \sqrt{N}
  \left\langle \gamma_{p-1}^N\left(
  h_{p-1}(\cdot,z_{p-1}^*)
\right),
  Z_{p-1}^N - z_{p-1}^*\right\rangle.
    \end{split}
  \end{equation*}
  \begin{itemize}
    \item
      For the first part, we have by Cauchy-Schwarz inequality
      \begin{equation*}
        \begin{split}
          &
          \left| \sqrt{N}
 \left\langle \gamma_{p-1}^N\left(
  h_{p-1}(\cdot,Z_{p-1}^N)
  -
  h_{p-1}(\cdot,z_{p-1}^*)
\right),
  Z_{p-1}^N - z_{p-1}^*\right\rangle\right|
         \\
         \leq 
          &
          \sqrt{N}\
          \gamma_{p-1}^N(1)
           \abs{
         Z_{p-1}^N - z_{p-1}^*}
         \sup_{x\in E_{p-1}}
         \abs{
         h_{p-1}(x,Z_{p-1}^N)
         -
         h_{p-1}(x,z_{p-1}^*)
         }.
        \end{split}
      \end{equation*}
        Then, let us consider 
          $$
  \Omega_{p-1} =
  \left\{
    \omega\in\Omega:
    Z_{p-1}^N(\omega)
    \xrightarrow[N\rightarrow \infty]{}
    z_{p-1}^*
  \right\}.
  $$
By Theorem \ref{thm:consistency-gamma}, $\Omega_{p-1}$ has probability one. Therefore, 
      by $\mathcal{A}$\ref{A2}, for all $\omega\in\Omega_{p-1}$ and all $\epsilon>0$, there exists $N(\omega,\epsilon)>0$ such that, for all
      $N>N(\omega,\epsilon)$, 
      \begin{equation*}
       \sup_{x\in E_{p-1}}
       \abs{
       h_{p-1}(x,Z_{p-1}^N(\omega))
       -
       h_{p-1}(x,z_{p-1}^*)
       }
      <
      \epsilon.
      \end{equation*}
      This means that
      $$
      \sup_{x\in E_{p-1}}
      \abs{
      h_{p-1}(x,Z_{p-1}^N)
      -
      h_{p-1}(x,z_{p-1}^*)
      }
      \xrightarrow[N\rightarrow\infty]{a.s.}
      0.
      $$
      Thus, we deduce from (\ref{pf:ZN}) that
      $$
      \sqrt{N}
 \left\langle \gamma_{p-1}^N\left(
  h_{p-1}(\cdot,Z_{p-1}^N)
  -
  h_{p-1}(\cdot,z_{p-1}^*)
\right),
  Z_{p-1}^N - z_{p-1}^*\right\rangle
      \xrightarrow[N\rightarrow\infty]{\mathbf{P}}
      0.
      $$
    \item
      For the second part, since Theorem \ref{thm:consistency-gamma} and 
      $\mathcal{A}$\ref{A2} imply that
      $$
      \gamma_{p-1}^N(h_{p-1}(\cdot,z_{p-1}^*))
      \xrightarrow[N\rightarrow\infty]{a.s.}
      \gamma_{p-1}(h_{p-1}(\cdot,z_{p-1}^*)) = 0,
      $$
      we conclude by \eqref{pf:ZN} that  
      $$
      \sqrt{N}
  \left\langle \gamma_{p-1}^N\left(
  h_{p-1}(\cdot,z_{p-1}^*)
\right),
  Z_{p-1}^N - z_{p-1}^*\right\rangle
      \xrightarrow[N\rightarrow\infty]{\mathbf{P}}
      0.
      $$
  \end{itemize}
  Hence we have proved that
  $$
  D_p^N
  \xrightarrow[N\rightarrow\infty]{\mathbf{P}}
  0,
  $$
  which leads to
  $$
  \sum_{p=0}^{n}
  D_{p}^N
  \xrightarrow[N\rightarrow\infty]{\mathbf{P}}
  0.
  $$
  Next, it is easy to check that $(U_k^N)_{1\leq k\leq (n+1)N}$ is an
  $(\mathscr{E}_k^N)_{1\leq k\leq (n+1)N}$-martingale difference array.  
  In order to apply Theorem 2.3 in \cite{mg-CLT}, we just have to check that
  \begin{itemize}
    \item
      By $\mathcal{A}$\ref{A1}, 
      \begin{equation}\label{pf:max-U}
      \max_{1\leq k\leq (n+1)N}\left|U_k^N\right|
       \leq
      \frac{2}{\sqrt{N}}
            \norm{f}_{\infty}
      \max_{1\leq p\leq n}
      \prod_{q=0}^{p-1}\norm{G_{q,\cdot}}_{\infty}
      \leq
      \frac{2}{\sqrt{N}}
            \norm{f}_{\infty}
      \sum_{p=1}^{n}
      \prod_{q=0}^{p-1}\norm{G_{q,\cdot}}_{\infty},
      \end{equation}
      which shows that $(\max_{1\leq k\leq (n+1)N}|U_k^N|)$ is uniformly
      bounded in $\mathbb{L}^{\mathrm{2}}$-norm.
    \item
      From \eqref{pf:max-U}, we also get that
      $$
      \max_{1\leq k\leq (n+1)N}\left|U_k^N\right|
      \xrightarrow[N\rightarrow\infty]{\mathbf{P}}
      0.
      $$
    \item
      Standard calculation gives 
      $$
      \begin{aligned}
        &
      \sum_{k=1}^{(n+1)N}
      \left(
      U_k^N
      \right)^2
      \\=&
      \sum_{p=0}^n
      \left(
      \gamma_p^N(1)^2 \eta_p^N(f_p^2)
      +
      (\gamma_{p-1}^N Q_{p,N}(f_p))^2
      -
      2      \gamma_p^N(1)\eta_p^N(f_p)
           \gamma_{p-1}^NQ_{p,N}(f_p)
      \right).
      \end{aligned}
      $$
      As shown above, the convergence of $D_p^N$ indicates that
      $$
      \gamma_{p-1}^N Q_{p,N}(f_p)
      -
      \gamma_{p-1}^N Q_{p}(f_p)
      \xrightarrow[N\rightarrow\infty]{\mathbf{P}}
      0.
      $$
      Then, by applying Theorem \ref{thm:consistency-gamma}, we obtain 
      $$
      \sum_{k=1}^{(n+1)N}
      \left(
      U_k^N
      \right)^2
      \xrightarrow[N\rightarrow\infty]{\mathbf{P}}
      \sigma_{\gamma_n}^2(f),
      $$
  \end{itemize}
  Therefore, we have the following central limit theorem 
  $$
  \sum_{k=1}^{(n+1)N}
  U_k^N
  \xrightarrow[N\rightarrow\infty]{\mathrm{d}}
  \mathcal{N}\left(
    0,\sigma_{\gamma_n}^2(f)
  \right).
  $$
  Returning to \eqref{pf:clt-decomp}, 
  the conclusion follows from Slutsky's Lemma. 

\subsection{Proof of Theorem \ref{thm:coalescent-measure}}\label{oachohc}
We want to show that, under $\mathcal{A}$\ref{A1}-$\mathcal{A}$\ref{A2}, for any test functions 
$\phi,\psi\in \mathcal{B}_b(E_n)$ and for any coalescence indicator
$b\in\{0,1\}^{n+1}$, we have
$$
\Gamma_{n,N}^b(\phi\otimes\psi)
-
\Gamma_n^b(\phi\otimes\psi)
=\mathscr{O}_{\mathbf{p}}\left(\frac{1}{\sqrt{N}}\right).
$$
Before proceeding, let us introduce some additional notation.
With a slight abuse of notation, for a coalescence indicator $b= (b_0,\dots,b_n)\in
\{0,1\}^{n+1}$, we denote, for all $0\leq p\leq n$, 
$$
\Gamma_{p}^b := \Gamma_{p}^{(b_0,\dots,b_p)}
\qquad
\text{and}
\qquad
\Gamma_{p,N}^b := \Gamma_{p,N}^{(b_0,\dots,b_p)}
$$
with the convention
$$
\Gamma_{-1,N}^b=\Gamma_{-1}^b
:= \eta_0^{\otimes 2} C_{b_0}.
$$
Note that, with this convention, we have
$$
\Gamma_{p}^b=\Gamma_{p-1}^b Q_p^{\otimes 2}C_{b_p}.
$$
We also remark that, for any $b_n \in \{0,1\}$ and any
$\phi,\psi\in\mathcal{B}_b(E_n)$, there exists $f$ and $g$ in $\mathcal{B}_b(E_{n-1})$
such that
\begin{equation}\label{cnaocn}
Q_{n}^{\otimes 2}C_{b_n} (\phi\otimes \psi)
=f\otimes g.
\end{equation}
Specifically, for $b_n = 0$, it suffices to consider
$f =Q_{n}(\phi)$ and $g =Q_{n}(\psi)$, 
while for $b_n = 1$ one can take
$f =Q_{n}(\phi\psi)$ and $g =Q_{n}(1)=G_{n-1}$. 
As usual, the proof is done by induction. 
\medskip

\noindent
- Step 0:

\begin{itemize}
  \item If $b_0=1$, \eqref{sjcn} and Definition \ref{aecac} give 
    $$
    \mathbf{E}\left[
    \Gamma_{0,N}^{b}(\phi\otimes\psi)
    \right]
    = 
    \mathbf{E}\left[
    \frac{1}{N}\sum_{i=1}^N \phi(X_0^i) \psi(X_0^i)
    \right]
    =
    \eta_0(\phi\psi)=\Gamma_{0}^{b}(\phi\otimes\psi).
    $$
    Hence, the central limit theorem yields
    $$
    \sqrt{N}
    \left(
    \Gamma_{0,N}^{b}(\phi\otimes\psi)
    -
    \Gamma_{0}^{b}(\phi\otimes\psi)
  \right)
  \xrightarrow[N\rightarrow\infty]{\mathrm{d}}
  \mathcal{N}
  \left(
    0,
    \eta_0(\phi^2\psi^2) - \eta_0(\phi\psi)^2
  \right),
    $$
    so that
    $$
    \Gamma_{0,N}^{b}(\phi\otimes\psi)
    -
    \Gamma_{0}^{b}(\phi\otimes\psi)
    =\mathscr{O}_{\mathbf{p}}\left(\frac{1}{\sqrt{N}}\right).
    $$
  \item If $b_0=0$, the central limit theorem ensures that
    $$
    \eta_0^N(\phi) 
    -
    \eta_0(\phi) 
    = \mathscr{O}_{\mathbf{p}}\left(\frac{1}{\sqrt{N}}\right)
    \quad\text{and}\quad
    \eta_0^N(\psi)
    -
    \eta_0(\psi)
    = \mathscr{O}_{\mathbf{p}}\left(\frac{1}{\sqrt{N}}\right).
    $$
    Therefore, we have
    $$
    \begin{aligned}
    &\eta_0^N(\phi)\eta_0^N(\psi)
    -
    \eta_0(\phi)\eta_0(\psi)
    \\&=
    \left(\eta_0^N(\phi) - \eta_0(\phi)\right)\eta_0^N(\psi)
    +
    \eta_0(\phi)
    \left(\eta_0^N(\psi) - \eta_0(\psi)\right)
      \\&=\mathscr{O}_{\mathbf{p}}\left(\frac{1}{\sqrt{N}}\right).
    \end{aligned}
    $$
    Thanks to \eqref{lasjcn}, one has 
    $$
    \Gamma_{0,N}^b(\phi\otimes \psi)
    = \frac{N}{N-1}\left(
    \eta_0^N(\phi)\eta_0^N(\psi)
    -\frac{1}{N^2}\sum_{i=1}^N\phi(X_0^i)\psi(X_0^i)
    \right).
    $$
    Combined with Definition \ref{aecac} and the law of large numbers, one deduces that
    $$
    \Gamma_{0,N}^b(\phi\otimes \psi)
    -
    \Gamma_{0}^b(\phi\otimes \psi)
    =  \Gamma_{0,N}^b(\phi\otimes \psi)
    - \eta_0(\phi)\eta_0(\psi)
    =\mathscr{O}_{\mathbf{p}}\left(\frac{1}{\sqrt{N}}\right).
    $$
\end{itemize}

\noindent
- Step $n\geq 1$:
\medskip

We suppose that for any test functions $f,g\in\mathcal{B}_b(E_{n-1})$ and
coalescence indicator $b$, we have
$$
\Gamma_{n-1,N}^b(f\otimes g)
-
\Gamma_{n-1}^b(f\otimes g)
=\mathscr{O}_{\mathbf{p}}\left(\frac{1}{\sqrt{N}}\right).
$$
Next, we consider the following decomposition
\begin{equation}\label{eq:decomp-Gamma-N}
  \begin{aligned}
\Gamma_{n,N}^b(\phi\otimes \psi)
-
\Gamma_{n}^b(\phi\otimes \psi)
    =&
\underbrace{
\Gamma_{n,N}^b(\phi\otimes \psi)
-
\Gamma_{n-1,N}^bQ_{n,N}^{\otimes 2}C_{b_n}(\phi\otimes \psi)
}_{R_1(N)}
   \\&+
\underbrace{
\Gamma_{n-1,N}^bQ_{n,N}^{\otimes 2}C_{b_n}(\phi\otimes \psi)
-
\Gamma_{n-1,N}^bQ_{n}^{\otimes 2}C_{b_n}(\phi\otimes \psi)
}_{R_2(N)}
\\
     &+
\underbrace{
\Gamma_{n-1,N}^bQ_{n}^{\otimes 2}C_{b_n}(\phi\otimes \psi)
-
\Gamma_{n-1}^bQ_{n}^{\otimes 2}C_{b_n}(\phi\otimes \psi).
}_{R_3(N)}
\end{aligned}
\end{equation}
The tools to terminate the proof are the following ones:
\begin{itemize}
  \item Lemma \ref{lm:cvg-proba-Gamma} shows that
    $$
    R_1(N)
    =\mathscr{O}_{\mathbf{p}}\left(\frac{1}{\sqrt{N}}\right).
    $$
  \item Lemma \ref{lm:cvg-Gamma-Q} and the fact that one may write $C_{b_n}(\phi\otimes \psi)$ as $f\otimes g$ for any $b_n$  ensure that
    $$
    R_2(N)
    =\mathscr{O}_{\mathbf{p}}\left(\frac{1}{\sqrt{N}}\right).
    $$
  \item Finally, the convergence rate
    $$
    R_3(N)
    =\mathscr{O}_{\mathbf{p}}\left(\frac{1}{\sqrt{N}}\right).
    $$
    is a direct consequence of (\ref{cnaocn}) and the induction hypothesis.
\end{itemize}

\subsection{Technical results}\label{sec:TR}

This section presents some useful technical results. Before going further, remind that
$$
\Gamma_{n,N}^b(1):=
\gamma_n^N(1)^2
\frac{N^{n-1}}{(N-1)^{n+1}}
\sum_{\ell_{0:n}^{[2]}\in \left((N)^2\right)^{\times(n+1)}}
\left\{\prod_{p=0}^{n-1}
\lambda_p^b(A_p^{\ell_{p+1}^{[2]}},\ell_p^{[2]})
\right\}.
$$
If we set 
\begin{equation}\label{eq:LambdaElln}
\Lambda_{n}^{\ell_n^{[2]}}
:=
\sum_{\ell_{0:n-1}^{[2]}\in \left((N)^2\right)^{\times n}}
\left\{\prod_{p=0}^{n-1}
\lambda_p^b(A_p^{\ell_{p+1}^{[2]}},\ell_p^{[2]})
\right\}
\end{equation}
together with the convention
$
\Lambda_{0}^{\ell_0^{[2]}}
:=1,
$ 
we may write
\begin{equation}\label{pf:Lambda}
\Gamma_{n,N}^b(1):=
\gamma_n^N(1)^2
\frac{N^{n-1}}{(N-1)^{n+1}}
\sum_{\ell_n^{[2]}\in (N)^2}
\Lambda_{n}^{\ell_n^{[2]}},
\end{equation}
so that
\begin{equation}\label{pf:Lambda2}
\Gamma_{n,N}^b(1)^2=
\gamma_n^N(1)^4
\left(\frac{N^{n-1}}{(N-1)^{n+1}}\right)^2
\sum_{(\ell_n^{[2]},\ell_n^{'[2]}) \in ((N)^2)^{\times 2}}
\Lambda_{n}^{\ell_n^{[2]}}\Lambda_{n}^{\ell_n^{'[2]}}.
\end{equation}
Note that \eqref{pf:Lambda} is still true when $n=0$.
Then, for $n\geq 1$, we have by definition
\begin{equation}\label{pf:Lambda1}
\Lambda_{n}^{\ell_{n}^{[2]}}
=
\sum_{\ell_{n-1}^{[2]}\in (N)^2} \Lambda_{n-1}^{\ell_{n-1}^{[2]}}\
\lambda_{n-1}^b(A_{n-1}^{\ell_n^{[2]}},\ell_{n-1}^{[2]}).
\end{equation}
This decomposition will appear several times in the sequel for it is a keystone to study the behavior of the coalescent tree-based measures.

\begin{proposition}\label{prop:Gamma1}
  Assume $\mathcal{A}$\ref{A1}. For any coalescence indicator $b$, we have
  $$
  \sup_{N>1}
  \mathbf{E}
  \left[
    \Gamma_{n,N}^b(1)^2
  \right]
  <+\infty.
  $$
  In particular, the sequence $(\Gamma_{n,N}^b(1); N\geq 1)$ is uniformly tight.
  \end{proposition}
\begin{proof}
  We give a proof by induction. The verification for step $0$ is trivial as
  $\Gamma_{0,N}^b(1) = 1$.
  For $n\geq 1$, we suppose that
  $$
    \sup_{N>1}
    \mathbf{E}
    \left[
      \Gamma_{n-1,N}^b(1)^2
    \right]
    <+\infty.
   $$
 As defined in Section \ref{sec-old:IPS}, the IPS  associated with ASMC is a Markov chain $(\mathbf{X_{n}})_{n\geq 0}$ with genealogy $(\mathbf{A_{n}})_{n\geq 0}$ tracking the indice of the parent of each particle at each level. More precisely, $A_{p-1}^i =j$ means that the parent of the particle $X_{p}^i$ is $X_{p-1}^j$. Accordingly, the filtration 
$(\mathscr{G}_n^N)_{n\geq 0}$ with the 
genealogy of the IPS is defined by
$$
\mathscr{G}_0^N := \sigma(\mathbf{X_{0}})
$$
and, for $n\geq 1$, 
$$
\mathscr{G}_n^N :=
\sigma(\mathbf{A_{0}},\dots,\mathbf{A_{n-1}},\mathbf{X_{0}},\dots,\mathbf{X_{n}}).
$$
 By combining \eqref{pf:Lambda2} and \eqref{pf:Lambda1}, and taking into account that
 $$
 \gamma_n^N(1)=\prod_{p=0}^{n-1} \eta_p^N(G_{p,N})=\gamma_{n-1}^N(1)\eta_{n-1}^N(G_{n-1,N})=\gamma_{n-1}^N(1)m(\mathbf{X_{n-1}})(G_{n-1,N})
$$
 is $\mathscr{G}_{n-1}^N$-measurable, we have
  \begin{multline}\label{pf:Lambda-decomp}
  \mathbf{E}
  \left[
  \Gamma_{n,N}^b(1)^2
  \Given{\mathscr{G}_{n-1}^N}\right]
  =
  \gamma_{n-1}^N(1)^4
  \left(\frac{N^{n-1}}{(N-1)^{n+1}}\right)^2
  \sum_{(\ell_{n-1}^{[2]},\ell_{n-1}^{'[2]}) \in ((N)^2)^{\times 2}}
  \Lambda_{n-1}^{\ell_{n-1}^{[2]}}
  \Lambda_{n-1}^{\ell_{n-1}^{'[2]}}
  \\
  \sum_{(\ell_n^{[2]},\ell_n^{'[2]}) \in ((N)^2)^{\times 2}}\ m(\mathbf{X_{n-1}})(G_{n-1,N})^4\ \mathbf{E}\left[
    \lambda_{n-1}^b(A_{n-1}^{\ell_n^{[2]}},\ell_{n-1}^{[2]})
    \lambda_{n-1}^b(A_{n-1}^{\ell_n^{'[2]}},\ell_{n-1}^{'[2]})
    \;\Bigg\vert\;\mathscr{G}_{n-1}^N
  \right].
  \end{multline}
  For the notation concerning the indices in the IPS, we
    use 
    $$
    [N]_p^q :=\left\{(i_1,\dots,i_q)\in [N]^q: \mathrm{Card}\{i_1,\dots,i_q\} =p\right\}.
    $$
    In particular, we denote $(N)^q := [N]^q_q$.
    We also write
        $$
    \left((N)^2\right)^{\times q}:=\underbrace{(N)^2 \times (N)^2\times \cdots \times (N)^2}_{q \text{
    times}}.
    $$
    With a slight abuse of notation, we admit that
    $$
    \left((i,j),(k,l)\right) = (i,j,k,l).
    $$
    With this notation, for $N\geq 4$, we have the decomposition
$$((N)^2)^{\times 2} 
  =
  \left(
  ((N)^2)^{\times 2} 
  \cap [N]_2^4
  \right)
  \cup
  \left(
  ((N)^2)^{\times 2} 
  \cap [N]_3^4
  \right)
  \cup
  (N)^4.$$
The idea of the proof consists in analyzing (\ref{pf:Lambda-decomp}) with respect to the three terms that appear in the right-hand side of the latter. Recall from (\ref{eq:selection}) that, given $\mathbf{X_{n-1}}$, we make an
    independent multinomial selection of the parent of each particle at step $n$ according to the discrete probability measure
    $$
    S_{n-1,N}(\mathbf{X_{n-1}},\cdot)
    =
    \sum_{k=1}^{N}
    \frac{G_{n-1,N}(X_{n-1}^k)}{\sum_{j=1}^N
    G_{n-1,N}(X_{n-1}^j)}
    \delta_k=\sum_{k=1}^{N}
    \frac{G_{n-1,N}(X_{n-1}^k)}{N\ m(\mathbf{X_{n-1}})(G_{n-1,N})}
    \delta_k,
   $$
   with, for all $k\in[N]$,
   $$0< \frac{G_{n-1,N}(X_{n-1}^k)}{N\ m(\mathbf{X_{n-1}})(G_{n-1,N})}\leq  \frac{\norm{G_{n-1,\cdot}}_{\infty}}{N\ m(\mathbf{X_{n-1}})(G_{n-1,N})}.$$
  We also recall that
   $$
   \lambda_{n-1}^b(A_{n-1}^{\ell_n^{[2]}},\ell_{n-1}^{[2]})
   =
   \mathbf{1}_{\{b_{n-1} = 0\}} \mathbf{1}_{\{A_{n-1}^{\ell_n^1}  =
   \ell_{n-1}^1\neq A_{n-1}^{\ell_n^2}=\ell_{n-1}^2\}}
  +
  \mathbf{1}_{\{b_{n-1} = 1\}} \mathbf{1}_{\{A_{n-1}^{\ell_n^1} =
      \ell_{n-1}^1=A_{n-1}^{\ell_n^2}
  \neq\ell_{n-1}^2\}}.
   $$
  \begin{itemize}
    \item Case 1: 
      $
      (\ell_n^{[2]},\ell_n^{'[2]})  \in
      ((N)^2)^{\times 2} 
      \cap [N]_2^4.
      $
    
      In this case, there are only two distinct random variables among 
      $A_{n-1}^{\ell_n^1}$, $A_{n-1}^{\ell_n^2}$, $A_{n-1}^{\ell_n^{'1}}$,
      $A_{n-1}^{\ell_n^{'2}}$.
      Recall that $\ell_n^1 \neq \ell_n^2$ by construction.
      Let us first suppose that
      $$
      \ell_n^1 = \ell_n^{'1}
      \quad\text{and}\quad
      \ell_n^2 = \ell_n^{'2}.
      $$
      Thus, we deduce that
      $$
      \begin{aligned}
        &
      \mathbf{E}\left[
        \lambda_{n-1}^b(A_{n-1}^{\ell_n^{[2]}},\ell_{n-1}^{[2]})
        \lambda_{n-1}^b(A_{n-1}^{\ell_n^{'[2]}},\ell_{n-1}^{'[2]})
        \;\Bigg\vert\;\mathscr{G}_{n-1}^N
      \right]
      \\ &\leq
      \ \mathbf{E}\left[
        \lambda_{n-1}^b(A_{n-1}^{\ell_n^{[2]}},\ell_{n-1}^{[2]})
        \;\Bigg\vert\;\mathscr{G}_{n-1}^N
      \right]
      \\&\ \ \ \ =
      \mathbf{1}_{\{b_{n-1} = 0\}}
      \mathbf{P}\left(
      A_{n-1}^{\ell_n^1} =
      \ell_{n-1}^1,A_{n-1}^{\ell_n^2}
      =\ell_{n-1}^2
        \;\bigg\vert\;\mathscr{G}_{n-1}^N
      \right)\\
            &\ \ \ \ \ \ \ \ +
      \mathbf{1}_{\{b_{n-1} = 1\}}
      \mathbf{P}\left(
      A_{n-1}^{\ell_n^1} =
      \ell_{n-1}^1=A_{n-1}^{\ell_n^2}
        \;\bigg\vert\;\mathscr{G}_{n-1}^N
      \right)
      \\
      &\leq
      \left(
        \frac{\norm{G_{n-1,\cdot}}_{\infty}}{N\ m(\mathbf{X_{n-1}})(G_{n-1,N})}
      \right)^2.
      \end{aligned}
      $$
      The analysis for the case where
      $$
      \ell_n^1 = \ell_n^{'2}
      \quad\text{and}\quad
      \ell_n^2 = \ell_n^{'1}
      $$
      is analogue.
      Hence, we conclude that
      \begin{equation}\label{pf:case1}
      \mathbf{E}\left[
        \lambda_{n-1}^b(A_{n-1}^{\ell_n^{[2]}},\ell_{n-1}^{[2]})
        \lambda_{n-1}^b(A_{n-1}^{\ell_n^{'[2]}},\ell_{n-1}^{'[2]})
        \;\Bigg\vert\;\mathscr{G}_{n-1}^N
      \right]
      \leq
      \left(
        \frac{\norm{G_{n-1,\cdot}}_{\infty}}{N\ m(\mathbf{X_{n-1}})(G_{n-1,N})}
      \right)^2.
      \end{equation}
      Meanwhile, we notice that
      $$
      \#\left(
      ((N)^2)^{\times 2} 
      \cap [N]_2^4
      \right)
      =2N(N-1).
      $$
      Putting all things together yields
      \begin{multline*}
      \sum_{
      (\ell_n^{[2]},\ell_n^{'[2]})  \in
      ((N)^2)^{\times 2} 
      \cap [N]_2^4
      }
      m(\mathbf{X_{n-1}})(G_{n-1,N})^4\
      \mathbf{E}\left[
        \lambda_{n-1}^b(A_{n-1}^{\ell_n^{[2]}},\ell_{n-1}^{[2]})
        \lambda_{n-1}^b(A_{n-1}^{\ell_n^{'[2]}},\ell_{n-1}^{'[2]})
        \;\Bigg\vert\;\mathscr{G}_{n-1}^N
      \right]
      \\
      \leq
      \frac{2(N-1)}{N}  \norm{G_{n-1,\cdot}}_{\infty}^4.
      \end{multline*}
    \item Case 2: 
      $
      (\ell_n^{[2]},\ell_n^{'[2]})  \in
      ((N)^2)^{\times 2} 
      \cap [N]_3^4.
      $

      As noticed in the previous case, the
      number of different indices within
      $
      (\ell_n^{[2]},\ell_n^{'[2]})
      $
      is the only thing that matters for the upper-bound in \eqref{pf:case1}.
      Accordingly, the same reasoning gives this time
      $$
      \mathbf{E}\left[
        \lambda_{n-1}^b(A_{n-1}^{\ell_n^{[2]}},\ell_{n-1}^{[2]})
        \lambda_{n-1}^b(A_{n-1}^{\ell_n^{'[2]}},\ell_{n-1}^{'[2]})
        \;\Bigg\vert\;\mathscr{G}_{n-1}^N
      \right]
      \leq
      \left(
        \frac{\norm{G_{n-1,\cdot}}_{\infty}}{N\ m(\mathbf{X_{n-1}})(G_{n-1,N})}
      \right)^3.
      $$
      Since the total number of choices is 
      $$
      \#\left(
      ((N)^2)^{\times 2} 
      \cap [N]_3^4
      \right)
      =4N(N-1)(N-2),
      $$
      it comes
      $$
      \begin{aligned}
      \sum_{
      (\ell_n^{[2]},\ell_n^{'[2]})  \in
      ((N)^2)^{\times 2} 
      \cap [N]_3^4
      }
      m(\mathbf{X_{n-1}})(G_{n-1,N})^4\
      \mathbf{E}\left[
        \lambda_{n-1}^b(A_{n-1}^{\ell_n^{[2]}},\ell_{n-1}^{[2]})
        \lambda_{n-1}^b(A_{n-1}^{\ell_n^{'[2]}},\ell_{n-1}^{'[2]})
        \;\Bigg\vert\;\mathscr{G}_{n-1}^N
      \right]
      \\
      \leq
      \frac{4(N-1)(N-2)}{N^2}  \norm{G_{n-1,\cdot}}_{\infty}^4.
        \end{aligned}
        $$
    \item Case 3: 
      $
      (\ell_n^{[2]},\ell_n^{'[2]})  \in
      (N)^4.
      $
      
 This time, we get
      $$
      \mathbf{E}\left[
        \lambda_{n-1}^b(A_{n-1}^{\ell_n^{[2]}},\ell_{n-1}^{[2]})
        \lambda_{n-1}^b(A_{n-1}^{\ell_n^{'[2]}},\ell_{n-1}^{'[2]})
        \;\Bigg\vert\;\mathscr{G}_{n-1}^N
      \right]
      \leq
      \left(
        \frac{\norm{G_{n-1,\cdot}}_{\infty}}{N\ m(\mathbf{X_{n-1}})(G_{n-1,N})}
      \right)^4,
      $$
      and
      $$
      \#\left(
        (N)^4
      \right)
      =N(N-1)(N-2)(N-3),
      $$
 so that    
 $$
 \begin{aligned}
 \sum_{
 (\ell_n^{[2]},\ell_n^{'[2]})  \in
 (N)^4
 }
 m(\mathbf{X_{n-1}})(G_{n-1,N})^4\
 \mathbf{E}\left[
   \lambda_{n-1}^b(A_{n-1}^{\ell_n^{[2]}},\ell_{n-1}^{[2]})
   \lambda_{n-1}^b(A_{n-1}^{\ell_n^{'[2]}},\ell_{n-1}^{'[2]})
   \;\Bigg\vert\;\mathscr{G}_{n-1}^N
 \right]
 \\
 \leq
 \frac{(N-1)(N-2)(N-3)}{N^3}  \norm{G_{n-1,\cdot}}_{\infty}^4.
 \end{aligned}
 $$
  \end{itemize}
  As a consequence, since
  $$2+\frac{4(N-2)}{N}+\frac{(N-2)(N-3)}{N^2}\leq 7,$$
  an upper-bound for (\ref{pf:Lambda-decomp}) is
  $$
  \begin{aligned}
    &
  \mathbf{E}
  \left[
  \Gamma_{n,N}^b(1)^2
  \Given{\mathscr{G}_{n-1}^N}\right]
  \\\leq&
 7  \left(\frac{N^{n-1}}{(N-1)^{n+1}}\right)^2\frac{N-1}{N}\norm{G_{n-1,\cdot}}_{\infty}^4
 \gamma_{n-1}^N(1)^4 \sum_{(\ell_{n-1}^{[2]},\ell_{n-1}^{'[2]}) \in ((N)^2)^{\times 2}}
  \Lambda_{n-1}^{\ell_{n-1}^{[2]}}
  \Lambda_{n-1}^{\ell_{n-1}^{'[2]}}.
  \end{aligned}
  $$
Replacing $n$ with $(n-1)$ in $(\ref{pf:Lambda2})$ allows us to conclude that
$$ \mathbf{E}
  \left[
  \Gamma_{n,N}^b(1)^2
  \Given{\mathscr{G}_{n-1}^N}\right]
  \leq
 7
  \frac{N}{N-1}\norm{G_{n-1,\cdot}}_{\infty}^4\Gamma_{n-1,N}^b(1)^2.$$  
Finally, by applying the induction hypothesis, we have 
$$
\sup_{N\geq 4}
\mathbf{E}
\left[
  \Gamma_{n,N}^b(1)^2
\right]
=
\sup_{N\geq4}
\mathbf{E}
\left[
\mathbf{E}
\left[
\Gamma_{n,N}^b(1)^2
\Given{\mathscr{G}_{n-1}^N}\right]
\right]
\leq 
\frac{28}{3}\norm{G_{n-1,\cdot}}_{\infty}^4
\sup_{N\geq4}
\mathbf{E}
\left[
  \Gamma_{n-1,N}^b(1)^2
\right]
<+\infty,
$$
which ends the proof of Proposition \ref{prop:Gamma1}.
\end{proof}

\begin{lemma}
  \label{lm:cvg-proba-Gamma}
  Under $\mathcal{A}$\ref{A1}, for any test functions $f,g\in \mathcal{B}_b(E_n)$, we
  have, for all $n\geq 1$, 
  \begin{equation}
    \label{pf:cond-Gamma}
  \mathbf{E}
  \left[
    \Gamma_{n,N}^b(f\otimes g) 
  \;\bigg\vert\;
  \mathscr{G}_{n-1}^N
  \right]
  =
  \Gamma_{n-1,N}^b Q_{n,N}^{\otimes 2}C_{b_n}(f\otimes g),
  \end{equation}
  as well as
  $$
  \Gamma_{n,N}^b(f\otimes g) 
  -
  \Gamma_{n-1,N}^b Q_{n,N}^{\otimes 2}C_{b_n}(f\otimes g)
  =\mathscr{O}_{\mathbf{p}}\left(\frac{1}{\sqrt{N}}\right).
  $$
\end{lemma}
\begin{proof}
  First, by exploiting the notation defined in \eqref{eq:LambdaElln}, we have
  $$
  \Gamma_{n,N}^b(f\otimes g):=
  \gamma_n^N(1)^2
  \frac{N^{n-1}}{(N-1)^{n+1}}
  \sum_{\ell_n^{[2]}\in (N)^2}
  \Lambda_{n}^{\ell_n^{[2]}} 
  C_{b_n}(f\otimes g)(X_n^{\ell_n^{[2]}}),
  $$
  and \eqref{pf:cond-Gamma} is then a direct consequence of Proposition \ref{lm:LackOfBias} since for any $\ell_n^{[2]}\in (N)^2$
  $$ \mathbf{E}\left[
  \gamma_n^N(1)^2\ 
  \frac{N^{n-1}}{(N-1)^{n+1}}
 \Lambda_{n}^{\ell_n^{[2]}} 
  C_{b_n}(f\otimes g)(X_n^{\ell_n^{[2]}})
  \;\Bigg\vert\;
  \mathscr{G}_{n-1}^N
  \right]
  =
  \frac{1}{N(N-1)}
  \Gamma_{n-1,N}^b Q_{n,N}^{\otimes 2}C_{b_n}(f\otimes g),$$
  where the right-hand side does not depend on $\ell_n^{[2]}$.
  Second, thanks to Chebyshev's inequality, it suffices to verify
  that
  \begin{equation*}
  \mathrm{Var}
  \left[
  \Gamma_{n,N}^b(f\otimes g) 
  -
  \Gamma_{n-1,N}^b Q_{n,N}^{\otimes 2}C_{b_n}(f\otimes g)
  \right]
  =\mathscr{O}\left(\frac{1}{N}\right).
  \end{equation*}
  For this, by \eqref{pf:cond-Gamma}, we just have to show that 
  $$
  \mathbf{E}
  \left[
  \Gamma_{n,N}^b(f\otimes g)^2
  -
  (\Gamma_{n-1,N}^b Q_{n,N}^{\otimes 2}C_{b_n}(f\otimes g))^2
  \right]
  =\mathscr{O}\left(\frac{1}{N}\right).
  $$
  Then, recall that, by definition,
\begin{equation*}
  \begin{split}
  &\Gamma_{n,N}^b(f\otimes g)^2
  \\
  =
  &
  \gamma_n^N(1)^4
  \left(
  \frac{N^{n-1}}{(N-1)^{n+1}}
  \right)^2
  \sum_{(\ell_n^{[2]},\ell_n^{'[2]})\in ((N)^2)^{\times 2}}
  \Lambda_{n}^{\ell_n^{[2]}} 
  \Lambda_{n}^{\ell_n^{'[2]}} 
  \left(
  C_{b_n}(f\otimes g)
\right)^{\otimes 2}(X_n^{\ell_n^{[2]}},X_n^{\ell_n^{'[2]}})
  \\
  =
  &
  \underbrace{
  \gamma_n^N(1)^4
  \left(
  \frac{N^{n-1}}{(N-1)^{n+1}}
  \right)^2
  \sum_{(\ell_n^{[2]},\ell_n^{'[2]})\in (N)^4}
  \Lambda_{n}^{\ell_n^{[2]}} 
  \Lambda_{n}^{\ell_n^{'[2]}} 
  \left(
  C_{b_n}(f\otimes g)
\right)^{\otimes 2}(X_n^{\ell_n^{[2]}},X_n^{\ell_n^{'[2]}})
  }_{R_1(N)}
  \\
  &
  +
  \underbrace{
  \gamma_n^N(1)^4
    \left(
  \frac{N^{n-1}}{(N-1)^{n+1}}
  \right)^2
  \sum_{(\ell_n^{[2]},\ell_n^{'[2]})\in ((N)^2)^{\times 2} \backslash (N)^4}
  \Lambda_{n}^{\ell_n^{[2]}} 
  \Lambda_{n}^{\ell_n^{'[2]}} 
  \left(
  C_{b_n}(f\otimes g)
\right)^{\otimes 2}(X_n^{\ell_n^{[2]}},X_n^{\ell_n^{'[2]}})
  }_{R_2(N)}.
  \end{split}
\end{equation*}

\begin{itemize}
  \item
    For $R_1(N)$, our goal is to establish that
    \begin{equation*}\label{pf:R1N}
    \mathbf{E}
    \left[
      R_1(N)
      -
      (\Gamma_{n-1,N}^b Q_{n,N}^{\otimes 2}C_{b_n}(f\otimes g))^2
    \right]
    =\mathscr{O}\left(\frac{1}{N}\right).
    \end{equation*}
    In fact, for any $(\ell_{n}^{[2]},\ell_{n}^{'[2]})\in (N)^4$, 
    $$
    (A_n^{\ell_n^1}, X_n^{\ell_n^1},
    A_n^{\ell_n^2},X_{n}^{\ell_n^2})
    \qquad
    \text{and}
    \qquad
    (A_n^{\ell_n^{'1}},X_n^{\ell_n^{'1}},
    A_n^{\ell_n^{'2}},X_n^{\ell_n^{'2}})
    $$
    are conditionally
    independent given $\mathscr{G}_{n-1}^N$ by construction of the IPS.
    Hence, by applying Proposition \ref{lm:LackOfBias} respectively for $\ell_n^{[2]}$ and
     for $\ell_n^{'[2]}$, we have
    \begin{multline*}
    \mathbf{E}
    \left[
    \gamma_n^N(1)^4
    \left(
    \frac{N^{n-1}}{(N-1)^{n+1}}
    \right)^2
    \Lambda_{n}^{\ell_n^{[2]}} 
    \Lambda_{n}^{\ell_n^{'[2]}} 
    C_{b_n}(f\otimes g)(X_n^{\ell_n^{[2]}})
    C_{b_n}(f\otimes g)(X_n^{\ell_n^{'[2]}})
    \;\Bigg\vert\;
    \mathscr{G}_{n-1}^N
    \right]
    \\
    =
    \frac{1}{N^2(N-1)^2}
   ( \Gamma_{n-1}^N Q_{n,N}^{\otimes 2}C_{b_n}(f\otimes g))^2.
    \end{multline*}
    Then, since 
    $$\#\left(
        (N)^4
      \right)
      =N(N-1)(N-2)(N-3),$$
    we deduce that
    \begin{multline*}
      \mathbf{E}\left[
    R_1(N)
    -
    (\Gamma_{n-1,N}^b Q_{n,N}^{\otimes 2}C_{b_n}(f\otimes g))^2
    \;\bigg\vert\;
    \mathscr{G}_{n-1}^N
  \right]
    \\
    =
    \left(
    \frac{N(N-1)(N-2)(N-3)}{N^2(N-1)^2}
    -1
    \right)
    (\Gamma_{n-1}^N Q_{n,N}^{\otimes 2}C_{b_n}(f\otimes g))^2,
    \end{multline*}
    and
    \begin{multline*} 
 \mathbf{E}\left[
    R_1(N)
    -
    (\Gamma_{n-1,N}^b Q_{n,N}^{\otimes 2}C_{b_n}(f\otimes g))^2 
    \right]
       \\
    =
    \left(
    \frac{N(N-1)(N-2)(N-3)}{N^2(N-1)^2}
    -1
    \right)
    \mathbf{E}\left[(\Gamma_{n-1}^N Q_{n,N}^{\otimes 2}C_{b_n}(f\otimes g))^2
    \right]
    =\mathscr{O}\left(\frac{1}{N}\right),
    \end{multline*}
    where the final equality is due to Proposition \ref{prop:Gamma1}, taking into account that $f$ and $g$ are bounded, and so is $G_{n-1,N}$ uniformly with respect to $N$ by $\mathcal{A}$\ref{A1}.
  \item
    For $R_2(N)$, the nonnegativity of $\Lambda_n^{\ell_n^{[2]}}$ implies 
    $$
    \mathbf{E}[R_2(N)]
    \leq
    \mathbf{E}
    \left[
    \gamma_n^N(1)^4
    \left(
    \frac{N^{n-1}}{(N-1)^{n+1}}
    \right)^2
    \sum_{(\ell_n^{[2]},\ell_n^{'[2]})\in ((N)^2)^{\times 2} \backslash (N)^4}
    \Lambda_{n}^{\ell_n^{[2]}}
    \Lambda_{n}^{\ell_n^{'[2]}}
    \right]
    \norm{f}_{\infty}^2
    \norm{g}_{\infty}^2.
    $$
So the proof will be finished once we have shown that
  $$
  \mathbf{E}
  \left[
  \gamma_n^N(1)^4
  \left(
  \frac{N^{n-1}}{(N-1)^{n+1}}
  \right)^2
  \sum_{(\ell_n^{[2]},\ell_n^{'[2]})\in ((N)^2)^{\times 2} \backslash (N)^4}
  \Lambda_{n}^{\ell_n^{[2]}}
  \Lambda_{n}^{\ell_n^{'[2]}}
  \right]
  =\mathscr{O}\left(\frac{1}{N}\right).
  $$
Once again, we proceed by induction. At step 0, we have
  $$
  \frac{1}{N^2(N-1)^2}
  \sum_{(\ell_0^{[2]},\ell_0^{'[2]})\in ((N)^2)^{\times 2} \backslash (N)^4}
  1
  =
  1
  -
  \frac{N(N-1)(N-2)(N-3)}{N^2(N-1)^2}
  =\mathscr{O}\left(\frac{1}{N}\right).
  $$
  For step $n\geq 1$, we suppose that
  $$
  \mathbf{E}
  \left[
    \gamma_{n-1}^N(1)^4
  \left(
  \frac{N^{n-2}}{(N-1)^{n}}
  \right)^2
  \sum_{(\ell_{n-1}^{[2]},\ell_{n-1}^{'[2]})\in ((N)^2)^{\times 2} \backslash (N)^4}
  \Lambda_{n-1}^{\ell_{n-1}^{[2]}}
  \Lambda_{n-1}^{\ell_{n-1}^{'[2]}}
  \right]
  =\mathscr{O}\left(\frac{1}{N}\right).
  $$
  The adaptation of (\ref{pf:Lambda-decomp}) to the present context gives
  \begin{multline*}
  \mathbf{E}
  \left[
    \gamma_{n}^N(1)^4
  \left(
  \frac{N^{n-1}}{(N-1)^{n+1}}
  \right)^2
  \sum_{(\ell_{n}^{[2]},\ell_{n}^{'[2]})\in ((N)^2)^{\times 2} \backslash (N)^4}
  \Lambda_{n}^{\ell_{n}^{[2]}}
  \Lambda_{n}^{\ell_{n}^{'[2]}}
  \;\Bigg\vert\;
  \mathscr{G}_{n-1}^N
  \right]
  \\
  =
  \gamma_{n-1}^N(1)^4
  \left(
  \frac{N^{n-1}}{(N-1)^{n+1}}
  \right)^2
  \sum_{(\ell_{n-1}^{[2]},\ell_{n-1}^{'[2]}) \in ((N)^2)^{\times 2}}
  \Lambda_{n-1}^{\ell_{n-1}^{[2]}}
  \Lambda_{n-1}^{\ell_{n-1}^{'[2]}}
  \\
  \sum_{(\ell_{n}^{[2]},\ell_{n}^{'[2]})\in ((N)^2)^{\times 2} \backslash (N)^4}
  m(\mathbf{X_{n-1}})(G_{n-1,N})^4\ 
  \\
  \mathbf{E}\left[
    \lambda_{n-1}^b(A_{n-1}^{\ell_n^{[2]}},\ell_{n-1}^{[2]})
    \lambda_{n-1}^b(A_{n-1}^{\ell_n^{'[2]}},\ell_{n-1}^{'[2]})
    \;\Bigg\vert\;\mathscr{G}_{n-1}^N
  \right].
  \end{multline*}
  Now, for $N\geq4$, it is clear that
  $$
  ((N)^2)^{\times 2} \backslash (N)^4
  =
  \left(
  ((N)^2)^{\times 2} 
  \cap [N]_2^4
  \right)
  \cup
  \left(
  ((N)^2)^{\times 2} 
  \cap [N]_3^4
  \right).
  $$
  \begin{itemize}
    \item Case 1: 
      $
      (\ell_n^{[2]},\ell_n^{'[2]})  \in
      ((N)^2)^{\times 2} 
      \cap [N]_2^4.
      $

      By definition of $S_{n-1,N}$ in \eqref{eq:selection}, 
      \begin{equation*}
        \begin{split}
        &
        \mathbf{E}\left[
          \lambda_{n-1}^b(A_{n-1}^{\ell_n^{[2]}},\ell_{n-1}^{[2]})
          \lambda_{n-1}^b(A_{n-1}^{\ell_n^{'[2]}},\ell_{n-1}^{'[2]})
          \;\Bigg\vert\;\mathscr{G}_{n-1}^N
        \right]
        \\
        \leq
        &
        \left(
          \frac{\norm{G_{n-1,\cdot}}_{\infty}}{N\ m(\mathbf{X_{n-1}})(G_{n-1,N})}
        \right)^2
        \left(
        \mathbf{1}_{\{b_{n-1}=1,\ell_{n-1}^1 = \ell_{n-1}^{'1}\}}
        +
        \mathbf{1}_{\{b_{n-1}=0,\ell_{n-1}^1 = \ell_{n-1}^{'1},\ell_{n-1}^2 = \ell_{n-1}^{'2}\}}
        \right)
        \\
        \leq
        &
        \left(
          \frac{\norm{G_{n-1,\cdot}}_{\infty}}{N\ m(\mathbf{X_{n-1}})(G_{n-1,N})}
        \right)^2
          \mathbf{1}_{\left\{\#\{\ell_{n-1}^{1},\ell_{n-1}^{'1},\ell_{n-1}^{2},\ell_{n-1}^{'2}\}
          <4\right\}},
        \end{split}
      \end{equation*}
      and since
      $$
      \#\left(
      ((N)^2)^{\times 2} 
      \cap [N]_2^4
      \right)
      =2N(N-1),
      $$
      it comes
      \begin{multline*}
      \sum_{
      (\ell_n^{[2]},\ell_n^{'[2]})  \in
      ((N)^2)^{\times 2} 
      \cap [N]_2^4
      }
      m(\mathbf{X_{n-1}})(G_{n-1,N})^4
      \\
      \mathbf{E}\left[
        \lambda_{n-1}^b(A_{n-1}^{\ell_n^{[2]}},\ell_{n-1}^{[2]})
        \lambda_{n-1}^b(A_{n-1}^{\ell_n^{'[2]}},\ell_{n-1}^{'[2]})
        \;\Bigg\vert\;\mathscr{G}_{n-1}^N
      \right]
      \\
      \leq
      \frac{2N(N-1)}{N^2}  \norm{G_{n-1,\cdot}}_{\infty}^4
          \mathbf{1}_{\left\{\#\{\ell_{n-1}^{1},\ell_{n-1}^{'1},\ell_{n-1}^{2},\ell_{n-1}^{'2}\}
          <4\right\}}.
      \end{multline*}
    \item Case 2: 
      $
      (\ell_n^{[2]},\ell_n^{'[2]})  \in
      ((N)^2)^{\times 2} 
      \cap [N]_3^4.
      $

      First, we suppose that $\ell_n^1 = \ell_n^{'1}$. As for the previous
      case, we have
      \begin{equation*}
        \begin{split}
        &
        \mathbf{E}\left[
          \lambda_{n-1}^b(A_{n-1}^{\ell_n^{[2]}},\ell_{n-1}^{[2]})
          \lambda_{n-1}^b(A_{n-1}^{\ell_n^{'[2]}},\ell_{n-1}^{'[2]})
          \;\Bigg\vert\;\mathscr{G}_{n-1}^N
        \right]
        \\
        \leq
        &
        \left(
          \frac{\norm{G_{n-1,\cdot}}_{\infty}}{N\ m(\mathbf{X_{n-1}})(G_{n-1,N})}
        \right)^3
        \left(
        \mathbf{1}_{\{b_{n-1}=1,\ell_{n-1}^1 = \ell_{n-1}^{'1}\}}
        +
        \mathbf{1}_{\{b_{n-1}=0,\ell_{n-1}^1 = \ell_{n-1}^{'1}\}}
        \right)
        \\
        \leq
        &
        \left(
          \frac{\norm{G_{n-1,\cdot}}_{\infty}}{N\ m(\mathbf{X_{n-1}})(G_{n-1,N})}
        \right)^3
          \mathbf{1}_{\left\{\#\{\ell_{n-1}^{1},\ell_{n-1}^{'1},\ell_{n-1}^{2},\ell_{n-1}^{'2}\}
          <4\right\}}.
        \end{split}
      \end{equation*}
      By the same reasoning, for $\ell_n^1 = \ell_n^{'2}$, 
      $\ell_n^2 = \ell_n^{'1}$ and $\ell_n^2 = \ell_n^{'2}$, we also have 
      \begin{equation*}
        \begin{split}
        &
        \mathbf{E}\left[
          \lambda_{n-1}^b(A_{n-1}^{\ell_n^{[2]}},\ell_{n-1}^{[2]})
          \lambda_{n-1}^b(A_{n-1}^{\ell_n^{'[2]}},\ell_{n-1}^{'[2]})
          \;\Bigg\vert\;\mathscr{G}_{n-1}^N
        \right]
        \\
        \leq
        &
        \left(
          \frac{\norm{G_{n-1,\cdot}}_{\infty}}{N\ m(\mathbf{X_{n-1}})(G_{n-1,N})}
        \right)^3
          \mathbf{1}_{\left\{\#\{\ell_{n-1}^{1},\ell_{n-1}^{'1},\ell_{n-1}^{2},\ell_{n-1}^{'2}\}
          <4\right\}}.
        \end{split}
      \end{equation*}
      In addition, since
      $$
      \#\left(
      ((N)^2)^{\times 2} 
      \cap [N]_3^4
      \right)
      =4N(N-1)(N-2),
      $$
      we get this time
      \begin{multline*}
      \sum_{
      (\ell_n^{[2]},\ell_n^{'[2]})  \in
      ((N)^2)^{\times 2} 
      \cap [N]_3^4
      }
      m(\mathbf{X_{n-1}})(G_{n-1,N})^4
      \\
      \mathbf{E}\left[
        \lambda_{n-1}^b(A_{n-1}^{\ell_n^{[2]}},\ell_{n-1}^{[2]})
        \lambda_{n-1}^b(A_{n-1}^{\ell_n^{'[2]}},\ell_{n-1}^{'[2]})
        \;\Bigg\vert\;\mathscr{G}_{n-1}^N
      \right]
      \\
      \leq
      \frac{4N(N-1)(N-2)}{N^3}\norm{G_{n-1,\cdot}}_{\infty}^4
        \mathbf{1}_{\left\{\#\{\ell_{n-1}^{1},\ell_{n-1}^{'1},\ell_{n-1}^{2},\ell_{n-1}^{'2}\}
        <4\right\}}.
      \end{multline*}
  \end{itemize}
  By gathering both cases, we have
  \begin{equation*}
  \begin{split}
    &
  \mathbf{E}
  \left[
    \gamma_{n}^N(1)^4
  \left(
  \frac{N^{n-1}}{(N-1)^{n+1}}
  \right)^2
  \sum_{(\ell_{n}^{[2]},\ell_{n}^{'[2]})\in ((N)^2)^{\times 2} \backslash (N)^4}
  \Lambda_{n}^{\ell_{n}^{[2]}}
  \Lambda_{n}^{\ell_{n}^{'[2]}}
  \;\Bigg\vert\;
  \mathscr{G}_{n-1}^N
  \right]
  \\
  \leq
    &
    \gamma_{n-1}^N(1)^4
  \left(
  \frac{N^{n-2}}{(N-1)^{n}}
  \right)^2
        \frac{N^2}{(N-1)^2}
  \left(
  \frac{2N(N-1)}{N^2}  +
  \frac{4N(N-1)(N-2)}{N^3}
  \right) 
 \norm{G_{n-1,\cdot}}_{\infty}^4
  \\
    &
     \sum_{(\ell_{n-1}^{[2]},\ell_{n-1}^{'[2]}) \in ((N)^2)^{\times 2}}
  \Lambda_{n-1}^{\ell_{n-1}^{[2]}}
  \Lambda_{n-1}^{\ell_{n-1}^{'[2]}}
  \mathbf{1}_{\left\{\#\{\ell_{n-1}^{1},\ell_{n-1}^{'1},\ell_{n-1}^{2},\ell_{n-1}^{'2}\}
  <4\right\}}
  \\
  \leq
    &
  6  \frac{N^2}{(N-1)^2}
  \norm{G_{n-1,\cdot}}_{\infty}^4
    \gamma_{n-1}^N(1)^4
  \left(
  \frac{N^{n-2}}{(N-1)^{n}}
  \right)^2
  \sum_{(\ell_{n-1}^{[2]},\ell_{n-1}^{'[2]}) \in ((N)^2)^{\times 2}\backslash (N)^4}
  \Lambda_{n-1}^{\ell_{n-1}^{[2]}}
  \Lambda_{n-1}^{\ell_{n-1}^{'[2]}}.
  \end{split}
  \end{equation*}
  The conclusion follows from the induction hypothesis by taking the expectation on both sides. 
 \end{itemize}
 This ends the proof of Lemma \ref{lm:cvg-proba-Gamma}.
\end{proof}

\begin{lemma}\label{lm:cvg-Gamma-Q}
  Assume $\mathcal{A}$\ref{A1}-$\mathcal{A}$\ref{A2}, then
  for any test functions $f,g\in\mathcal{B}_b(E_n)$ and for all $n\geq
  1$, 
  $$
    \Gamma_{n-1,N}^b Q_{n,N}^{\otimes 2}(f\otimes g)
  -
  \Gamma_{n-1,N}^b Q_{n}^{\otimes 2}(f\otimes g)
  =\mathscr{O}_{\mathbf{p}}\left(\frac{1}{\sqrt{N}}\right).
  $$
\end{lemma}
\begin{proof}
The verification 
shares some resemblance with the convergence of $P_2(N)$ in the proof
of Theorem \ref{thm:consistency-gamma}. Specifically, we start with the following
decomposition:
\begin{multline*}
   \left| \Gamma_{n-1,N}^b Q_{n,N}^{\otimes 2}(f\otimes g)
  -
  \Gamma_{n-1,N}^b Q_{n}^{\otimes 2}(f\otimes g)\right|
  \\
  \leq
  \underbrace{
    \left| \Gamma_{n-1,N}^b Q_{n,N}^{\otimes 2}(f\otimes g)
  -
    \Gamma_{n-1,N}^b (Q_{n,N}\otimes Q_{n})(f\otimes g)\right|
}_{D_1(N)}
  \\
  +
  \underbrace{
    \left| \Gamma_{n-1,N}^b (Q_{n,N}\otimes Q_{n})(f\otimes g)
  -
  \Gamma_{n-1,N}^b Q_{n}^{\otimes 2}(f\otimes g)\right|
}_{D_2(N)}.
\end{multline*}
For $D_1(N)$, we may write
$$(Q_{n,N}^{\otimes 2}(f\otimes g)-(Q_{n,N}\otimes Q_{n})(f\otimes g))(x,y)=Q_{n,N}(f)(x)(Q_{n,N}(g)(y)-Q_{n}(g)(y)).$$
By $\mathcal{A}$\ref{A2}, for any $g\in \mathcal{B}_b(E_n)$,
there exists a bounded function $h_{n-1}$ such that
$$
\left|Q_{n,N}(g)(y) - Q_n(g)(y)\right| 
=
\left|\left\langle h_{n-1}(y,Z_{n-1}^N),Z_{n-1}^N-z_{n-1}^*\right\rangle\right|
\leq 
\norm{h_{n-1}}_{\infty}
\abs{Z_{n-1}^N-z_{n-1}^*}.
$$
Since, in addition,
$$|Q_{n,N}(f)(x)|\leq \norm{G_{n-1,\cdot}}_{\infty}\norm{f}_{\infty},$$
it comes
$$
D_1(N)
\leq 
\Gamma_{n-1,N}^b(1) \norm{G_{n-1,\cdot}}_{\infty}\norm{f}_{\infty}
\norm{h_{n-1}}_{\infty}
\abs{Z_{n-1}^N-z_{n-1}^*}.$$
By Proposition \ref{prop:Gamma1}, one has 
$$
\Gamma_{n-1,N}^b(1)
=\mathscr{O}_{\mathbf{p}}\left(1\right).
$$
In addition, a by-product \eqref{pf:ZN} of Theorem \ref{thm-old:clt} is that
$$
\abs{Z_{n-1}^N-z_{n-1}^*}
=
\mathscr{O}_{\mathbf{p}}\left(\frac{1}{\sqrt{N}}\right).
$$
Hence, one concludes that
$$
D_1(N)
=
\mathscr{O}_{\mathbf{p}}\left(\frac{1}{\sqrt{N}}\right).
$$
The reasoning for $D_2(N)$ is the same.
\end{proof}

\subsection{Connection between the estimators}\label{zicj}
In this section, we give some combinatorial results on the coalescent
tree-based measures $\Gamma_{n,N}^b$. In particular, they allow us to
connect the variance estimator (\ref{ljsncn}) of Lee \& Whiteley and our term by term
estimators. As mentioned before, these relations do not depend on
$\mathcal{A}$\ref{A2}: they are  provided by the structure of the IPS and the
underlying multinomial selection scheme. In this respect, recall that, under $\mathcal{A}$\ref{A1}, Equation
(\ref{eq:selection}) is always well-defined for the denominator is always
strictly positive, and the same holds true for the IPS itself. This is in fact the only condition required here.

\begin{proposition}\label{prop:decomposition-fixed}
Provided that the IPS is well-defined, then for any test function $F \in \mathcal{B}_b(E_n^2)$,
  we have the decompositions:
$$
(\gamma_n^N)^{\otimes 2}(F)
=
\sum_{b\in \{0,1\}^{n+1}}
\left\{
\prod_{p=0}^n
\frac{(N-1)^{1-b_p}}{N}
\right\}
\Gamma_{n,N}^b(F),
$$
and 
$$
(\eta_n^N)^{\otimes 2}(F)
=
\sum_{b\in \{0,1\}^{n+1}}
\left\{
\prod_{p=0}^n
\frac{(N-1)^{1-b_p}}{N}
\right\}
\bar{\Gamma}_{n,N}^b(F).
$$
\end{proposition}

\begin{proof}
Since
$$ \bar{\Gamma}_{n,N}^b(F):=
  \frac{N^{n-1}}{(N-1)^{n+1}}
  \sum_{\ell_{0:n}^{[2]}\in \left((N)^2\right)^{\times(n+1)}}
  \left\{\prod_{p=0}^{n-1}
  \lambda_p^b(A_p^{\ell_{p+1}^{[2]}},\ell_p^{[2]})
  \right\}
  C_{b_n}(F)(X_n^{\ell_n^{[2]}}),$$
we have 
\begin{align*}
&\sum_{b\in \{0,1\}^{n+1}}
\left\{\prod_{p=0}^n
\frac{(N-1)^{1-b_p}}{N}\right\}
\bar{\Gamma}_{n,N}^b(F)
\\
&=
\frac{N^{n-1}}{(N-1)^{n+1}}
\sum_{b\in \{0,1\}^{n+1}}
\sum_{\ell_{0:n}^{[2]}\in \left((N)^2\right)^{\times(n+1)}}
\left\{\prod_{p=0}^n
\frac{(N-1)^{1-b_p}}{N}\right\}
\left\{\prod_{p=0}^{n-1}\lambda_p^b(A_p^{\ell_{p+1}^{[2]}},\ell_p^{[2]})\right\}
C_{b_n}(F)(X_n^{\ell_n^{[2]}}).
\end{align*}
Enumerating all the possibilities for the coalescence indicator $b\in \{0,1\}^{n+1}$ leads to
\begin{align*}
&\sum_{b\in \{0,1\}^{n+1}}
\left\{\prod_{p=0}^n
\frac{(N-1)^{1-b_p}}{N}\right\}
\bar{\Gamma}_{n,N}^b(F)
\\
=&
\sum_{\ell_0^{[2]}\in (N)^2}
\cdots
\sum_{\ell_{n-1}^{[2]}\in (N)^2}
\left\{
\prod_{p=0}^{n-1}
\left(\frac{1}{N}
    \mathbf{1}_{\{A_p^{\ell_{p+1}^1}
    =A_p^{\ell_{p+1}^2} = \ell_p^1\neq\ell_p^2\}
    }
     +
    \frac{N-1}{N} \mathbf{1}_{\{A_p^{\ell_{p+1}^1}
    =\ell_p^1 \neq A_p^{\ell_{p+1}^2} = \ell_p^2\}
    }
  \right)
\right\}
\\
&
\left(
\frac{N}{N-1}
\right)^n
\left\{
  \frac{N-1}{N}
  m^{\odot 2}(\mathbf{X_n})C_0(F)
  +
  \frac{1}{N}
  m^{\odot 2}(\mathbf{X_n})C_1(F)
\right\}.
\end{align*}
To conclude, one just has to observe that
$$\sum_{\ell_0^{[2]}\in (N)^2}
\cdots
\sum_{\ell_{n-1}^{[2]}\in (N)^2}
\left\{
\prod_{p=0}^{n-1}
\left(\frac{1}{N}
    \mathbf{1}_{\{A_p^{\ell_{p+1}^1}
    =A_p^{\ell_{p+1}^2} = \ell_p^1\neq\ell_p^2\}
    }
     +
    \frac{N-1}{N} \mathbf{1}_{\{A_p^{\ell_{p+1}^1}
    =\ell_p^1 \neq A_p^{\ell_{p+1}^2} = \ell_p^2\}
    }
  \right)
\right\}=\left(
    \frac{N-1}{N}
\right)^n,$$
while, by (\ref{lasjcn}),
$$ \frac{N-1}{N}
  m^{\odot 2}(\mathbf{X_n})C_0(F)
  +
  \frac{1}{N}
  m^{\odot 2}(\mathbf{X_n})C_1(F)=m^{\otimes 2}(\mathbf{X_n})(F)=(\eta_n^N)^{\otimes 2}(F).$$
Multiplying both sides by $\gamma_n^N(1)^2$ gives the corresponding relation for $(\gamma_n^N)^{\otimes 2}(F)$.
\end{proof}

We can now proceed with the proof of Proposition \ref{lm:convergence-proba}. Recall that the goal is to show that 
$$
  N V_n^N(f)
  -
 \sigma^2_{\eta_{n,N}}(f) 
  =N V_n^N(f)
  -
  \sum_{p = 0}^n
  \left\{ \bar{\Gamma}_{n,N}^{(p)}(f^{\otimes 2}) -
  \bar{\Gamma}_{n,N}^{(\varnothing)}(f^{\otimes 2})\right\} 
  =\mathscr{O}_{\mathbf{p}}\left(\frac{1}{N}\right),
  $$
  and
  $$
  N V_n^N(f-\eta_n^N(f))
  -
  \sum_{p = 0}^n
  \left\{ \bar{\Gamma}_{n,N}^{(p)}
    \left(
      \left[f-\eta_n^N(f)
      \right]^{\otimes 2}
      \right) -
  \bar{\Gamma}_{n,N}^{(\varnothing)}
    \left(
      \left[f-\eta_n^N(f)
      \right]^{\otimes 2}
      \right) 
    \right\}
  =
  \mathscr{O}_{\mathbf{p}}\left(\frac{1}{N}\right).
$$
By construction, we have
$$
V_n^N(f) =
\eta_n^N(f)^2 - \bar\Gamma_{n,N}^{(\varnothing)}(f^{\otimes 2})=(\eta_n^N)^{\otimes 2}(f^{\otimes 2})- \bar\Gamma_{n,N}^{(\varnothing)}(f^{\otimes 2}).
$$
An implication of Proposition \ref{prop:Gamma1} is that, for any test function $f$ and any  coalescence indicator $b$, 
$$\bar{\Gamma}_{n,N}^{b}(f^{\otimes 2})=\mathscr{O}_{\mathbf{p}}\left(1\right).$$
Thus, a consequence of Proposition \ref{prop:decomposition-fixed} is
$$(\eta_n^N)^{\otimes 2}(f^{\otimes 2})=\left(\frac{N-1}{N}\right)^{n+1}\bar\Gamma_{n,N}^{(\varnothing)}(f^{\otimes 2})+\frac{1}{N}\left(\frac{N-1}{N}\right)^{n}\sum_{p=0}^{n}\bar{\Gamma}_{n,N}^{(p)}(f^{\otimes 2})+\mathscr{O}_{\mathbf{p}}\left(\frac{1}{N^2}\right).$$
The desired formula is then obtained by remarking that
$$\left(\frac{N-1}{N}\right)^{n}=1-\mathscr{O}\left(\frac{1}{N}\right)\quad\mbox{and}\quad\left(\frac{N-1}{N}\right)^{n+1}-1=-\frac{n+1}{N}+\mathscr{O}\left(\frac{1}{N^2}\right).$$
Similarly, since
$$
\bar{\Gamma}_{n,N}^b
\left(
  \left[f-\eta_n^N(f)\right]^{\otimes 2}
  \right) 
= \mathscr{O}_{\mathbf{p}}\left(1\right),
$$
the same algebraic manipulation yields 
$$
  N V_n^N(f-\eta_n^N(f))
  -
  \sum_{p = 0}^n
  \left\{ \bar{\Gamma}_{n,N}^{(p)}
    \left(
      \left[f-\eta_n^N(f)
      \right]^{\otimes 2}
      \right) -
  \bar{\Gamma}_{n,N}^{(\varnothing)}
    \left(
      \left[f-\eta_n^N(f)
      \right]^{\otimes 2}
      \right) 
    \right\}
  =
  \mathscr{O}_{\mathbf{p}}\left(\frac{1}{N}\right).
$$
This closes the proof of Proposition \ref{lm:convergence-proba}.

\begin{appendices}

\section{Many-body Feynman-Kac models}
\label{sec:mbfk}

The many-body Feynman-Kac model was proposed
in \cite{del2016gibbs} to study the propagation of chaos property of
the Conditional Particle Markov Chain Monte Carlo introduced in \cite{Andrieu2010pmcmc}.
The basic idea is to trace the information of all particles in the IPS along
with its genealogy, and to construct an instrumental particle block which is 
heavily dependent (identical) to some specific particles. We call these
instrumental particles the \emph{coupled particle block} of the IPS.

\subsection{Duality formula}
At each layer, the particles in the original IPS are denoted by
$\mathbf{X_{p}}$, with its genealogy $\mathbf{A_{p-1}}$. The coupled particle
block of $q$ particles is denoted by $\tilde{X}_{p}^{[q]}$, with
its genealogy denoted by $\tilde{A}_{p-1}^{[q]}$. The corresponding variables in the integral
operators will be denoted by $\bm{x_{p}}$, $\bm{a_{p-1}}$,
$\tilde{x}_{p}^{[q]}$ and $\tilde{a}_{p-1}^{[q]}$ respectively.

Before giving specific definitions, we want to mention that the
mathematical object we would like to look into is the \emph{whole particle
system}, namely the \emph{original IPS} and the \emph{coupled particle
block} with
genealogy. At each layer $p$, we are interested by the tuple:
$$
(\mathbf{X_{p}}, \mathbf{A_{p-1}}, \tilde{X}_p^{[q]}, \tilde{A}_{p-1}^{[q]}).
$$
As for the basic idea of Particle Markov Chain Monte Carlo method
\cite{Andrieu2010pmcmc}, we study respectively the distributions of 
$$
\mathbf{X_{p}}, \mathbf{A_{p-1}} 
\;\big\vert\;
\tilde{X}_p^{[q]},\tilde{A}_{p-1}^{[q]}
$$
and
$$
\tilde{X}_p^{[q]},\tilde{A}_{p-1}^{[q]}
\;\big\vert\;
\mathbf{X_{p}}, \mathbf{A_{p-1}} 
.
$$
Thanks to the specific construction, as well as the relatively simple
multinomial resampling scheme of Feynman-Kac IPS, Lemma \ref{lm:duality} provides a duality
formula to connect both distributions and leads in particular to Proposition \ref{lm:LackOfBias}. This latter result is crucial to prove the consistency of our term by term estimator in Theorem \ref{thm:coalescent-measure}.

In this section, a transition kernel denoted by the letter $Q$ is a
Feynman-Kac kernel, meaning that its total mass is not
necessarily 1, and it can be expressed by the product of a positive potential function and
a Markov kernel. All transition kernels denoted by the letter $M$ are
Markov kernels.

Notice that the transition from level $p-1$ to
level $p$ of the IPS with its genealogy defined in Section \ref{sec-old:IPS} can be expressed as
$$
(\mathbf{A_{p-1}},\mathbf{X_{p}})
\sim
\bigotimes_{i=1}^N \Phi_{p,N}(\mathbf{X_{p-1}},d(A_{p-1}^i,X_{p}^i))
$$
with $\Phi_{p,N}$ defined by
$$
\Phi_{p,N}(\bm{x_{p-1}},d(a_{p-1}^i,x_{p}^i)) = 
S_{p-1,N}(\bm{x_{p-1}},da_{p-1}^i)\times 
M_{p,N}(x_{p-1}^{a_{p-1}^i},dx_{p}^i).
$$
We define the transition of the original IPS with its genealogy by
$$
\mathcal{M}_{p}
(\bm{x_{p-1}},d(\bm{a_{p-1}},\bm{x_p}))
:=
\prod_{i=1}^N
  \Phi_{p,N}(\bm{x_{p-1}},d(a_{p-1}^i,x_p^i))
$$
and the potential function of the particle block of size $q$ by
$$
\mathcal{G}_{p-1}^{(q)}(\bm{x_{p-1}})
:=
m(\bm{x_{p-1}})(G_{p-1,N})^q.
$$
We denote the associated Feynman-Kac transition kernel 
$$
\mathbf{Q}_{p}^{(q)}
(\bm{x_{p-1}},d(\bm{a_{p-1}},\bm{x_p}))
:=
\mathcal{G}_{p-1}^{(q)}(\bm{x_{p-1}})
\times
\mathcal{M}_{p}
(\bm{x_{p-1}},d(\bm{a_{p-1}},\bm{x_p})).
$$
Given $\ell_p^{[q]}\in (N)^q$, $\tilde{a}_{p-1}^{[q]}\in [N]^q$
and $\tilde{x}_p^{[q]}\in E_{p}^q$,
we define
\begin{multline*}
\mathbb{M}_{p}^{\tilde{a}_{p-1}^{[q]},\ell_p^{[q]},\tilde{x}_p^{[q]}}
(\bm{x_{p-1}},d(\bm{a_{p-1}},\bm{x_p}))
\\
:=
\prod_{i\in [N]\backslash \{\ell_p^1,\dots,\ell_p^q\}}\left\{
  \Phi_{p,N}(\bm{x_{p-1}},d(a_{p-1}^i,x_p^i))
\right\}
\times
\delta_{\tilde{x}_p^{[q]}}(d
 x_p^{\ell_p^{[q]}})
 \times
\delta_{\tilde{a}_{p-1}^{[q]}}(da_{p-1}^{\ell_p^{[q]}})
\end{multline*}
the conditional transition for the 
original particle system given the coupled particle block
$\tilde{X}_p^{[q]}=\tilde{x}_p^{[q]}$ at position $\ell_p^{[q]}$ with frozen genealogy   
$\tilde{A}_{p-1}^{[q]} = \tilde{a}_{p-1}^{[q]}$.
In particular, we denote
$$
\mathbb{M}_{0}^{\ell_0^{[q]},\tilde{x}_0^{[q]}}(d\bm{x_0}) := 
\left\{\prod_{i\in[N]\backslash
\{\ell_0^1,\dots,\ell_0^q\}}\eta_0(dx_0^i)\right\}
\times \delta_{\tilde{x}_0^{[q]}}(dx_0^{\ell_0^{[q]}}).
$$
We also define
$$
\mathbb{Q}_{p}^{(q)}(\bm{x_{p-1}},d(\tilde{a}_{p-1}^{[q]},\tilde{x}_{p}^{[q]}))
:=
m([N])^{\otimes q}(d \tilde{a}_{p-1}^{[q]})\ 
Q_{p,N}^{\otimes q}(x_{p-1}^{\tilde{a}_{p-1}^{[q]}}, d\tilde{x}_p^{[q]}),
$$
and
$$
\mathbf{M}_p^{\ell_p^{[q]}}\left(
  (\bm{a_{p-1}},\bm{x_{p}}),d(\tilde{a}_{p-1}^{[q]},\tilde{x}_{p}^{[q]})
  \right)
:=
\delta_{a_{p-1}^{\ell_p^{[q]}}}(d\tilde{a}_{p-1}^{[q]})\ \delta_{x_p^{\ell_p^{[q]}}}(d\tilde{x}_p^{[q]}).
$$

Then we have the following pivotal duality formula, which is simply a generalization of Lemma 4.1 in  \cite{del2016gibbs}. We will apply it in the proof of Proposition \ref{lm:LackOfBias} with $q=2$.

\begin{lemma}\label{lm:duality}
  For $p \geq 1$, $q\in [N]$ and $\ell_p^{[q]}\in(N)^q$, we have the following 
  duality formula between integral operators
  \begin{multline*}
    \mathbf{Q}_{p}^{(q)}(\bm{x_{p-1}},d(\bm{a_{p-1}},\bm{x_p}))\ 
     \mathbf{M}_p^{\ell_p^{[q]}}\left(
       (\bm{a_{p-1}},\bm{x_{p}}),d(\tilde{a}_{p-1}^{[q]},\tilde{x}_{p}^{[q]})
       \right)
       \\= \mathbb{Q}_{p}^{(q)}
     (\bm{x_{p-1}}, d(\tilde{a}_{p-1}^{[q]},\tilde{x}_p^{[q]}))\ 
     \mathbb{M}_{p}^{\tilde{a}_{p-1}^{[q]},\ell_p^{[q]},\tilde{x}_p^{[q]}}
        (\bm{x_{p-1}},d(\bm{a_{p-1}},\bm{x_p})),
  \end{multline*}
   and 
  \begin{equation*}
    \eta_0^{\otimes N}(d \bm{x_0})\ 
       \delta_{x_0^{\ell_0^{[q]}}}(d\tilde{x}_0^{[q]})
      = \eta_0^{\otimes q} 
      (d\tilde{x}_0^{[q]})\ 
     \mathbb{M}_{0}^{\ell_0^{[q]},\tilde{x}_0^{[q]}}(d\bm{x_0}).
  \end{equation*}
\end{lemma}

\begin{proof}
    Step $0$ is clear. For
    $p\geq 1$, it suffices to check that the nonidentical
    parts are equal, namely
    \begin{multline*}
    \mathcal{G}_{p-1}^{(q)}(\bm{x_{p-1}})
    \left\{\sum_{k = 1}^{n} \frac{G_{p-1,N}(x_{p-1}^k)}{N\ 
    m(\bm{x_{p-1}})(G_{p-1,N})}\delta_k\right\}^{\otimes q}(d
    a_{p-1}^{\ell_p^{[q]}})\ 
    M_{p,N}^{\otimes q}(x_{p-1}^{a_{p-1}^{\ell_p^{[q]}}},dx_p^{\ell_p^{[q]}})\\
    \mathbf{M}_p^{\ell_p^{[q]}}\left(
    (\bm{a_{p-1}},\bm{x_{p}}),d(\tilde{a}_{p-1}^{[q]},\tilde{x}_{p}^{[q]})
    \right)
    =
    \mathbb{Q}_{p}^{(q)}(\bm{x_{p-1}},d(\tilde{a}_{p-1}^{[q]},\tilde{x}_{p}^{[q]}))\ 
    \delta_{\tilde{x}_p^{[q]}}(dx_p^{\ell_p^{[q]}})\ 
    \delta_{\tilde{a}_{p-1}^{[q]}}(da_{p-1}^{\ell_p^{[q]}}).
    \end{multline*}
    Fixing $\ell_p^{[q]} \in(N)^q$ and $\bm{x_{p-1}}\in E_{p-1}^N$, 
   consider a function $\mathbf{F}\in\mathcal{B}_b([N]^q\times[N]^q\times E_p^q\times E_p^q)$. Moreover, let $a^{[q]}=(a_1,\dots,a_q)$ and $x^{[q]}=(x_1,\dots,x_q)$ denote generic variables belonging respectively to $[N]^q$ and $E_p^q$. Then,
    we may write 
    \begin{equation*}
        \begin{split}
            &\int
          m(\bm{x_{p-1}})(G_{p-1,N})^q\left\{\sum_{k = 1}^{N} \frac{G_{p-1,N}(x_{p-1}^k)}{N\ 
        m(\bm{x_{p-1}})(G_{p-1,N})}\delta_k\right\}^{\otimes q}(d
            a_{p-1}^{\ell_p^{[q]}})
            \\& 
            \ \ \ \ \ M_{p,N}^{\otimes q}(x_{p-1}^{a_{p-1}^{\ell_p^{[q]}}},dx_p^{\ell_p^{[q]}})\ 
            \delta_{x_p^{\ell_p^{[q]}}}(d\tilde{x}_p^{[q]})\ 
            \delta_{a_{p-1}^{\ell_p^{[q]}}}(d\tilde{a}_{p-1}^{[q]})\
            \mathbf{F}(\tilde{a}_{p-1}^{[q]},a_{p-1}^{\ell_p^{[q]}}
            ,x_p^{\ell_p^{[q]}},\tilde{x}_p^{[q]})            
            \\=&
            \int             m([N])^{\otimes q}(d  a^{[q]})\
            G_{p-1,N}^{\otimes
            q}(x_{p-1}^{ a^{[q]}})\ 
            M_{p,N}^{\otimes q}(x_{p-1}^{ a^{[q]}},d x^{[q]})\ \mathbf{F}( a^{[q]}, a^{[q]}, x^{[q]}, x^{[q]})
             \\=&
            \int            m([N])^{\otimes q}(d  a^{[q]})
            \  Q_{p,N}^{\otimes q}(x_{p-1}^{ a^{[q]}},d x^{[q]})\  \mathbf{F}( a^{[q]}, a^{[q]}, x^{[q]}, x^{[q]})
            \\=&
            \int  m([N])^{\otimes q}(d
            \tilde{a}_{p-1}^{[q]})\ Q_{p,N}^{\otimes
            q}(x_{p-1}^{\tilde{a}_{p-1}^{[q]}}, d
            \tilde{x}_p^{[q]})\ \delta_{\tilde{x}_p^{[q]}}(dx_p^{\ell_p^{[q]}})\ 
            \delta_{\tilde{a}_{p-1}^{[q]}}(da_{p-1}^{\ell_p^{[q]}})\ \mathbf{F}(\tilde{a}_{p-1}^{[q]},a_{p-1}^{\ell_p^{[q]}},x_p^{\ell_p^{[q]}},\tilde{x}_p^{[q]})
            \\=&
            \int  \mathbb{Q}_{p}^{(q)}(\bm{x_{p-1}},d(\tilde{a}_{p-1}^{[q]},\tilde{x}_{p}^{[q]}))\ 
            \delta_{\tilde{x}_p^{[q]}}(dx_p^{\ell_p^{[q]}})\ 
            \delta_{\tilde{a}_{p-1}^{[q]}}(da_{p-1}^{\ell_p^{[q]}})\ \mathbf{F}(\tilde{a}_{p-1}^{[q]},a_{p-1}^{\ell_p^{[q]}},x_p^{\ell_p^{[q]}},\tilde{x}_p^{[q]}).
        \end{split}
    \end{equation*}
    This ends the proof of the duality formula.
\end{proof}

Let us recall \eqref{eq:LambdaElln} and \eqref{pf:Lambda1}:
\begin{equation}\label{lsjclcn}
\Lambda_{n}^{\ell_n^{[2]}}
=
\sum_{\ell_{0:n-1}^{[2]}\in \left((N)^2\right)^{\times n}}
\left\{\prod_{p=0}^{n-1}
\lambda_p^b(A_p^{\ell_{p+1}^{[2]}},\ell_p^{[2]})
\right\}=
\sum_{\ell_{n-1}^{[2]}\in (N)^2} \Lambda_{n-1}^{\ell_{n-1}^{[2]}}\
\lambda_{n-1}^b(A_{n-1}^{\ell_n^{[2]}},\ell_{n-1}^{[2]}).
\end{equation}
with the convention
$
\Lambda_{0}^{\ell_0^{[2]}}
=1
$. 
In fact, this gives another representation of the approximation of the
coalescent tree-based measures:
$$
\Gamma_{n,N}^b(f\otimes g)=
\gamma_n^N(1)^2\ 
\frac{N^{n-1}}{(N-1)^{n+1}}
\sum_{\ell_n^{[2]}\in (N)^2}
\Lambda_{n}^{\ell_n^{[2]}} 
C_{b_n}(f\otimes g)(X_n^{\ell_n^{[2]}}).
$$
Recall that
$$
\mathscr{G}_{n-1}^N :=
\sigma(\mathbf{X_{0}},\dots,\mathbf{X_{n-1}},\mathbf{A_{0}},\dots,\mathbf{A_{n-2}}).
$$
The upcoming result is useful in the proof of Lemma \ref{lm:cvg-proba-Gamma}.

\begin{proposition}\label{lm:LackOfBias}
  Under $\mathcal{A}$\ref{A1}, 
  for any $\ell_n^{[2]}\in(N)^2$, 
  any coalescence indicator $b$, 
  and any test functions $f$ and $g$ in $\mathcal{B}_b(E_n)$, we have, for all
  $n\geq 1$, that
  $$
  \mathbf{E}\left[
  \gamma_n^N(1)^2\ 
  \frac{N^{n-1}}{(N-1)^{n+1}}
 \Lambda_{n}^{\ell_n^{[2]}} 
  C_{b_n}(f\otimes g)(X_n^{\ell_n^{[2]}})
  \;\Bigg\vert\;
  \mathscr{G}_{n-1}^N
  \right]
  =
  \frac{1}{N(N-1)}
  \Gamma_{n-1,N}^b Q_{n,N}^{\otimes 2}C_{b_n}(f\otimes g).
  $$
\end{proposition}

\begin{proof}
  By applying \eqref{lsjclcn}, we obtain
  \begin{multline*}
  \gamma_n^N(1)^2
  \frac{N^{n-1}}{(N-1)^{n+1}}
  \Lambda_{n}^{\ell_n^{[2]}} 
  C_{b_n}
  (f\otimes g)(X_n^{\ell_n^{[2]}})
  \\
  =
  \gamma_n^N(1)^2 
  \frac{N^{n-1}}{(N-1)^{n+1}}
  \sum_{\ell_{n-1}^{[2]}\in (N)^2} \Lambda_{n-1}^{\ell_{n-1}^{[2]}}
  \lambda_{n-1}^b(A_{n-1}^{\ell_n^{[2]}},\ell_{n-1}^{[2]})
  C_{b_n}
  (f\otimes g)(X_n^{\ell_n^{[2]}}).
  \end{multline*}
  Since $\Lambda_{n-1}^{\ell_{n-1}^{[2]}}$ is
  $\mathscr{G}_{n-1}^N$-measurable,
  it is sufficient to show that for each $\ell_{n-1}^{[2]}\in (N)^2$,
  we have 
  \begin{equation}\label{pf:StepnLambda}
  \begin{split}    
  &
  \mathbf{E}\left[
   m(\mathbf{X_{n-1}})(G_{n-1,N})^2
  \lambda_{n-1}^b(A_{n-1}^{\ell_n^{[2]}},\ell_{n-1}^{[2]})
  C_{b_n}
  (f\otimes g)(X_n^{\ell_n^{[2]}})
  \;\Bigg\vert\;
  \mathscr{G}_{n-1}^N
  \right]
  \\
    =&
  \frac{1}{N^2}
  C_{b_{n-1}}
  Q_{n,N}^{\otimes 2}C_{b_n} (f\otimes g)(X_{n-1}^{\ell_{n-1}^{[2]}}).
  \end{split}
  \end{equation}
  Before starting our reasoning, for the sake of simplification, we remark that 
  $$
  \lambda_{n-1}^b(A_{n-1}^{\ell_n^{[2]}},\ell_{n-1}^{[2]})
  C_{b_n}
  (f\otimes g)(X_n^{\ell_n^{[2]}})
  $$
  can be seen as a bounded measurable function of $(\mathbf{A_{n-1}},\mathbf{X_n})$, 
  rather than a measurable
  function of $(X_n^{\ell_n^{[2]}},A_{n-1}^{\ell_n^{[2]}})$.
  With this in mind, for any test function 
  $$
  F \in \mathcal{B}_b(\underbrace{[N]^N\times\cdots\times[N]^N}_{(n-1) \text{ times}}\times E_0^N\times\cdots\times E_{n-1}^N),
  $$ we have, by definition of
  $
  \mathbf{Q}_{p}^{(2)}
  (\bm{x_{p-1}},d(\bm{a_{p-1}},\bm{x_p}))
  $, 
  \begin{equation}\label{pf:cond-exp1}
  \begin{aligned}
  &\mathbf{E}\left[
  m(\mathbf{X_{n-1}})(G_{n-1,N})^2
  \lambda_{n-1}^b(A_{n-1}^{\ell_n^{[2]}},\ell_{n-1}^{[2]})
  C_{b_n}
  (f\otimes g)(X_n^{\ell_n^{[2]}})
  F(\mathbf{A_{0:n-2}},\mathbf{X_{0:n-1}})
  \right]
  \\=&
   \int
   \mathbf{Q}_n^{(2)}(\bm{x_{n-1}}, d(\bm{a_{n-1}},\bm{x_n}))
   \lambda_{n-1}(a_{n-1}^{\ell_n^{[2]}},\ell_{n-1}^{[2]})
   C_{b_n}(f\otimes g)(x_n^{\ell_n^{[2]}})
   \\
      &\quad\; 
      F(\bm{a_{0:n-2}},\bm{x_{0:n-1}})
    \mu_{n-1}(d\bm{a_{0:n-2}},d\bm{x_{0:n-1}}),
  \end{aligned}
  \end{equation}
  where $\mu_{n-1}$ denotes the measure corresponding to the underlying 
  joint distribution of the IPS from
  step 0 to step $n-1$ with genealogy.
  Taking into account that
  $$
  \mathbf{M}_n^{\ell_n^{[2]}}\left(
  (\bm{a_{n-1}},\bm{x_{n}}),d(\tilde{a}_{n-1}^{[2]},\tilde{x}_{n}^{[2]})
  \right)=\delta_{a_{n-1}^{\ell_n^{[2]}}}(d\tilde{a}_{n-1}^{[2]})\ \delta_{x_n^{\ell_n^{[2]}}}(d\tilde{x}_n^{[2]})
  $$
  is a Markov kernel, we can introduce it in the right-hand side of
  \eqref{pf:cond-exp1} to obtain  
  $$
  \begin{aligned}
  &\mathbf{E}\left[
  m(\mathbf{X_{n-1}})(G_{n-1,N})^2
  \lambda_{n-1}^b(A_{n-1}^{\ell_n^{[2]}},\ell_{n-1}^{[2]})
  C_{b_n}
  (f\otimes g)(X_n^{\ell_n^{[2]}})
  F(\bm{a_{0:n-2}},\bm{x_{0:n-1}})
  \right]
  \\=&
   \int
   \mathbf{Q}_n^{(2)}(\bm{x_{n-1}}, d(\bm{a_{n-1}},\bm{x_n}))
    \mathbf{M}_n^{\ell_n^{[2]}}\left((\bm{a_{n-1}},\bm{x_n}),d
      (\tilde{a}_{n-1}^{[2]},\tilde{x}_{n}^{[2]})\right)
   \\
      &\quad 
   \lambda_{n-1}(a_{n-1}^{\ell_n^{[2]}},\ell_{n-1}^{[2]})
   C_{b_n}(f\otimes g)(x_n^{\ell_n^{[2]}})
      F(\bm{a_{0:n-2}},\bm{x_{0:n-1}})
    \mu_{n-1}(d\bm{a_{0:n-2}},d\bm{x_{0:n-1}}).
  \end{aligned}
  $$
  The design of many-body Feynman-Kac models allows replacing
  $(a_{n-1}^{\ell_n^{[2]}},x_n^{\ell_n^{[2]}})$ with
  $(\tilde{a}_{n-1}^{[2]},\tilde{x}_{n}^{[2]})$ in the observation functions,
  as they are equal by definition. Hence, one has
  the following equality:
  $$
  \begin{aligned}
  &
   \int
   \mathbf{Q}_n^{(2)}(\bm{x_{n-1}}, d(\bm{a_{n-1}},\bm{x_n}))
    \mathbf{M}_n^{\ell_n^{[2]}}\left((\bm{a_{n-1}},\bm{x_n}),d
      (\tilde{a}_{n-1}^{[2]},\tilde{x}_{n}^{[2]})\right)
   \\
      &\quad 
   \lambda_{n-1}(a_{n-1}^{\ell_n^{[2]}},\ell_{n-1}^{[2]})
   C_{b_n}(f\otimes g)(x_n^{\ell_n^{[2]}})
      F(\bm{a_{0:n-2}},\bm{x_{0:n-1}})
    \mu_{n-1}(d\bm{a_{0:n-2}},d\bm{x_{0:n-1}})
  \\=&
   \int
   \mathbf{Q}_n^{(2)}(\bm{x_{n-1}}, d(\bm{a_{n-1}},\bm{x_n}))
    \mathbf{M}_n^{\ell_n^{[2]}}\left((\bm{a_{n-1}},\bm{x_n}),d
      (\tilde{a}_{n-1}^{[2]},\tilde{x}_{n}^{[2]})\right)
   \\
      &\quad 
      \lambda_{n-1}(\tilde{a}_{n-1}^{[2]},\ell_{n-1}^{[2]})
   C_{b_n}(f\otimes g)(\tilde{x}_n^{[2]})
      F(\bm{a_{0:n-2}},\bm{x_{0:n-1}})
    \mu_{n-1}(d\bm{a_{0:n-2}},d\bm{x_{0:n-1}}).
  \end{aligned}
  $$
  Now, the duality formula given in Lemma \ref{lm:duality} yields
  $$
  \begin{aligned}
  &
   \int
   \mathbf{Q}_n^{(2)}(\bm{x_{n-1}}, d(\bm{a_{n-1}},\bm{x_n}))
    \mathbf{M}_n^{\ell_n^{[2]}}\left((\bm{a_{n-1}},\bm{x_n}),d
      (\tilde{a}_{n-1}^{[2]},\tilde{x}_{n}^{[2]})\right)
   \\
      &\quad 
      \lambda_{n-1}(\tilde{a}_{n-1}^{[2]},\ell_{n-1}^{[2]})
   C_{b_n}(f\otimes g)(\tilde{x}_n^{[2]})
      F(\bm{a_{0:n-2}},\bm{x_{0:n-1}})
    \mu_{n-1}(d\bm{a_{0:n-2}},d\bm{x_{0:n-1}})
    \\=&
   \int
     \mathbb{Q}_{n}^{(2)}
     (\bm{x_{n-1}}, d(\tilde{a}_{n-1}^{[2]},\tilde{x}_n^{[2]}))\ 
     \mathbb{M}_{n}^{\tilde{a}_{n-1}^{[2]},\ell_n^{[2]},\tilde{x}_n^{[2]}}
        (\bm{x_{n-1}},d(\bm{a_{n-1}},\bm{x_n}))
   \\
      &\quad 
      \lambda_{n-1}(\tilde{a}_{n-1}^{[2]},\ell_{n-1}^{[2]})
   C_{b_n}(f\otimes g)(\tilde{x}_n^{[2]})
      F(\bm{a_{0:n-2}},\bm{x_{0:n-1}})
    \mu_{n-1}(d\bm{a_{0:n-2}},d\bm{x_{0:n-1}}).
  \end{aligned}
  $$
  In addition, since
  $$
  \mathbb{M}_{n}^{\tilde{a}_{n-1}^{[2]},\ell_n^{[2]},\tilde{x}_n^{[2]}}
     (\bm{x_{n-1}},d(\bm{a_{n-1}},\bm{x_n}))
  $$
  is a Markov kernel for any choice of $(\tilde{a}_{n-1}^{[2]},\ell_n^{[2]},\tilde{x}_n^{[2]})$,
  we deduce that
  $$
  \begin{aligned}
  &
   \int
     \mathbb{Q}_{n}^{(2)}
     (\bm{x_{n-1}}, d(\tilde{a}_{n-1}^{[2]},\tilde{x}_n^{[2]}))\ 
     \mathbb{M}_{n}^{\tilde{a}_{n-1}^{[2]},\ell_n^{[2]},\tilde{x}_n^{[2]}}
        (\bm{x_{n-1}},d(\bm{a_{n-1}},\bm{x_n}))
   \\
      &\quad 
      \lambda_{n-1}(\tilde{a}_{n-1}^{[2]},\ell_{n-1}^{[2]})
   C_{b_n}(f\otimes g)(\tilde{x}_n^{[2]})
      F(\bm{a_{0:n-2}},\bm{x_{0:n-1}})
    \mu_{n-1}(d\bm{a_{0:n-2}},d\bm{x_{0:n-1}})
    \\=&
   \int
     \mathbb{Q}_{n}^{(2)}
     (\bm{x_{n-1}}, d(\tilde{a}_{n-1}^{[2]},\tilde{x}_n^{[2]}))\ 
      \lambda_{n-1}(\tilde{a}_{n-1}^{[2]},\ell_{n-1}^{[2]})
   C_{b_n}(f\otimes g)(\tilde{x}_n^{[2]})
   \\
      &\quad 
      F(\bm{a_{0:n-2}},\bm{x_{0:n-1}})
    \mu_{n-1}(d\bm{a_{0:n-2}},d\bm{x_{0:n-1}}).
  \end{aligned}
  $$
  Next, let us recall that 
  $$
  \mathbb{Q}_{n}^{(2)}(\bm{x_{n-1}},d(\tilde{a}_{n-1}^{[2]},\tilde{x}_{n}^{[2]}))
  :=
  m([N])^{\otimes 2}(d \tilde{a}_{n-1}^{[2]})\ 
  Q_{n,N}^{\otimes 2}(x_{n-1}^{\tilde{a}_{n-1}^{[2]}}, d\tilde{x}_n^{[2]})
  $$
  and
  $$
  \lambda_{n-1}(\tilde{a}_{n-1}^{[2]},\ell_{n-1}^{[2]})
  :=
  \mathbf{1}_{\{b_{n-1} = 1, \tilde{a}_{n-1}^1 =
      \tilde{a}_{n-1}^2 =  
  \ell_{n-1}^1\neq\ell_{n-1}^2\}}
  +
  \mathbf{1}_{\{b_{n-1} = 0, \tilde{a}_{n-1}^1  =
  \ell_{n-1}^1\neq\tilde{a}_{n-1}^2=\ell_{n-1}^2\}},
  $$
  whence we get the equality concerning the operator
  $C_{b_{n-1}}$.
  More precisely, if $b_{n-1} = 0$, we have
  \begin{equation}\label{pf:cond-exp3}
  \begin{aligned}
  &
   \int
     \mathbb{Q}_{n}^{(2)}
     (\bm{x_{n-1}}, d(\tilde{a}_{n-1}^{[2]},\tilde{x}_n^{[2]}))\ 
      \lambda_{n-1}(\tilde{a}_{n-1}^{[2]},\ell_{n-1}^{[2]})
   C_{b_n}(f\otimes g)(\tilde{x}_n^{[2]})
   \\
      &\quad 
      F(\bm{a_{0:n-2}},\bm{x_{0:n-1}})
    \mu_{n-1}(d\bm{a_{0:n-2}},d\bm{x_{0:n-1}})
    \\=&
   \int
      \frac{1}{N^2}
      Q_{n,N}^{\otimes 2}C_{b_n}(f\otimes g)(x_{n-1}^{\ell_{n-1}^{[2]}})
      F(\bm{a_{0:n-2}},\bm{x_{0:n-1}})
    \mu_{n-1}(d\bm{a_{0:n-2}},d\bm{x_{0:n-1}}).
  \end{aligned}
  \end{equation}
  Otherwise, if $b_{n-1} = 1$, we get, with the convention $x_{n-1}^{\ell_{n-1}^{(1,1)}}=(x_{n-1}^{\ell_{n-1}^{1}},x_{n-1}^{\ell_{n-1}^{1}})$,
  \begin{equation}\label{pf:cond-exp4}
  \begin{aligned}
  &
   \int
     \mathbb{Q}_{n}^{(2)}
     (\bm{x_{n-1}}, d(\tilde{a}_{n-1}^{[2]},\tilde{x}_n^{[2]}))\ 
      \lambda_{n-1}(\tilde{a}_{n-1}^{[2]},\ell_{n-1}^{[2]})
   C_{b_n}(f\otimes g)(\tilde{x}_n^{[2]})
   \\
      &\quad 
      F(\bm{a_{0:n-2}},\bm{x_{0:n-1}})
    \mu_{n-1}(d\bm{a_{0:n-2}},d\bm{x_{0:n-1}})
    \\=&
   \int
      \frac{1}{N^2}
      Q_{n,N}^{\otimes 2}C_{b_n}(f\otimes g)(x_{n-1}^{\ell_{n-1}^{(1,1)}})
      F(\bm{a_{0:n-2}},\bm{x_{0:n-1}})
    \mu_{n-1}(d\bm{a_{0:n-2}},d\bm{x_{0:n-1}}).
  \end{aligned}
  \end{equation}
  Combining \eqref{pf:cond-exp3} and \eqref{pf:cond-exp4}, we safely deduce
  that
  \begin{equation*}
  \begin{aligned}
  &
   \int
     \mathbb{Q}_{n}^{(2)}
     (\bm{x_{n-1}}, d(\tilde{a}_{n-1}^{[2]},\tilde{x}_n^{[2]}))\ 
      \lambda_{n-1}(\tilde{a}_{n-1}^{[2]},\ell_{n-1}^{[2]})
   C_{b_n}(f\otimes g)(\tilde{x}_n^{[2]})
   \\
      &\quad 
      F(\bm{a_{0:n-2}},\bm{x_{0:n-1}})
    \mu_{n-1}(d\bm{a_{0:n-2}},d\bm{x_{0:n-1}})
    \\=&
   \int
      \frac{1}{N^2}
      C_{b_{n-1}}Q_{n,N}^{\otimes 2}C_{b_n}(f\otimes g)(x_{n-1}^{\ell_{n-1}^{[2]}})
      F(\bm{a_{0:n-2}},\bm{x_{0:n-1}})
    \mu_{n-1}(d\bm{a_{0:n-2}},d\bm{x_{0:n-1}})
  \\=&
  \mathbf{E}\left[
    \frac{1}{N^2}
    C_{b_{n-1}}Q_{n,N}^{\otimes 2}C_{b_n}(f\otimes g)(X_{n-1}^{\ell_{n-1}^{[2]}})
    F(\mathbf{A_{0:n-2}},\mathbf{X_{0:n-1}})
  \right].
  \end{aligned}
  \end{equation*}
  In conclusion, we have established that
  \begin{equation*}
  \begin{aligned}
  &\mathbf{E}\left[
  m(\mathbf{X_{n-1}})(G_{n-1,N})^2
  \lambda_{n-1}^b(A_{n-1}^{\ell_n^{[2]}},\ell_{n-1}^{[2]})
  C_{b_n}
  (f\otimes g)(X_n^{\ell_n^{[2]}})
  F(\mathbf{A_{0:n-2}},\mathbf{X_{0:n-1}})
  \right]
  \\=&
  \mathbf{E}\left[
    \frac{1}{N^2}
    C_{b_{n-1}}Q_{n,N}^{\otimes 2}C_{b_n}(f\otimes g)(X_{n-1}^{\ell_{n-1}^{[2]}})
    F(\mathbf{A_{0:n-2}},\mathbf{X_{0:n-1}})
  \right],
  \end{aligned}
  \end{equation*}
  which terminates the verification of \eqref{pf:StepnLambda} and the proof of
  Proposition \ref{lm:LackOfBias}.

\end{proof}
\subsection{Some intuition}
In general, the coupled particle block does not necessarily
have the parents-children relations. Let us see a representation of the duality
formula given in Lemma \ref{lm:duality} 
recursively applied in a mini IPS from level 0 to level 5 to some
randomly chosen indices $\ell_{0:5}^{[2]}$ 
(see Figure
\ref{fig:mbfk1}).

\begin{figure}[htb]
  \centering
  \begin{tikzpicture}[scale = 0.8, every node/.style={transform shape}]
  \path [use as bounding box] (-6,-5) rectangle (6,2.5);

\tikzset{norm/.style={circle, black, draw=black,
fill=white, minimum size=1mm}}
\tikzset{dotted/.style={circle, black, draw=black, densely dotted,
fill=white, minimum size=1mm}}

\tikzset{noire/.style={circle, black, draw=black,
fill=black, minimum size=1mm}}

\node [style=norm] (01) at (-5, 2) {};
\node [style=dotted] (02) at (-5, 1) {};
\node [style=dotted] (03) at (-5, -0) {};
\node [style=noire] (04) at (-5, -1) {};
\node [style=noire] (05) at (-5, -2) {};
\node [] (step0) at (-5, -4.8) {step 0};

\node [style=dotted] (11) at (-3, 2) {};
\node [style=noire] (12) at (-3, 1) {};
\node [style=dotted] (13) at (-3, -0) {};
\node [style=noire] (14) at (-3, -1) {};
\node [style=dotted] (15) at (-3, -2) {};
\node [] (step1) at (-3, -4.8) {step 1};

\draw [arrows=-latex, style=densely dotted] (01) to (11);
\draw [arrows=-latex] (01) to (12);
\draw [arrows=-latex, style=densely dotted] (03) to (13);
\draw [arrows=-latex] (05) to (14);
\draw [arrows=-latex, style=densely dotted] (03) to (15);

\node [style=norm] (21) at (-1, 2) {};
\node [style=noire] (22) at (-1, 1) {};
\node [style=noire] (23) at (-1, -0) {};
\node [style=dotted] (24) at (-1, -1) {};
\node [style=dotted] (25) at (-1, -2) {};
\node [] (step2) at (-1, -4.8) {step 2};

\draw [arrows=-latex, style=densely dotted] (14) to (21);
\draw [arrows=-latex] (14) to (22);
\draw [arrows=-latex] (12) to (23);
\draw [arrows=-latex, style=densely dotted] (14) to (24);
\draw [arrows=-latex, style=densely dotted] (13) to (25);

\node [style=norm] (31) at (1, 2) {};
\node [style=norm] (32) at (1, 1) {};
\node [style=noire] (33) at (1, -0) {};
\node [style=dotted] (34) at (1, -1) {};
\node [style=noire] (35) at (1, -2) {};
\node [] (step3) at (1, -4.8) {step 3};

\draw [arrows=-latex, style=densely dotted] (25) to (31);
\draw [arrows=-latex, style=densely dotted] (22) to (32);
\draw [arrows=-latex] (21) to (33);
\draw [arrows=-latex, style=densely dotted] (25) to (34);
\draw [arrows=-latex] (22) to (35);

\node [style=noire] (41) at (3, 2) {};
\node [style=norm] (42) at (3, 1) {};
\node [style=norm] (43) at (3, -0) {};
\node [style=dotted] (44) at (3, -1) {};
\node [style=noire] (45) at (3, -2) {};
\node [] (step4) at (3, -4.8) {step 4};

\draw [arrows=-latex] (31) to (41);
\draw [arrows=-latex, style=densely dotted] (32) to (42);
\draw [arrows=-latex, style=densely dotted] (34) to (43);
\draw [arrows=-latex, style=densely dotted] (32) to (44);
\draw [arrows=-latex] (32) to (45);

\node [style=noire] (51) at (5, 2) {};
\node [style=noire] (52) at (5, 1) {};
\node [style=dotted] (53) at (5, -0) {};
\node [style=dotted] (54) at (5, -1) {};
\node [style=dotted] (55) at (5, -2) {};
\node [] (step5) at (5, -4.8) {step 5};

\draw [arrows=-latex] (42) to (51);
\draw [arrows=-latex] (43) to (52);
\draw [arrows=-latex, style=densely dotted] (45) to (53);
\draw [arrows=-latex, style=densely dotted] (43) to (54);
\draw [arrows=-latex, style=densely dotted] (43) to (55);

\node [style=noire] (x0) at (-5, -4) {};
\node [style=noire] (y0) at (-5, -3) {};

\node [style=noire] (x1) at (-3, -4) {};
\node [style=noire] (y1) at (-3, -3) {};

\node [style=noire] (x2) at (-1, -4) {};
\node [style=noire] (y2) at (-1, -3) {};

\node [style=noire] (x3) at (1, -4) {};
\node [style=noire] (y3) at (1, -3) {};

\node [style=noire] (x4) at (3, -4) {};
\node [style=noire] (y4) at (3, -3) {};

\node [style=noire] (x5) at (5, -4) {};
\node [style=noire] (y5) at (5, -3) {};

\draw [arrows=-latex, in=180, out = 0] (01) to (x1);
\draw [arrows=-latex, in=180, out = 0] (05) to (y1);

\draw [arrows=-latex, in=180, out = 0] (14) to (x2);
\draw [arrows=-latex, in=180, out = 0] (12) to (y2);

\draw [arrows=-latex, in=180, out = 0] (21) to (x3);
\draw [arrows=-latex, in=180, out = 0] (22) to (y3);

\draw [arrows=-latex, in=180, out = 0] (32) to (x4);
\draw [arrows=-latex, in=180, out = 0] (31) to (y4);

\draw [arrows=-latex, in=180, out = 0] (43) to (x5);
\draw [arrows=-latex, in=180, out = 0] (42) to (y5);

\draw[style=densely dashed] (-6,-2.5) to (6,-2.5);
\node at (-6,-3.5) {coupled};
\node at (-6, 0) {original};

\draw [arrows=latex-latex, in=-135-90, out=-45-90,densely dotted,red!50!black] (05) to (y0);
\draw [arrows=latex-latex, in=-135-90, out=-45-90,densely dotted,red!50!black] (04) to (x0);
\draw [arrows=latex-latex, in=-135-90, out=-45-90,densely dotted,red!50!black] (14) to (y1);
\draw [arrows=latex-latex, in=-135-90, out=-45-90,densely dotted,red!50!black] (12) to (x1);
\draw [arrows=latex-latex, in=-135-90, out=-45-90,densely dotted,red!50!black] (23) to (y2);
\draw [arrows=latex-latex, in=-135-90, out=-45-90,densely dotted,red!50!black] (22) to (x2);

\draw [arrows=latex-latex, in=-135-90, out=-45-90,densely dotted,red!50!black] (35) to (y3);
\draw [arrows=latex-latex, in=-135-90, out=-45-90,densely dotted,red!50!black] (33) to (x3);
\draw [arrows=latex-latex, in=-135-90, out=-45-90,densely dotted,red!50!black] (41) to (y4);
\draw [arrows=latex-latex, in=-135-90, out=-45-90,densely dotted,red!50!black] (45) to (x4);
\draw [arrows=latex-latex, in=-135-90, out=-45-90,densely dotted,red!50!black] (51) to (y5);
\draw [arrows=latex-latex, in=-135-90, out=-45-90,densely dotted,red!50!black] (52) to (x5);

\end{tikzpicture}
\caption{An illustration of the duality formula recursively applied to a 
mini IPS of $n+1 = 6$ levels with 5 particles at each level. 
Every straight black or dotted arrow within the original IPS represents a
Markov transition $M_{p,N}$ and the black twisted ones pointing to the particles in the
coupled particle block represent the Feynman-Kac transition kernels $Q_{p,N}$.
The red dotted bending arrows are identities.
The indices of the original particles in the coupled particle block
are $\ell_0^{[2]} = (4,5)$, $\ell_1^{[2]} = (2,4)$,
$\ell_2^{[2]} = (2,3)$, $\ell_3^{[2]} = (3,5)$, $\ell_4^{[2]} = (1,5)$
and $\ell_5^{[2]} = (1,2)$.}
\label{fig:mbfk1}
\end{figure}
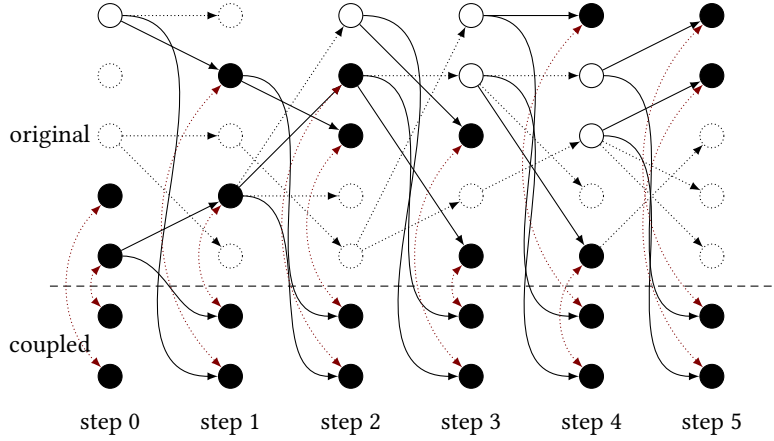
However, we can get any ancestral relations or
coalescent tree-based form by manipulating the genealogical
information encoded in the coupled genealogy. This is the essential
idea we used by introducing many-body Feynman-Kac models.
To make it clearer, we consider an event defined by
\begin{equation}\label{event-coal}
\left\{
  \ell_{p-2}^{[2]} = \tilde{A}_{p-2}^{[2]},
  \ell_{p-1}^1 = \tilde{A}_{p-1}^1 = \tilde{A}_{p-1}^2\neq \ell_{p-1}^2,
  \ell_p^{[2]} = \tilde{A}_p^{[2]}
\right\}.
\end{equation}
On this event, we are able to track the coalescent tree-based
form as in Figure \ref{fig:mbfk2}. 
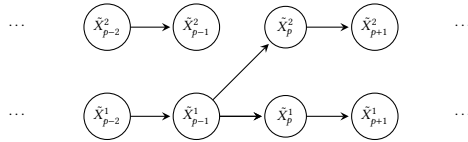
\begin{figure}[H]
  \centering
  \begin{tikzpicture}[->,>=stealth, shorten >=1pt,node
    distance=75pt, auto,transform shape, scale=0.45]
  \tikzset{every state/.style={minimum size=35pt,scale=1}}
  
  \node [state] (1p-2)      {$\tilde{X}_{p-2}^1$};
  \node [state] (1p-1)   [right of = 1p-2]   {$\tilde{X}_{p-1}^1$};
  \node         (dots1)  [left of = 1p-2]   {$\cdots$};
  \node [state] (1p)   [right of = 1p-1]   {$\tilde{X}_p^1$};
  \node [state] (1p+1)   [right of = 1p]   {$\tilde{X}_{p+1}^1$};
  \node         (dots11)  [right of = 1p+1]   {$\cdots$};
  \node [state] (2p)   [above of = 1p]   {$\tilde{X}_p^2$};
  \node [state] (2p+1)   [right of = 2p]   {$\tilde{X}_{p+1}^2$};
  \node         (dots22)  [right of = 2p+1]   {$\cdots$};
  \node [state] (2p-2)   [above of = 1p-2]   {$\tilde{X}_{p-2}^2$};
  \node [state] (2p-1)   [right of = 2p-2]   {$\tilde{X}_{p-1}^2$};
  \node         (dots2)  [left of = 2p-2]   {$\cdots$};
  \path (1p-2) edge[above]  (1p-1);
  \path (2p-2) edge[above]  (2p-1);
  \path (1p-1) edge[above]  (2p);
  \path (1p-1) edge[above]  (1p);
  \path (1p-1) edge[above]  (1p);
  \path (1p) edge[above]  (1p+1);
  \path (2p) edge[above] (2p+1);
  \end{tikzpicture}
  \caption{The coupled particle block tracked by the event defined by
  \eqref{event-coal}.}
  \label{fig:mbfk2}
\end{figure} 
The coupled particle block and its genealogy are defined as the copies of
certain particles and parents indices in the associated original IPS.  On one
hand, we select certain events such that the desired structure is trapped in
the coupled particle block. On the other hand, we define the estimator based on
the information reflected in the original IPS as no additional 
randomness are added by introducing the
coupled particle block. Since their distributions
are connected by the duality formula, we can use the information coded in the
original IPS to estimate the measures corresponding to these coalescent
tree-based particle blocks (see Figure \ref{alkcnalkcn}).  

\begin{figure}[htb]
  \centering
  \begin{tikzpicture}[scale = 0.8, every
      node/.style={transform shape}]
    \path [use as bounding box] (-6,-5) rectangle (6,2.5);

  \tikzset{norm/.style={circle, black, draw=black,densely dotted,
  fill=white, minimum size=1mm}}

  \tikzset{noire/.style={circle, black, draw=black,
  fill=black, minimum size=1mm}}

  \node [style=norm] (01) at (-5, 2) {};
  \node [style=norm] (02) at (-5, 1) {};
  \node [style=noire] (03) at (-5, -0) {};
  \node [style=norm] (04) at (-5, -1) {};
  \node [style=noire] (05) at (-5, -2) {};
  \node [] (step0) at (-5, -4.8) {step 0};

  \node [style=norm] (11) at (-3, 2) {};
  \node [style=norm] (12) at (-3, 1) {};
  \node [style=noire] (13) at (-3, -0) {};
  \node [style=noire] (14) at (-3, -1) {};
  \node [style=norm] (15) at (-3, -2) {};
  \node [] (step1) at (-3, -4.8) {step 1};

  \draw [arrows=-latex,style=densely dotted] (01) to (11);
  \draw [arrows=-latex,style=densely dotted] (01) to (12);
  \draw [arrows=-latex] (03) to (13);
  \draw [arrows=-latex] (05) to (14);
  \draw [arrows=-latex,style=densely dotted] (03) to (15);

  \node [style=norm] (21) at (-1, 2) {};
  \node [style=noire] (22) at (-1, 1) {};
  \node [style=norm] (23) at (-1, -0) {};
  \node [style=norm] (24) at (-1, -1) {};
  \node [style=noire] (25) at (-1, -2) {};
  \node [] (step2) at (-1, -4.8) {step 2};

  \draw [arrows=-latex,style=densely dotted] (14) to (21);
  \draw [arrows=-latex] (14) to (22);
  \draw [arrows=-latex,style=densely dotted] (12) to (23);
  \draw [arrows=-latex,style=densely dotted] (14) to (24);
  \draw [arrows=-latex] (13) to (25);

  \node [style=noire] (31) at (1, 2) {};
  \node [style=noire] (32) at (1, 1) {};
  \node [style=norm] (33) at (1, -0) {};
  \node [style=norm] (34) at (1, -1) {};
  \node [style=norm] (35) at (1, -2) {};
  \node [] (step3) at (1, -4.8) {step 3};

  \draw [arrows=-latex] (25) to (31);
  \draw [arrows=-latex] (22) to (32);
  \draw [arrows=-latex,style=densely dotted] (21) to (33);
  \draw [arrows=-latex,style=densely dotted] (25) to (34);
  \draw [arrows=-latex,style=densely dotted] (22) to (35);

  \node [style=norm] (41) at (3, 2) {};
  \node [style=noire] (42) at (3, 1) {};
  \node [style=norm] (43) at (3, -0) {};
  \node [style=norm] (44) at (3, -1) {};
  \node [style=noire] (45) at (3, -2) {};
  \node [] (step4) at (3, -4.8) {step 4};

  \draw [arrows=-latex,style=densely dotted] (31) to (41);
  \draw [arrows=-latex] (32) to (42);
  \draw [arrows=-latex,style=densely dotted] (34) to (43);
  \draw [arrows=-latex,style=densely dotted] (32) to (44);
  \draw [arrows=-latex] (32) to (45);

  \node [style=noire] (51) at (5, 2) {};
  \node [style=norm] (52) at (5, 1) {};
  \node [style=noire] (53) at (5, -0) {};
  \node [style=norm] (54) at (5, -1) {};
  \node [style=norm] (55) at (5, -2) {};
  \node [] (step5) at (5, -4.8) {step 5};

  \draw [arrows=-latex] (42) to (51);
  \draw [arrows=-latex,style=densely dotted] (43) to (52);
  \draw [arrows=-latex] (45) to (53);
  \draw [arrows=-latex,style=densely dotted] (43) to (54);
  \draw [arrows=-latex,style=densely dotted] (43) to (55);

  \node [style=noire] (x0) at (-5, -4) {};
  \node [style=noire] (y0) at (-5, -3) {};

  \node [style=noire] (x1) at (-3, -4) {};
  \node [style=noire] (y1) at (-3, -3) {};

  \node [style=noire] (x2) at (-1, -4) {};
  \node [style=noire] (y2) at (-1, -3) {};

  \node [style=noire] (x3) at (1, -4) {};
  \node [style=noire] (y3) at (1, -3) {};

  \node [style=noire] (x4) at (3, -4) {};
  \node [style=noire] (y4) at (3, -3) {};
  
  \node [style=noire] (x5) at (5, -4) {};
  \node [style=noire] (y5) at (5, -3) {};

  \draw [arrows=-latex] (x0) to (x1);
  \draw [arrows=-latex] (y0) to (y1);

  \draw [arrows=-latex] (x1) to (x2);
  \draw [arrows=-latex] (y1) to (y2);

  \draw [arrows=-latex] (x2) to (x3);
  \draw [arrows=-latex] (y2) to (y3);

  \draw [arrows=-latex] (x3) to (x4);
  \draw [arrows=-latex] (x3) to (y4);

  \draw [arrows=-latex] (x4) to (x5);
  \draw [arrows=-latex] (y4) to (y5);

  \draw[style=densely dashed] (-6,-2.5) to (6,-2.5);
  \node at (-6,-3.5) {coupled};
  \node at (-6, 0) {original};
  
  \draw [arrows=latex-latex, in=-135-90, out=-45-90,densely dotted,red!50!black] (03) to (y0);
  \draw [arrows=latex-latex, in=-135-90, out=-45-90,densely dotted,red!50!black] (05) to (x0);
  \draw [arrows=latex-latex, in=-135-90, out=-45-90,densely dotted,red!50!black] (13) to (y1);
  \draw [arrows=latex-latex, in=-135-90, out=-45-90,densely dotted,red!50!black] (14) to (x1);
  \draw [arrows=latex-latex, in=-135-90, out=-45-90,densely dotted,red!50!black] (22) to (y2);
  \draw [arrows=latex-latex, in=-135-90, out=-45-90,densely dotted,red!50!black] (25) to (x2);

  \draw [arrows=latex-latex, in=-135-90, out=-45-90,densely dotted,red!50!black] (31) to (y3);
  \draw [arrows=latex-latex, in=-135-90, out=-45-90,densely dotted,red!50!black] (32) to (x3);
  \draw [arrows=latex-latex, in=-135-90, out=-45-90,densely dotted,red!50!black] (42) to (y4);
  \draw [arrows=latex-latex, in=-135-90, out=-45-90,densely dotted,red!50!black] (45) to (x4);
  \draw [arrows=latex-latex, in=-135-90, out=-45-90,densely dotted,red!50!black] (51) to (y5);
  \draw [arrows=latex-latex, in=-135-90, out=-45-90,densely dotted,red!50!black] (53) to (x5);
  \end{tikzpicture}
\caption{
An illustration of the duality formula recursively applied to a 
mini IPS of $n+1 = 6$ levels with 5 particles at each level. 
Every straight black or dotted arrow within the original IPS represents a
Markov transition $M_{p,N}$ and the black ones within the
coupled particle block represent the Feynman-Kac transition kernels $Q_{p,N}$.
The red dotted bending arrows are identities.
The indices of the original particles in the coupled particle block
are $\ell_0^{[2]} = (3,5)$, $\ell_1^{[2]} = (3,4)$,
$\ell_2^{[2]} = (2,5)$, $\ell_3^{[2]} = (1,2)$, $\ell_4^{[2]} = (2,5)$
and $\ell_5^{[2]} = (1,3)$.}
\label{alkcnalkcn}
\end{figure}
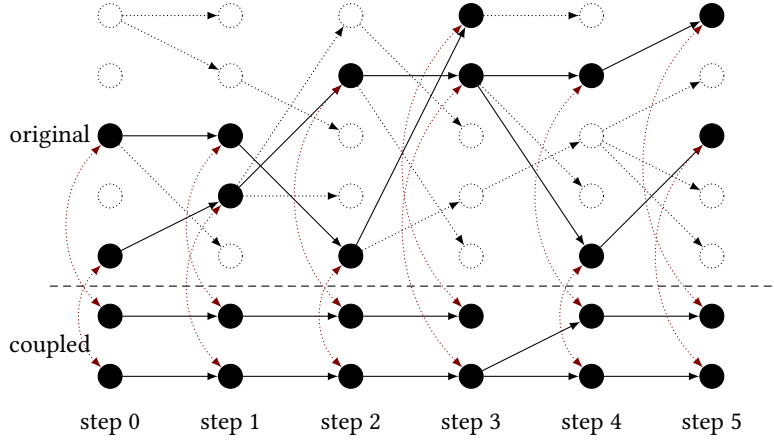
The duality formula provides a way to touch the
adaptive versions of the coalescent tree-based measures $\Gamma_n^b$, i.e., all
the Feynman-Kac transition kernels $Q_{p}$ in the definition are replaced with
the adaptive version $Q_{p,N}$. 
This is the idea underlying the construction of the estimators $\Gamma_{n,N}^b$.

\subsection{Connection with SMC}\label{final-remarks}
To conclude, let us say a few words about the behavior of
$\Gamma_{n,N}^b$. One remark is that,
in general, this estimator is not unbiased in the ASMC framework. This is
a consequence of the adaptive parametrization, as witnessed by Lemma \ref{lm:cvg-Gamma-Q}. On the opposite, in a 
nonadaptive case (SMC), the estimation is unbiased, 
exactly as $\gamma_n^N$ is an unbiased
estimation of $\gamma_n$ (see for example \cite{del2004feynman-kac} Section 3.5.1).
It turns out that
the classical SMC framework corresponds to the case where the function $h_{n}$ in
$\mathcal{A}\ref{A2}$ is equal to zero, meaning that $Q_{n,N}=Q_n$ for all $n$. Thus, Lemma
\ref{lm:cvg-proba-Gamma} and 
\eqref{eq:decomp-Gamma-N} give the following proposition. 
\begin{proposition}\label{prop:unbiased-SMC}
  Assume $\mathcal{A}\ref{A1}$-$\mathcal{A}\ref{A2}$ and suppose that $h_{n}\equiv 0$ for all $n\geq 0$. Then, for all test
  functions $f,g\in\mathcal{B}_b(E_n)$, 
  $$
  \mathbf{E}\left[\Gamma_{n,N}^b(f\otimes g)\right]
  =
  \Gamma_{n}^b(f\otimes g).
  $$
  In particular, we also have
  $$
  \mathbf{E}\left[\gamma_n^N(1)^2V_n^N(f)\right]
  =
  \mathrm{Var}
  \left[\gamma_n^N(f)\right].
  $$
\end{proposition}
In fact, the essential technical results in Section
\ref{sec:TR} and Section \ref{zicj} only require 
$\mathcal{A}\ref{A1}$. 
In other words, $\mathcal{A}\ref{A2}$
can be studied separately in order to adapt to applications not covered in
this article.

Another remark is about the difference between $\Gamma_{n,N}^b$ and $\mu_b$
as defined in Section 3.2. of \cite{lee2018var} in the nonadaptive context. 
However, since it is not straightforward to
compare these estimators that are extremely notation-heavy, we would just like to
briefly and heuristically 
mention that the main difference comes from the step where there is
a coalescence, namely $b_p=1$. If we consider Figure \ref{kjzcbxkzjxb} in Section \ref{ajsncoac},
our estimator is not the
most ``precise'' that one could propose. Let us look at
the case where
$$
\ell_{0:6}^{[2]} = \left(
(5,3),(4,3),(2,5),(2,4),(2,5),(1,3),(2,4)\right).
$$ 
For the terminal point $X_3^4$, 
the conditional distribution of $A_2^4$ is simply the categorical distribution since $X_3^4$ is a 
terminal point.
Roughly speaking, once all the genealogy of the terminal points is
calculated, one can deduce $\mu_b$. Hence, the take-home message is simple:
if one is interested in estimating
$\Gamma_n^b$ numerically, then the estimator $\mu_b$ proposed in \cite{lee2018var}
is expected to be more accurate, meaning that the variance should be smaller in general.

Nevertheless, as a theoretical tool, our estimator is easier to deal with in the adaptive
framework. 
Indeed, induction is highly involved in our proof
of consistency, so estimators that are stepwise easy to manipulate are required.  
Another difference is that we do not use instrumental random variables such
as $K^1$ and $K^2$ in the definition of $\mu_b$. This also simplifies the
analysis in an adaptive context where there is already more randomness than in a 
nonadaptive context. 

\section{Numerical experiment}
\label{sec:num}
We provide in this section a numerical experiment based on the same toy example
as the one presented in
Section 4.1 of \cite{beskos2016ASMC}. In particular, this ensures that assumptions $\mathcal{A}$\ref{A1}-$\mathcal{A}$\ref{A2} are satisfied.
Namely, consider a sequence of centered Gaussian
target distributions $(\eta_n;0\leq n\leq 50)$ on $\mathbf{R}^{10}$ given by
$$
\eta_n(x)\propto \exp\left(-\frac{1}{2}\left\langle x, \Sigma_n^{-1}
x\right\rangle\right).
$$
Denote by $\mathrm{Id}$ the identity matrix on $\mathbf{R}^{10}$ and 
$\mathrm{J}$ the lower triangular matrix such that $\mathrm{J}_{ij} = 1$ for
$1\leq j \leq i-1\leq 9$. 
The covariance matrices are defined by
$$
\Sigma_n = \mathrm{L}_n \mathrm{L}_n^{\mathrm{T}}, \text{ with } \mathrm{L}_n = \left(10\left(1- \frac{n}{99}\right) +
\frac{1}{10}\frac{n}{99} \right) \mathrm{Id} +
\frac{1}{2}\frac{n}{99}\mathrm{J}.
$$
Thus, the initial distribution $\eta_0$ consists in 10 centered and independent Gaussian
components with variance 10. As $n$ grows, the covariance 
structure becomes more complicated. 
We consider an implementation of SMC with (nonadaptive) potential functions
$$G_n(x) :=
\exp\left(-\frac{1}{2}\left\langle x, \left(\Sigma_{n+1}^{-1} - \Sigma_n^{-1} \right)
x\right\rangle\right),$$ 
and some random walk Metropolis
kernels $M_{n}$ such that, at each step,
$M_{n}$ is reversible with respect to $\eta_n$. 
In this scenario, a popular choice for $M_n$ is based on the Gaussian proposal with covariance matrix $\Sigma_n$. This is the ``limiting'' (nonadaptive) scenario that we will consider in the sequel. When one does not know the covariance matrices $\Sigma_n$, a natural choice is to use the estimated
covariance matrix $\Sigma_n^N$. Our goal is to compare the respective behaviors of
adaptive SMC and nonadaptive SMC. In particular,
we want to show that the Lee and Whiteley variance estimator, in an adaptive context,  goes to the
asymptotic variance of the ``limiting'' (nonadaptive) SMC when $N$ grows. 

For this, we consider the test function $f :\mathbf{R}^{10}\ni(x^{(1)},x^{(2)},\dots,x^{(10)})\mapsto
x^{(1)}\in \mathbf{R}$.
Keeping the notation of the previous sections, we illustrate the asymptotic
variance estimators $NV_n^N\left(f-\eta_n^N(f)\right)$, 
which estimate the asymptotic variances of $\eta_n^N(f)$
respectively for the adaptive and
nonadaptive SMC algorithms, see Figure \ref{fig:num}. On the latter, the so-called reference value is the estimation of the theoretical asymptotic variance $\sigma^2_{\eta_n}(f-\eta_n(f))$. This value is estimated via Crude Monte Carlo through $2\times 10^3$
independent
runs of nonadaptive SMC with $N = 5\times 10^3$ (notice that, stricto sensu, it does of course not depend on N). 
At each iteration of the algorithm, the
random walk Metropolis kernel is applied 4 times in order to ensure a
certain level of acceptance. 

\begin{figure}[htb]
\centering
\begin{tabular}{c}
\subf{\includegraphics[width=16cm]{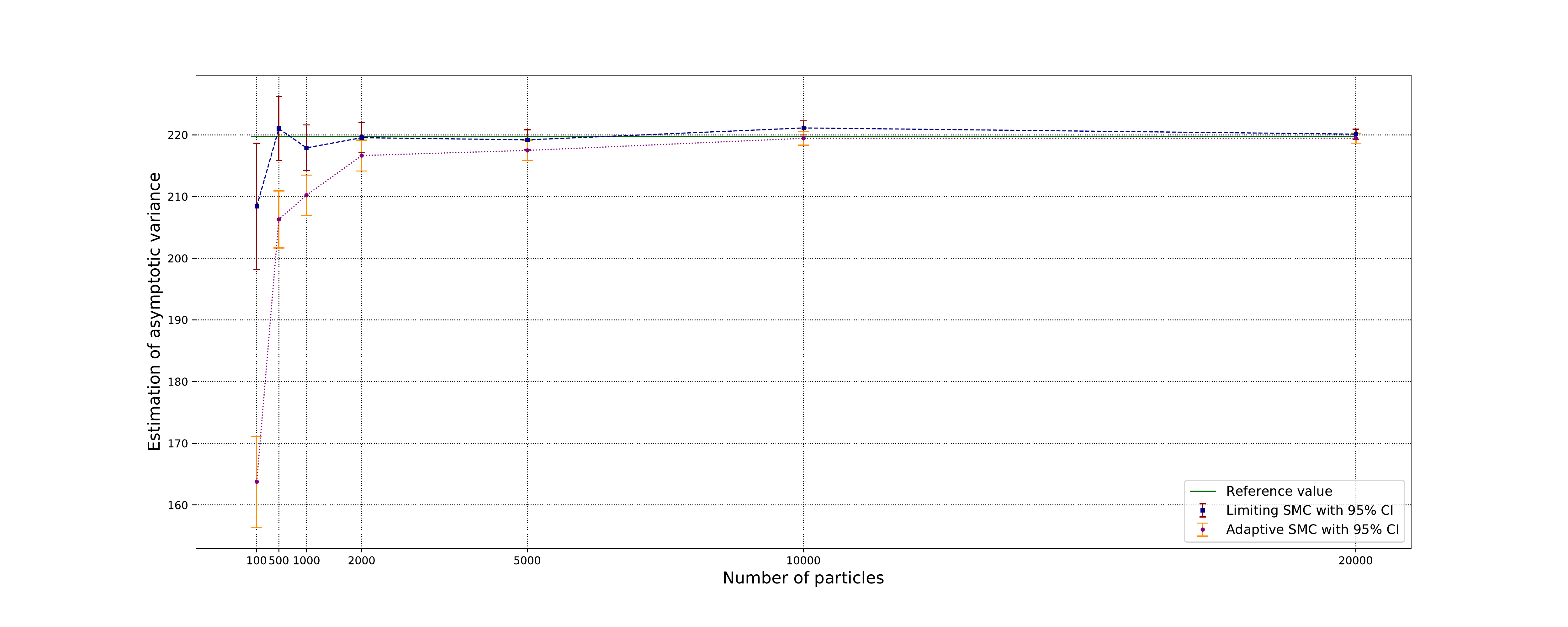}}
{(a) Variance estimation of $\eta_n^N(f)$ with $n=10$.}
\\
\subf{\includegraphics[width=16cm]{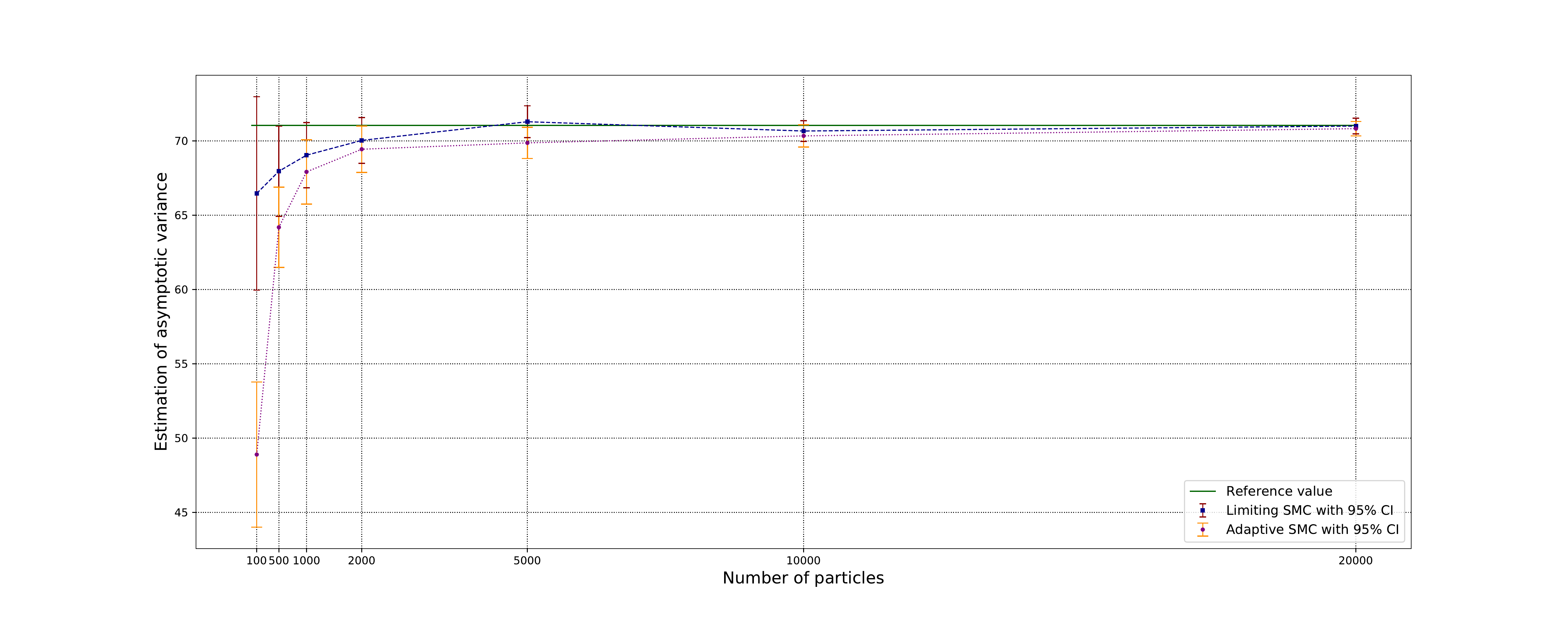}}
     {(b) Variance estimation of $\eta_n^N(f)$ with $n=50$.}
\end{tabular}
\caption{Variance estimators for adaptive and
nonadaptive ``limiting'' SMC for $n=10$ and $n=50$ with $N$
varying from $100$ to $2\times10^4$. We trace
the variance estimators and its $95\%$ confidential intervals based on $500$
independent runs of each algorithm. The reference values $\sigma^2_{\eta_n}(f-\eta_n(f))$ are obtained through $2\times 10^3$
independent
runs of nonadaptive ``limiting'' SMC with $N = 5\times 10^3$.}
\label{fig:num}
\end{figure}

It is clear that when
$N$ is relatively small, the variance estimations are biased. However, as the
number $N$ of particles
grows, we see that the behaviors of the adaptive and nonadaptive algorithms
are similar in terms of asymptotic variance estimations. More precisely, the overlaps
of the $95\%$ confidential intervals indicate that the adaptive SMC algorithm
is indeed very ``close'' to its nonadaptive ``limiting'' counterpart. As expected, both
variance estimators 
converge to the reference value $\sigma^2_{\eta_n}(f-\eta_n(f))$ as $N$ grows.

\section{Truncated variance estimators}
\label{jscnnlncln}
As mentioned in \cite{lee2018var}, their variance estimators degenerate when
$n$ is very large compared to the number $N$ of particles.
Typically, no disjoint ancestral lines exist in such
a particle system. In this case, we recommend to use the same kind of
fixed-lag variance estimators as the ones proposed in \cite{olsson2019}. More precisely,
we only use part of the genealogy of the particle system (e.g., truncated at
time $n-H$ for a relatively small lag $H\in \mathbf{N}^*$) to
construct the variance estimators. Hereafter, we provide a heuristic in order to
justify the relevance of such estimators in practice. 

In general, the application of SMC sampling on a large time scale requires some ``forgetting'' (mixing) properties
of the underlying model. For example, in the toy example presented in Appendix
\ref{sec:num}, if the Metropolis kernel is implemented a large amount of times
at each iteration, the particles will somehow ``forget'' the dependence caused
by the genealogy. 
Basically, in terms of coalescent tree-based measures, $\bar{\Gamma}_n^{(p)}$
would be very ``close'' to the disjoint ancestral lines based measure
$\bar{\Gamma}_n^{(\varnothing)}$.
More concretely, if we look at the asymptotic variance
$\sigma_{\eta_n}^2$, we have
$$
\sigma_{\eta_n}^2 (f)
= 
\underbrace{
  \sum_{p=0}^{n-H-1} \left(\bar{\Gamma}_{n}^{(p)}(f^{\otimes 2})
  -\bar{\Gamma}_{n}^{(\varnothing)}(f^{\otimes 2})\right)
}_{\text{small by the ``forgetting'' properties of the model}}
+
\sum_{p=n-H}^{n} \left(\bar{\Gamma}_{n}^{(p)}(f^{\otimes 2})
-\bar{\Gamma}_{n}^{(\varnothing)}(f^{\otimes 2})\right).
$$
Accordingly, a natural idea is to estimate only the second part of the right hand side
in order to approximate the asymptotic variance. A truncated term by term estimator
can therefore be defined as
$$
\sum_{p=n-H}^{n} \left(\bar{\Gamma}_{n,N}^{(p)}(f^{\otimes 2})
-\bar{\Gamma}_{n,N}^{(\varnothing)}(f^{\otimes 2})\right).
$$
Unfortunately, this estimator is not always numerically stable, as it requires
that disjoint ancestral lines exist in the particle system from time $0$ to time $n-H$. Following 
the same mechanism as in the proofs of Proposition \ref{lm:convergence-proba} and 
Proposition \ref{prop:decomposition-fixed}, we can show that
$$
N \left(\eta_n^N(f)^2 - \bar{\Gamma}_{n,N}^{(\varnothing,H)}(f^{\otimes 2})\right)
  \approx
  \sum_{p=n-H}^{n} \left(\bar{\Gamma}_{n,N}^{(p)}(f^{\otimes 2})
  -\bar{\Gamma}_{n,N}^{(\varnothing)}(f^{\otimes 2})\right),
$$
where, if $E_n^i(H)$ denotes the index of the ancestor of $X_n^i$ at step $n-H$,
$$
\bar{\Gamma}_{n,N}^{(\varnothing,H)}(f^{\otimes 2})
:= 
\frac{N^{H-1}}{(N-1)^{H+1}}\sum_{E_{n}^i(H)\neq E_{n}^j(H)} f(X_n^i)
f(X_n^j).
$$ 
The estimator $N(\eta_n^N(f)^2 - \bar{\Gamma}_{n,N}^{(\varnothing,H)}(f^{\otimes 2}))$
is more or less the one proposed in \cite{olsson2019} and is indeed a truncated version of $NV_n^N(f)$ proposed in \cite{lee2018var}. We refer the interested reader to \cite{olsson2019} for theoretical results as well as numerical illustrations.
When $H$ is properly chosen, the fixed-lag variance estimator is expected to be
able to balance the memory and the degeneracy of the genealogy of the particle
system.
However, finding a suitable $H$ in a specific application is highly nontrivial. As
explained in \cite{olsson2019}, it is then natural to consider adaptive mechanisms to
determine $H$. Nevertheless, to the best of our knowledge, this is still an open problem, which is beyond the scope of the present paper.

\end{appendices}

\section*{Acknowledgements}
This work was partially supported by the French Agence Nationale de la Recherche, under grant ANR-14-CE23-0012, and by the European Research Council under the European Union's Seventh Framework Programme (FP/2007-2013) / ERC Grant Agreement number 614492.


\end{document}